\documentclass[reqno]{amsart}
\usepackage{amssymb}
\usepackage{amsmath}
\usepackage{bm}
\usepackage{amsthm}
\usepackage[usenames]{color}
\usepackage{graphicx}
\usepackage{cite}
\usepackage{bbm}

\usepackage[colorlinks,linkcolor=blue,anchorcolor=red,citecolor=blue]{hyperref}
\usepackage[margin=1in]{geometry} 
\usepackage{marginnote}

\usepackage{color}

\newtheorem{lemma}{Lemma}[section]
\newtheorem{theorem}{Theorem}[section]

\newtheorem{proposition}{Proposition}[section]

\numberwithin{equation}{section}

\newcommand{\dis}{\displaystyle}

\newcommand{\rmi}{{\mathrm i}}
\newcommand{\rmre}{{\rm Re}}

\newcommand{\C}{\mathbb{C}}

\newcommand{\R}{\mathbb{R}}

\newcommand{\Z}{\mathbb{Z}}

\renewcommand{\S}{\mathbb{S}}
\newcommand{\T}{\mathbb{T}}

\newcommand{\CE}{\mathcal{E}}

\newcommand{\CZ}{\mathcal{Z}}

\newcommand{\para}{\shortparallel}

\newcommand{\ep}{\epsilon}

\newcommand{\na}{\nabla}

\newcommand{\al}{\alpha}
\newcommand{\be}{\beta}
\newcommand{\ga}{\gamma}
\newcommand{\om}{\omega}
\newcommand{\la}{\lambda}
\newcommand{\de}{\delta}
\newcommand{\si}{\sigma}
\newcommand{\pa}{\partial}
\newcommand{\ka}{\kappa}
\newcommand{\eps}{\epsilon}

\newcommand{\De}{\Delta}
\newcommand{\Ga}{\Gamma}

\newcommand{\vertiii}[1]{{\left\vert\kern-0.25ex\left\vert\kern-0.25ex\left\vert #1
		\right\vert\kern-0.25ex\right\vert\kern-0.25ex\right\vert}}

\makeatletter
\@namedef{subjclassname@2020}{%
	\textup{2020} Mathematics Subject Classification}
\makeatother

\begin{document}
\title[Mixture of Monatomic and Polyatomic Gases]{Well-posedness and time-asymptotic of Boltzmann equations for monatomic and polyatomic mixtures}

\author[R. Alonso]{Ricardo Alonso}
\address{College of Science and Engineering, HBKU, Education City, Doha, Qatar}
\email{ralonso@hbku.edu.qa}

\author[Z.-G. Li]{Zongguang Li}
\address{Department of Applied Mathematics, The Hong Kong Polytechnic University, Hung Hom, Hong Kong, P.R.~China}
\email{zongguang.li@polyu.edu.hk}

\begin{abstract}
This paper considers a system of Boltzmann equations modelling the mixture of monatomic and polyatomic gases in an $L^{2}-L^{\infty}$ perturbation theory around global modified Maxwellians accounting for the internal energy of the mixture in the whole space and the torus. We investigate the pointwise decay in velocity and internal energy of the linearized Boltzmann operators in the four types of collisions.  A novel approach is developed to deal with the additional internal energy variable $I\in \R_+$ and the loss of symmetry due to dissimilar masses of the mixture components.  Subsequently, we carry out a classical $L^2-L^\infty$ method to establish the well-posedness theory of the system.  The optimal polynomial time decay rate on the whole space is obtained accordingly based on the spatial Fourier's study of the linearized system.   The analysis shows the structure of a perturbed Euler-type model for the solution's macroscopic quantities: density, bulk velocity, and temperature, near the steady state, which gives a potential application to investigate fluid limit problems.   In addition, this work proves exponential time decay in the torus and fills the gap of classical multi-species Boltzmann in the whole space.
\end{abstract}




\maketitle
\thispagestyle{empty}

\tableofcontents

\section{Introduction}
The objective of this paper is to establish rigorous estimates to understand inhomogeneous Boltzmann conservative systems of mixed monoatomic and polyatomic gases in the context of bounded solutions close to thermal equilibrium.   These systems are improved models of the standard Boltzmann systems of ideal gases designed to account for non-ideal effects by introducing an additional continuous kinetic variable associated to the internal energy of the system for the polyatomic gases, see \cite{BBG,BBBD}.

The model is based on the distinction of two type of agents: monoatomic and polyatomic.  The main core of the document's technical analysis consists in estimating the contributions of the four induced type of interactions, namely, pure monoatomic, pure polyatomic, and the cross interactions of monoatomic-polyatomic and polyatomic-monoatomic agents, see \cite{ACG}.  Once a suitable estimation of each collisional form contribution is performed, a modified approach of $L^{2}-L^{\infty}$ decay, which is designed to deal with a polyatomic gaseous mixture, is implemented at the linear and, then, at the non-linear perturbed equation.  The perturbations are considered in $L^2\cap L^{\infty}$-spaces with Gaussian tails.   We mention that the final results are comparable in generality to the ones related to the classical theory of Boltzmann systems for ideal gases, see \cite{Guo,UY}.

Given the accuracy of these modified systems in matching macroscopical properties of the gases when using moment closure methods and expansions, as documented for example in \cite{DOPT}, the theory of polyatomic gases is a very active field of research.  Thus, we present an analysis of monoatomic and polyatomic systems in the context of continuous internal energy approach that uses the Borgnakke-Larsen procedure for the collision parametrization.  Such procedure not only allows to capture important features of non-ideal gases, but also make the model accessible to analytical and numerical tools. 

Let us mention that the analysis is performed for hard potential models with bounded collisional kernels.   We can differentiate rate of hard potentials and collisional cross sections between gases which satisfy an lower-upper control with bounded partition functions of the Borgnakke-Larsen type.  These assumptions cover the current models used in numerical applications. This paper provides the first well-posedness result on the spatially inhomogeneous Boltzmann equation for monatomic-polyatomic gaseous mixture. It also includes the solutions in the whole space for either single polyatomic gas or mixture of classical monatomic gases as its special cases. We expect that the $L^\infty$ estimates on the linearized operator and the $L^2\cap L^\infty$ method established in this document will have further applications in kinetic equations describing gas mixtures or a single polyatomic gas, such as gas mixtures in bounded domains, hydrodynamic limits, wave patterns, half space problems, etc.

\subsection{The model in question}
We consider the Boltzmann equation for a mixture of monatomic and polyatomic gases:
 \begin{eqnarray}\label{BE}
	&\dis \pa_tF+v\cdot \na_x F=Q(F,F),\quad F(t=0)=F_0,
\end{eqnarray}
where $F=(F_1,F_2,\cdots, F_n)^T$ is the distribution function for a integer number of agents $n\geq 1$, placed in $x\in \T^3$ or $\R^3$, and with velocity $v\in \R^3$. Correspondingly, the initial data is given by $F_0=(F_{10},\cdots, F_{n0})^T$. For integer $1\leq p\leq n$, the first $p$ components of $F$ are associated to monatomic gases with densities $(F_1,F_2,\cdots, F_p)=(F_1,F_2,\cdots, F_p)(t,x,v)$, and the rest are for polyatomic gases modeled using the additional continuous internal energy variable $I\in [0,\infty)$ so that their densities are given by $(F_{p+1},\cdots, F_n)=(F_{p+1},\cdots, F_n)(t,x,v,I)$. The vector valued collision operator $Q$ has the form $Q=(Q_1,Q_2,\cdots,Q_n)$ with
\begin{align}\label{Qi}
	Q_i(F,F)=\sum^n_{j=1}Q_{ij}(F,F).
\end{align}
For brevity we denote $(F_{i*},F^\prime_i,F^\prime_{i*})=(F_i(t,x,v_*),F_i(t,x,v^\prime),F_i(t,x,v^\prime_*))$ for monoatomic gases $1\leq i\leq p$, and  $(F_{i*},F^\prime_i,F^\prime_{i*})=(F_i(t,x,v_*,I_*),F_i(t,x,v^\prime,I^\prime),F_i(t,x,v^\prime_*,I^\prime_*))$ for polyatomic gases $p+1\leq i\leq n$.
Let $\de_i\in [2,\infty)$ be the number of internal degrees of freedom with $\de_i=2$ for $1\leq i\leq p$. Depending on the range of $i$ and $j$, the collision operator $Q_{ij}$ is defined by 
\begin{enumerate}
	\item $1\leq i,\, j\leq p$ (Mono-Mono interactions): \begin{align*}
		Q_{ij}(F,G)(t,x,v)=\int_{(\R^3)^3} W_{ij}(v,v_*|v^\prime,v^\prime_*)\left(F_i^\prime G_{j*}^\prime-F_iG_{j*}\right)dv_*dv^\prime dv^\prime_*.
	\end{align*}
	\item $1\leq i\leq p$, $p+1\leq j\leq n$ (Mono-Poly interactions): \begin{align*}
		Q_{ij}(F,G)(t,x,v)=\int_{(\R^3)^3\times (\R_+)^2} W_{ij}(v,v_*,I_*|v^\prime,v^\prime_*,I^\prime_*)\left(\frac{F_i^\prime G^\prime_{j*}}{( I^\prime_*)^{\de_j/2-1}}-\frac{F_iG_{j*}}{( I_*)^{\de_j/2-1}}\right)dv_*dv^\prime dv^\prime_* dI_* dI^\prime_*.
	\end{align*}
	
		\item $p+1\leq i\leq n$, $1\leq j\leq p$ (Poly-Mono interactions): \begin{align*}
		Q_{ij}(F,G)(t,x,v,I)=\int_{(\R^3)^3\times \R_+} W_{ij}(v,I,v_*|v^\prime,I^\prime,v^\prime_*)\left(\frac{F_i^\prime G^\prime_{j*}}{(I^\prime)^{\de_i/2-1}}-\frac{F_iG_{j*}}{(I)^{\de_i/2-1}}\right)dv_*dv^\prime dv^\prime_* dI^\prime .
	\end{align*}
	
	\item $p+1\leq i,\, j\leq n$ (Poly-Poly interactions): \begin{align*}
		Q_{ij}(F,G)(t,x,v,I)=\int_{(\R^3)^3\times (\R_+)^3} &W_{ij}(v,I,v_*,I_*|v^\prime,I^\prime,v^\prime_*,I^\prime_*)\notag\\
		\times&\left(\frac{F_i^\prime G^\prime_{j*}}{(I^\prime)^{\de_i/2-1}( I^\prime_*)^{\de_j/2-1}}-\frac{F_iG_{j*}}{(I)^{\de_i/2-1}(I_*)^{\de_j/2-1}}\right)dv_*dv^\prime dv^\prime_* dI_*dI^\prime dI^\prime_*.
	\end{align*}
\end{enumerate}
Denote $\bm{\delta}_3$ and $\bm{\delta}_1$ to be the Dirac's delta function in $\R^3$ and $\R$ respectively. Moreover, adopting the notation in \cite{Bernhoff}, we let	
\begin{align}\label{DefZ}
	Z=Z_i:=\left\{
	\begin{array}{rl}
		&v,\;\;  i\leq p,\\
		&(v,I),\;\; i> p,
	\end{array}\right. \in \CZ_i:=\left\{
	\begin{array}{rl}
	&\R^3,\;\; i\leq p,\\
	&\R^3\times\R_+,\;\; i> p,
	\end{array}\right.
\end{align}
and with $Z_*,Z',Z'_*$ defined analogously.
The collisional kernels $W_{ij}$ satisfy 
\begin{align}\label{DefW}
	W_{ij}(Z,&Z_*|Z',Z'_*)=(m_i+m_j)^2m_im_j(I)^{\de_i/2-1}(I_*)^{\de_j/2-1}\si_{ij}{(|v-v_*|,\cos\theta,I,I_*,I^\prime,I^\prime_*)}\frac{|v-v_*|}{|v^\prime-v^\prime_*|}\notag\\
	&\times\bm{\de}_3(v+v_*-v^\prime-v^\prime_*)\bm{\de}_1\Big(\frac{|v|^2}{2}+\frac{|v_*|^2}{2}-\frac{|v^\prime|^2}{2}-\frac{|v^\prime_*|^2}{2}+(I-I^\prime)\chi_{\{i>p\}}+(I_*-I^\prime_*)\chi_{\{j>p\}}\Big),
\end{align}
with the properties
\begin{align}
	&W_{ij}(Z,Z_*|Z^\prime,Z^\prime_*)=W_{ij}(Z_*,Z|Z^\prime_*,Z^\prime),\label{ProW1}\\
	&W_{ij}(Z,Z_*|Z^\prime,Z^\prime_*)=W_{ij}(Z^\prime,Z^\prime_*|Z,Z_*),\label{ProW2}\\
	&W_{ii}(Z,Z_*|Z^\prime,Z^\prime_*)=W_{ii}(Z,Z_*|Z^\prime_*,Z^\prime),\label{ProW3}
\end{align}
where $m_i>0$ is the mass for molecules of the $i-th$ species, $\cos\theta=\big|\frac{v-v_*}{|v-v_*|}\cdot\frac{v^\prime-v^\prime_*}{|v^\prime-v^\prime_*|}\big|$, and $\chi$ is the characteristic function. Note that $\si_{ij}$ is independent of $I,I'$ for $i\leq p$ and independent of $I_*,I'_*$ for $j\leq p$.   Throughout the paper, we assume that $\si_{ij}$ satisfies
\begin{align}\label{si}
	\frac{1}{C}\sqrt{|u|^2-\frac{2\De I}{\mu_{ij}}}\frac{(I^\prime )^{\de_i/2-1}(I^\prime_*)^{\de_j/2-1}}{|u|E_{ij}^{(\eta+\de_i\chi_{\{i>p\}}+\de_j\chi_{\{j>p\}})/2}}\leq \si_{ij}\leq C\sqrt{|u|^2-\frac{2\De I}{\mu_{ij}}}\frac{(I^\prime )^{\de_i/2-1}(I^\prime_*)^{\de_j/2-1}}{|u|E_{ij}^{(\eta+\de_i\chi_{\{i>p\}}+\de_j\chi_{\{j>p\}})/2}},
\end{align} 
for $0\leq\eta<1$ and some constant $C>0$.  We have introduced $\mu_{ij}=\frac{m_im_j}{m_i+m_j}$, and relative velocity, relative internal energy, and total energy
\begin{equation*}
u=v - v_*, \quad \De I=(I'-I)\chi_{\{i>p\}}+(I'_*-I_*)\chi_{\{j>p\}}, \quad \text{and} \quad E_{ij}=\frac{\mu_{ij}}{2}|u|^2+I\chi_{\{i>p\}}+I_*\chi_{\{j>p\}}.
\end{equation*}
We further define the post-collisional counterparts $u'=v'-v'_*$ and $E'_{ij}:=\frac{\mu_{ij}}{2}|u'|^2+I'\chi_{\{i>p\}}+I^\prime_*\chi_{\{j>p\}}$, and the energy ratios $$R=\frac{\mu_{ij}}{2}\frac{|u'|}{E_{ij}},\qquad r=\frac{I^\prime}{(1-R)E_{ij}}.$$  Thus, under these notations, the collision operator $Q_{ij}$ can be rewritten as:
\begin{align*}
	Q_{ij}(F,G)&=\left\{
	\begin{array}{rl}
		&\dis\int_{\R^3\times \S^2} B_{ij}\left(F_i^\prime G_{j*}^\prime-F_iG_{j*}\right)dv_*d\si,\qquad i \ ,\ j\leq p, 
		\\ [1.5em]
		&\dis\int_{\R^3\times \R_+\times [0,1]\times \S^2} B_{ij}\left(\frac{F_i^\prime G^\prime_{j*}}{( I^\prime_*)^{\de_j/2-1}}-\frac{F_iG_{j*}}{( I_*)^{\de_j/2-1}}\right)\\[1.5em]
		&\hspace{3cm}\times(1-R)^{\de_j/2-1}R^{1/2}(I_*)^{\de_j/2-1}dv_*dI_* dR d\si,\qquad i\leq p\ , \ j>p,\\[1.5em]
		&\dis\int_{\R^3\times [0,1]\times \S^2} B_{ij}\left(\frac{F_i^\prime G^\prime_{j*}}{(I^\prime)^{\de_i/2-1}}-\frac{F_iG_{j*}}{(I)^{\de_i/2-1}}\right)(1-R)^{\de_i/2-1}R^{1/2}(I)^{\de_i/2-1}dv_*dR d\si,\qquad i>p\ ,\ j\leq p,\\[1.5em]
		&\dis\int_{\R^3\times \R_+\times [0,1]^2\times \S^2} B_{ij}\left(\frac{F_i^\prime G^\prime_{j*}}{(I^\prime)^{\de_i/2-1}( I^\prime_*)^{\de_j/2-1}}-\frac{F_iG_{j*}}{(I)^{\de_i/2-1}(I_*)^{\de_j/2-1}}\right)r^{\de_i/2-1}(1-r)^{\de_j/2-1}\\[1.5em]
		&\qquad\dis\qquad\qquad\qquad\times (1-R)^{(\de_i+\de_j)/2-1}R^{1/2}(I)^{\de_i/2-1}(I_*)^{\de_j/2-1}dv_*dI_* dR dr d\si,\qquad i\ ,\ j>p,
	\end{array} \right.\notag\\\\
&=:Q^+_{ij}(F,G)-Q^-_{ij}(F,G).
\end{align*}
We also remark that condition \eqref{si} is equivalent to 
\begin{align*}
	\frac{1}{C}(\frac{\mu_{ij}}{2}|u|^2+I\chi_{\{i>p\}}+I_*\chi_{\{j>p\}})^{\frac{1-\eta}{2}}\leq B_{ij}\leq C(\frac{\mu_{ij}}{2}|u|^2+I\chi_{\{i>p\}}+I_*\chi_{\{j>p\}})^{\frac{1-\eta}{2}}.
\end{align*}
Hence, we see that \eqref{si} is equivalent to the hard potential assumption. The global thermal equilibrium $M$ is such that $Q(M,M)\equiv 0$ and it is given by
\begin{align*}
	M=(M_1,\cdots,M_n)^T,
\end{align*}
where $$M_i=\frac{c_i}{(2\pi)^{3/2}}e^{-m_i|v|^2/2}\quad \text{for}\ i\leq p,\quad M_i=\frac{c_iI^{\de_i/2-1}}{(2\pi)^{3/2}\Ga(\de_i/2)}e^{-m_i|v|^2/2-I}\quad \text{for}\ i>p,$$ and $\Ga(s)$ is the Gamma function defined by $\Ga(s)=\int_0^\infty x^{s-1}e^{-x}dx$. The collision operator $Q$ has $n+1$ collision invariants $\{e_1,\cdots,e_n,mv,m|v|^2+2\mathbb{I}\}$, $e_i=(\de_{ji})_{1\leq j\leq n}$ with $\de_{ji}$ being the Kronecker delta, $m=(m_1,\cdots,m_n)^T$, and $\mathbb{I}=(0,\cdots,0,I,\cdots,I)^T$, which leads to the conservation laws
\begin{align}
	&\int (F_i(t,x,v)-M_i(v))dxdv=M_{i0},\qquad i\leq p,\label{M1}\\
	&\int (F_i(t,x,v,I)-M_i(v,I))dxdvdI=M_{i0},\qquad i> p,\label{M2}\\
	&\sum_{i=1}^p\int m_i(F_i(t,x,v)-M_i(v))vdxdv+\sum_{i=p+1}^n\int m_i(F_i(t,x,v,I)-M_i(v,I))vdxdvdI=J_0,\label{J}\\
	&\sum_{i=1}^p\int m_i(F_i(t,x,v)-M_i(v))|v|^2dxdv+\sum_{i=p+1}^n\int m_i(F_i(t,x,v,I)-M_i(v,I))(|v|^2-2I)dxdvdI=E_0\label{E}.
\end{align}
Denoting $M^{1/2}=(M^{1/2}_1,\cdots,M^{1/2}_n)^T$ and $M^{1/2}f=(M^{1/2}_1 f_1,\cdots,M^{1/2}_n f_n)^T$, we define the perturbation near the equilibrium state to be  
\begin{align}\label{perturb}
	F=M+M^{1/2}f.
\end{align}
Substituting \eqref{perturb} into \eqref{BE}, we can deduce the equation for the perturbation $f$,
\begin{eqnarray}\label{PBE}
\pa_tf+v\cdot \na_x f+L f=\Ga (f,f),   \quad &\dis f(t=0)=f_0,
\end{eqnarray}
where $f=(f_1,\cdots, f_n)^T$ and $f_0=(f_{10},\cdots, f_{n0})^T$.
The operators $L=(L_1,\cdots,L_n)$ and $\Ga=(\Ga_1,\cdots,\Ga_n)$ in equation \eqref{PBE} are given by 
\begin{align}\label{DefLGa}
	L_if:=-M_i^{-1/2}[Q_i(M,M^{1/2}f)+Q_i(M^{1/2}f,M)],\quad \Ga_i:=M_i^{-1/2}Q_i(M^{1/2}f,M^{1/2}f).
\end{align}
From \eqref{Qi} and the explicit form of $Q_{ij}$, linearised operator $L_i$ can be written as
\begin{align}\label{reLi}
	L_if=\nu_if-K_if=\nu_if-\sum^n_{j=1}K_{ij}f=\nu_if-\sum^n_{j=1}(K^2_{ij}+K^3_{ij}-K^1_{ij})f,
\end{align}
with 
\begin{align}\label{Defnu}
	\nu_i=\left\{
	\begin{array}{rl}
		&\dis\sum^m_{j=1}\int_{(\R^3)^3} W_{ij}M_{j*}dv_*dv^\prime dv^\prime_*+\sum^n_{j=p+1}\int_{(\R^3)^3\times (\R_+)^2} W_{ij}\frac{M_{j*}}{(I_*)^{\de_j/2-1}}dv_*dv^\prime dv^\prime_*dI_*dI'_*,\quad i\leq p,\\[1.5em]
		&\dis\sum^m_{j=1}\int_{(\R^3)^3\times \R_+} W_{ij}\frac{M_{j*}}{(I)^{\de_i/2-1}}dv_*dv^\prime dv^\prime_*dI^\prime \\[1.5em]
		&\hspace{2.5cm}+ \sum^n_{j=p+1}\int_{(\R^3)^3\times (\R_+)^3} W_{ij}\frac{M_{j*}}{(I)^{\de_i/2-1}(I_*)^{\de_j/2-1}}dv_*dv^\prime dv^\prime_* dI_*dI^\prime dI^\prime_*,\quad i> p,
		\end{array}\right.
\end{align}
and 
\begin{align}\label{DefKij}
K_{ij}f&=\left\{
\begin{array}{rl}
	&\dis\int_{(\R^3)^3}(M_{j*}M'_iM'_{j*})^{1/2}\big(\frac{f'_{j*}}{(M'_{j*})^{1/2}}+\frac{f'_i}{(M'_i)^{1/2}}-\frac{f_{j*}}{(M_{j*})^{1/2}}\big) W_{ij}dv_*dv^\prime dv^\prime_*,\ i,j\leq p,\notag\\[1.5em]
	&\dis\int_{(\R^3)^3\times (\R_+)^2}\frac{ (M_{j*}M'_iM'_{j*})^{1/2}}{(I_*I'_*)^{\de_j/4-1/2}}\big(\frac{f'_{j*}}{(M'_{j*})^{1/2}}+\frac{f'_i}{(M'_i)^{1/2}}-\frac{f_{j*}}{(M_{j*})^{1/2}}\big) W_{ij}dv_*dv^\prime dv^\prime_*dI_*dI'_*,\ i\leq p,\ j>p,\\[1.5em]
	&\dis\int_{(\R^3)^3\times \R_+} \frac{ (M_{j*}M'_iM'_{j*})^{1/2}}{(II')^{\de_i/4-1/2}}\big(\frac{f'_{j*}}{(M'_{j*})^{1/2}}+\frac{f'_i}{(M'_i)^{1/2}}-\frac{f_{j*}}{(M_{j*})^{1/2}}\big) W_{ij}dv_*dv^\prime dv^\prime_*dI^\prime,\ i> p,\ j\leq p,\notag\\[1.5em]
	&\dis\int_{(\R^3)^3\times (\R_+)^3} \frac{ (M_{j*}M'_iM'_{j*})^{1/2}}{(II')^{\de_i/4-1/2}(I_*I'_*)^{\de_j/4-1/2}}\big(\frac{f'_{j*}}{(M'_{j*})^{1/2}}+\frac{f'_i}{(M'_i)^{1/2}}-\frac{f_{j*}}{(M_{j*})^{1/2}}\big) W_{ij}dv_*dv^\prime dv^\prime_* dI_*dI^\prime dI^\prime_*,\ i,j> p
\end{array}\right.\notag\\[1.5em]
&=: K^2_{ij}f+K^3_{ij}f-K^1_{ij}f.
\end{align}
 The linearized operator $L$ is symmetric and nonnegative, and an orthogonal basis of $\ker L$ is given by
 \begin{align}\label{Defphi}
&\left\{	\begin{pmatrix}
 		\sqrt{M_1}\\
 		0
 		\\
 		\vdots
 		\\
 		0
 	\end{pmatrix},
 \begin{pmatrix}
 	0
 	\\
 	\vdots
 	\\
 	0\\
 	\sqrt{M_n}
 \end{pmatrix}, \begin{pmatrix}
 m_1\sqrt{M_1}\\
  m_2\sqrt{M_2}
 \\
 \vdots\\
  m_{n-1}\sqrt{M_{n-1}}
 \\
 m_n\sqrt{M_n},
\end{pmatrix}v,\begin{pmatrix}
(m_1|v|^2-3)\sqrt{M_1}
\\
\vdots\\
(m_p|v|^2-3)\sqrt{M_1}\\
(m_{p+1}|v|^2-3-2I-\de_{p+1})\sqrt{M_{p+1}}\\
\vdots
\\
(m_n|v|^2-3-2I-\de_n)\sqrt{M_n}
\end{pmatrix}\right\},\notag\\[1.5em]
&=\{\phi_1,\cdots,\phi_{n+4}\}.
 \end{align}
Consider also the weights $w_\be=(w_{1\be},\cdots,w_{n\be})$ where $w_{i\be}=(1+|v|)^\be$ for $1\leq i\leq p$ and $w_{i\be}=(1+|v|+\sqrt{I})^\be$ for $p+1\leq i\leq n$. For any function $g=g(x,v,I)$ and $1\leq q ,  q'\leq\infty$, the mixed $L^p$ norms are defined by 
\begin{align*}
	\|g\|_{\infty}:=\|g\|_{L^\infty_{x,v,I}},\quad \|g\|_{L^\infty_{v,I}L^q_x}:=\sup_{v\in\R^3,I\in\R_+}\|g(v,I)\|_{L^q_x},\quad \|g\|_{L^\infty_{v,I}(L^q_x\cap L^{q'}_x)}:=\|g\|_{L^\infty_{v,I}L^q_x}+\|g\|_{L^\infty_{v,I}L^{q'}_x}.
\end{align*}
If $g$ depends only on $(x,v)$, that is $g=g(x,v)$, we define the norms as above since the $\sup_{I\in\R_+}$ can be ignored.   After this long introduction we can state our main result.
\begin{theorem}\label{global}
	Let $\be>6$, $(t,x,v,I)\in \R_+\times \Omega \times \R^3\times \R_+$, with $\Omega=\R^3$ or $\T^3$, and with initial data satisfying $F_{i0}=M_i+\sqrt{M_i}f_{i0}\geq0$ for $1\leq i\leq n$. Then, there exists a constant $\eps>0$ such that if
	$$
	\sum^n_{i=1}\left\|w_{i\be} f_{i0}\right\|_{L^\infty_{v,I}(L^1_x\cap L^\infty_x)}\leq \eps, 
	$$
	the Boltzmann equation \eqref{BE} has a unique global mild solution $F=M+\sqrt{M}f$ with $F_i\geq0$ for $1\leq i\leq n$. Moreover, if $\Omega = \R^3$, then there exists a constant $C>0$ such that
	\begin{align}\label{GE}
		\sum^n_{i=1}\left\|w_{i\be} f_i(t)\right\|_{L^\infty_{v,I}(L^2_x\cap L^\infty_x)}\leq C(1+t)^{-\frac{3}{4}}\sum^n_{i=1}\left\|w_{i\be} f_{i0}\right\|_{L^\infty_{v,I}(L^1_x\cap L^\infty_x)},
	\end{align}
and if $\Omega = \T^3$ with $(M_{i0}, J_0, E_0) = (0,0,0),$ $i=1,\cdots n$, where $M_{i0}, J_0, E_0$ are defined in \eqref{M1}, \eqref{M2}, \eqref{J} and \eqref{E}, then there are constants $C,\la>0$ such that
	\begin{align}\label{GET}
	\sum^n_{i=1}\left\|w_{i\be} f_i(t)\right\|_{\infty}\leq Ce^{-\la t}\sum^n_{i=1}\left\|w_{i\be} f_{i0}\right\|_{\infty}.
\end{align}
\end{theorem}

\subsection{Strategy of the proof}
The general framework to construct bounded solutions is the $L^2\cap L^\infty$ argument, an important method to study the classical Boltzmann equation for a single monatomic gas in either the whole space or in a bounded domain, presented in the works \cite{Guo,UY}.  The method requires certain decay properties on the linearized Boltzmann equation and a $L^2$-time decay on the linearized equation which correspond to Lemma \ref{leK}-Lemma \ref{leKmm} and Proposition \ref{leL2} respectively. To perform such approach in the polyatomic-monatomic gas mixture, after double Duhamel iteration on solutions to the linearized equation, the linear $L^\infty$ bound is separated in two key components, one is small by the decay in the kinetic variables $(v,I)$ of the linearized Boltzmann operator, the other is in an integral form over some bounded domain and can be transformed into a $L^2$ integral which is controlled by the aforementioned $L^2$ time decay.  After these estimations, we apply the semigroup generated by the solution operator for the linear equation and a $L^\infty$ bound of nonlinear operator to get the nonlinear $L^\infty$ time decay, this process is carefully presented in Section \ref{Section5}.  Noticing that we consider a mixture of polyatomic and monatomic gases with different masses, the decay of the linearized operator is presented not only in velocity but also in the internal energy variable.  In addition, the dissimilar particles' masses introduces a lack of symmetry which presents a technical difficulty relative to the classical Boltzmann equation case. 

Let us explain these difficulties and our strategy more precisely:  A first difficulty is caused by the extra internal energy variable in gaseous mixture, especially when two particles have different species and masses. The key part to estimate the operator $K_{ij}$, given in Lemma \ref{leKij}, is determined by the gain operators $Q^+_{ij}(M,\sqrt{M}f)$ and $Q^+_{ij}(\sqrt{M}f,M)$. In the case where particles have different masses, we no longer have the symmetry property for the collision kernel $W_{ij}$ defined in \eqref{DefW} as in the classical Boltzmann equation where it holds that $$W_{ii}(v,v_*|v',v'_*)=W_{ii}(v',v'_*|v_*,v)=W_{ii}(v,v_*|v'_*,v'),$$ which can be seen from \eqref{ProW1} to \eqref{ProW3} and allows us to bound both operators by only $Q^+_{ij}(\sqrt{M}f,M)$ with a fine Grad's estimate as in \cite{Grad,Guo}.  We discuss these two cases with different structures separately and prove the decay in both $v$ and $I$ for each of them. In particular, we explain how the extra internal energy variable causes exponential increase in the kernel function $k^2_{ij}$, defined in \eqref{rek2ij}, and induced by $Q^+_{ij}(M,\sqrt{M}f)$, which is bounded by
$$
\int \exp[-C|\frac{m_i}{m_i-m_j}v-\frac{m_j}{m_i-m_j}v_*|^2-C(|v-v_*|^2-C_1(m_i,m_j)\widehat{\De} I)]\hat{\si}\times\cdots,
$$
from \eqref{boundk21}.
In the monatomic mixture, by ignoring the additional relative energy $\widehat{\De} I=I'-I+I_*-I'_*$, the positivity of the relative velocities $|\frac{m_i}{m_i-m_j}v-\frac{m_j}{m_i-m_j}v_*|^2+|v-v_*|^2$ should be enough to establish the decay in $v$, see \cite{BGPS,BG}. However, in the aforementioned integral, due to the appearance of $\widehat{\De}I$ and possible positive $C_1(m_i,m_j)$, even with the positivity of relative velocities it remains a problem to control the possible exponential increase $\exp[C_1\widehat{\De}I]$ and obtain the decay in the internal energy variable $I$. We develop a velocity-internal energy splitting approach to consider the relation between pre/post-collision internal energy and velocity variables in several cases such that if $\widehat{\De} I$ is small relative to $|v|$, the positivity of relative velocities still holds under a small perturbation of $\widehat{\De} I$. Moreover, the decay of $I$ is provided by the cross section $\hat{\si}$. In the case that the internal energy variable is large relative to $|v|$, the decay in both $v$ and $I$ are both obtained from $\hat{\si}$, so that the remainder part of the integral shown above is bounded as
$$
	\int \exp[-C(|v-v_*|^2-C_1(m_i,m_j)\widehat{\De} I)]dv_*\leq C,
	$$
	using $|v-v_*|^2-C_1\widehat{\De} I\sim |v'-v'_*|\geq 0$ from the conservation of energy. The detailed calculations of such arguments for $k^2_{ij}$ are given for the six cases under \eqref{boundk2} for poly-poly interaction, see also the discussions under \eqref{kpm2} for poly-mono interaction.

We note that the internal energy variable and the different masses also lead to the reconstruction of weight functions for each $K^{1,2,3}_{ij}$ in all four types of collisions. Thus, it is required to estimate all 12 cases mentioned above respectively. In the double Duhamel iteration method for either single monatomic or polyatomic gas,  we should expect the kernels $k^{1,2,3}_{ij}$ to absorb some ratio of polynomial weight $w_{\be}$ in $(v,I)$ and the exponential increase on the relative velocity of $v$ and $v_*$ such that
	$$
	\int \exp[\de|v-v_*|^2]\frac{w_\be(v,I)}{w_\be(v_*,I_*)}k^{1,2,3}(v_*,I_*) \rightarrow 0,\quad \text{when}\quad v,I\rightarrow\infty,
	$$
as in Lemma 7 in \cite{Guo} and Lemma 2.2 in \cite{DL}. In the mixture of monatomic gases, even with different masses, we can directly assign $\exp[\de|v-v_*|^2]$ on all $k^{1,2,3}_{ij}$ with an small $\de$ since the kernel function provides exponential decay on $|v-v_*|$. However, we see from the bound of $k^2_{ij}$ in previous paragraph that only the exponential decay in $|v-v_*|^2-C_1(m_i,m_j)\widehat{\De} I$ is stablished, which could produce exponential growth if we add any extra exponential weight on $|v-v_*|^2$ as $(1-\de)|v-v_*|^2-C_1(m_i,m_j)\widehat{\De} I$ could be negative for any $\de>0$.  Hence, we select the weight $$\exp[\de|\frac{m_i}{m_i-m_j}v-\frac{m_j}{m_i-m_j}v_*|^2]$$ on $k^2_{ij}$ such that it can be absorbed by the exponential decay on $|\frac{m_i}{m_i-m_j}v-\frac{m_j}{m_i-m_j}v_*|$. One can refer to \eqref{kpp}, \eqref{k2pp}, \eqref{k3pp} where the arguments are presented. For the polynomial ratio in the above integral, it can be seen from \eqref{boundk} that for $k^3_{ij}$ and $k^{1,2}_{ij}$ we should choose $\frac{w_{i\be}(Z)}{w_{j\be}(Z_*)}$ and $\frac{w_{i\be}(Z)}{w_{i\be}(Z_*)}$ respectively since the index of $f_i$ changes for $K^3_{ij}$ and $K^{1,2}_{ij}$ as in \eqref{reKij} where the spaces of $Z,Z_*$ defined in \eqref{DefZ} will also change corresponding to the polynomial index which describes either the $i^{th}$ or $j^{th}$ species of particles.  In conclusion, all 12 types of linear operators behave differently with different weight functions and variables. Moreover, the Duhamel iteration deduced in \eqref{hi} requires us to track one more collision, which indicates we are dealing with the complexity of eight types for double collisions, that is, Poly/Mono-Poly/Mono-Poly/Mono collisions, as well as the behaviors of different kernel functions $k^{1,2,3}_{ij}$ for each collision. 

Additionally, the polyatomic cross section $\si_{ij}$ defined in \eqref{si} is more singular compared to the monatomic case in terms of $|v-v_*|^2+I+I_*$ with the power $-(\eta+\de_i+\de_j)/2$.  Indeed, in the monatomic case where we ignore the internal energy $\de_i$ and $\de_j$ such that $\si_{ij}\sim |v-v_*|^{-\eta}$, the singularity for small relative velocity also vanishes after integration on $v_*$ for both hard and soft potentials with some $\eta$ that is not too large.  Refer to the proof of Lemma \ref{leKmm} for details. On the other hand, this singularity is useful to obtain the decay in $I$ for large internal energy, see \eqref{boundtildeE} and \eqref{boundI}. Hence, it is delicate to control the cross section in the various cases of linearized operators and collision types. The general principle of the analysis is that when certain internal energy variables are small, we use the conservation of energy $|v-v_*|^2+I+I_*\sim |v'-v'_*|^2+I'+I'_*$ to split the singularities to other relative velocities and internal energies such that the small variables are integrable.   And when some internal energy variables are large, we preserve those singularities to obtain the decay as well as the integrability around infinity as in \eqref{boundtildeE} where both $I$ and $I_*$ decay fast when they are large and are integrable when they are close to zero. However, in case of poly-mono interaction, such decay in internal energy may not be integrable for small or large values; for example, the decay in pre-collision internal energy is of order $-1$ for the last term in \eqref{controltildeEpm}.  In order to handle all these cases, we establish an approach by a series of changes of variables and comparisons between $I$ and $v$ to trade decay of velocity, integrability, and decay of internal energies.  This is the reason that the decay in $v$ obtained from the linearized operator is $O(|v|^{-1/2})$, such as in \eqref{k2pp} and \eqref{boundJ1}, rather than $O(|v|^{-1})$ for the monatomic gases in \cite{Guo,Grad}. 

Let us mention that the analysis includes the whole space and the torus for the spatial domain. Due to the algebraic slow decay when the domain is the whole space, we use the Fourier transform to study the hypocoercivity property on the degenerate dissipative part of the equation. Such approach reveals the explicit decay in time and frequency and allows us to get the optimal decay in the sense of $L^2$.  This hypocoercivity method on Boltzmann equation has been used in \cite{Duan,Duan-Stain1,Duan-Stain2}. In addition, later in \eqref{eqabc}, we derive the system of the $n+4$ macroscopic quantities $a_i$, $b$ and $c$, which correspond to the perturbed macroscopic densities, bulk velocity and temperature respectively.  Although the formal Euler or Navier-Stokes type systems are derived from polyatomic-monatomic gaseous mixture in \cite{BBG,Shahine}, the rigorous justification of hydrodynamic limit remains open for either such gaseous mixture or, even, single polyatomic gases. The $a_i,b,c$ system derived in this paper shows the structure of the perturbed Euler type equation for macroscopic quantities of the solution near the steady state, which gives a potential application to fluid limit problems in a future work.

\subsection{Additional related literature}
In recent years the mathematical study of gaseous mixtures became an active area due to its multiple applications.  At the kinetic theory level, Boltzmann based models for gas mixtures are discussed in \cite{BBBD,BBG,BGPS}. The construction of solutions and integrability propagation for spatial homogeneous system are proved in \cite{ACG,AC}, with the well-posedness of the purely monoatomic case discussed in \cite{GP}. In the perturbation near equilibrium framework for inhomogeneous systems, the compactness and Fredholm properties of the linearized Boltzmann operator are studied in \cite{Bernhoff} from a probabilistic formulation of the collision operator and \cite{BST,Shahine} based on the Borgnakke-Larsen process. However, as far as we know, there are few results on the well-posedness theory.

Additional literature for purely monoatomic setting includes \cite{BBBG,Briant} for properties of linearized collision operator, \cite{BG} for the spatially inhomogeneous well-posedness on the torus close to thermal equilibrium, and \cite{BGPCS} for the one dimensional case.

In the polyatomic setting, where the gas is described with an additional internal energy variable,  there is extensive work in case of a single polyatmoic agent.  We can refer to \cite{BL,Desvillettes,DMS} where the mathematical models are calculated with the Borgnakke-Larsen approach. We mention \cite{Bernhoff1,Bernhoff2,BST1,BST2} for the properties on the linearized operator, and \cite{GP1,DL,KS} for constructions of solutions in the spatially homogeneous and the spatially inhomogeneous close to thermal equilibrium settings.

\subsection{Notations}
The proofs in this paper involve several notations on mixed $L^p$ norms for both scalar and vector valued functions.  First, the $L^2$ norm on $v,I$ is defined by \begin{align*}
	\|f\|^2_{v,I}=\|f\|^2_{L^2_{v,I}}:=\left\{
	\begin{array}{rl}
		&\int_{\R^3}|f(v)|^2dv,\quad \textit{if}\quad
		f=f(v),\\[.5em]
		&\int_{\R^3\times\R_+} |f(v,I)|^2dvdI,\quad \textit{if}\quad  f=f(v,I).
	\end{array}\right.
\end{align*}
If $f=(f_1,\cdots,f_n)$ is a vector valued function, then
\begin{align*}
	\|f\|^2_{v,I}:=\sum^n_{i=1}\|f_i\|^2_{L^2_{v,I}}.
\end{align*}
Furthermore, if $f$ depends additionally on $y\in\Omega$ such that $f=(f_1(y,v),\cdots,f_p(y,v),f_{p+1}(y,v,I),\cdots,f_n(y,v,I))$, then for $1\leq q,q'\leq \infty$ we denote
\begin{align*}
	\|f\|^2=&\|f\|^2_{L^2_{y,v,I}}:=\sum^n_{i=1}\int_{\Omega}\|f_i(y)\|^2_{L^2_{v,I}}dy,\\
\|f\|^2_{L^2_{v,I}L^q_{y}}:=&\sum^p_{i=1}\int_{\R^3}\big|\int_{\Omega}|f_i(y,v)|^qdy\big|^\frac{2}{q}dv+\sum^n_{i=p+1}\int_{\R^3\times\R_+}\big|\int_{\Omega}|f_i(y,v,I)|^qdy\big|^\frac{2}{q}dvdI,
\end{align*}
and
\begin{align*}
	\|f\|_{L^2_{v,I}(L^q_{y}\cap L^{q'}_{y})}:=\|f\|_{L^2_{v,I}L^q_{y}}+\|f\|_{L^2_{v,I} L^{q'}_{y}}.
\end{align*}

In addition, we use $\langle\cdot,\cdot\rangle$ to denote the inner product in $\C^n$ whereas the $L^2$ inner product is given by $(\cdot,\cdot)$.

\section{Proof of the Main Theorem}
The proof for Theorem \ref{global} mainly depends on five important technical lemmas which will be proved in the next section.  In this section we describe the strategy and proof of the main theorem by introducing these lemmas and showing how they are used along the the proof.  As a first step, it is crucial to control the linearized operators $K_{ij}$.
\begin{lemma}\label{leKij}
For $1\leq i,j\leq n,$ the linearized operator $K_{ij}$ defined in \eqref{DefKij} can be written as
\begin{align}\label{reKij}
	K_{ij}f=
	&\dis\int_{\CZ_j}(k^2_{ij}(Z,Z_*)-k^1_{ij}(Z,Z_*))f_j(Z_*)dZ_*+\int_{\CZ_i}k^3_{ij}(Z,Z_{*})f_i(Z_*)dZ_*,
\end{align}	
with
\begin{align}\label{boundk}
	\int_{\CZ_j}(|k^1_{ij}(Z,Z_*)|+|k^2_{ij}(Z,Z_*)|)\frac{w_{i\be}(Z)}{w_{j\be}(Z_*)}dZ_*+\int_{\CZ_i}|k^3_{ij}(Z,Z_{*})|\frac{w_{i\be}(Z)}{w_{i\be}(Z_*)}dZ_*\leq C,
\end{align}
for any $\be\geq 0$ where $C>0$ depends only on $\de_i$, $m_i$ and $\eta$.
\end{lemma}
Estimate \eqref{boundk} on $K_{ij}$ is relatively coarse since in the proof of nonlinear estimates only the boundedness of $\int k^{1,2,3}(Z,Z_*)\frac{w_{i\be}(Z)}{w_{j\be}(Z_*)}dZ_*$ are needed. However, in the linear $L^\infty$ estimates, we require subtler properties of $\int k^{1,2,3}(Z,Z_*)\frac{w_{i\be}(Z)}{w_{j\be}(Z_*)}dZ_*$ including its decay on $v$ and $I$ with additional weights.  At the end of Section \ref{Section3} we give a detailed study on $K_{ij}$ containing four lemmas which imply \eqref{reKij} and \eqref{boundk}.\\

Second, we use polynomial equivalency on the collision frequency $\nu_i$ defined in \eqref{Defnu}. The proof is given in the second part of Section \ref{Section3}.
\begin{lemma}\label{lenu}
	There exist positive constants $C$ and $\nu_0$ depending on $\de_i$, $m_i$ and $\eta$ such that
	\begin{align}
		\nu_0\leq\frac{1}{C}(1+|v|+\sqrt{I}\chi_{\{i>p\}})^{1-\eta}\leq &\nu_i(v,I) \leq C(1+|v|+\sqrt{I}\chi_{\{i>p\}})^{1-\eta},\quad 1\leq i\leq n. \label{nu}
	\end{align}
\end{lemma}
Third, regarding the nonlinear problem, we have the following bound for nonlinear operator $\Ga_i$. See the last part of Section \ref{Section3} for a proof of this lemma.
\begin{lemma}\label{lebound}
	Let $\Ga_i$ be defined in \eqref{DefLGa}. It holds that
	\begin{align}
		\left|w_i\Ga_i(f,f) \right|&\leq C\|(w_jf_j)(t)\|^2_\infty,\label{estGa}\\
		\|w_i\Ga_i(f,f) \|_{L^2_x}&\leq C\nu_i\sum^n_{j=1}\|(w_jf_j)(t,Z_j)\|_{L^2_x}^2+C\nu_i\sum^n_{j=1}\|(w_jf_j)(t)\|^2_\infty,\label{estGaL2}\\
		\|w_i\Ga_i(f,f) \|_{L^1_x}&\leq C\nu_i\sum^p_{j=1}\|(w_jf_j)(t,Z_j)\|_{L^2_x}^2,\label{estGaL1}
	\end{align}
for a constant $C>0$ depending on $\de_i$, $m_i$ and $\eta$.	
\end{lemma}

When considering the solution to the nonlinear problem, it is written in terms of the semigroup generated by the linearized Botlzmann equation given by
  \begin{eqnarray}\label{Lineareq}
	\pa_tf+v\cdot \na_x f+L f=0,   \quad &\dis f(t=0)=f_0.
\end{eqnarray}

To perform the $L^2\cap L^\infty$ approach on the linear equation, it is necessary to establish the $L^2$ decay for the solutions to \eqref{Lineareq} which is proved in Section \ref{Section4}. 
\begin{proposition}\label{leL2}
	Let $f$ be a solution to the linearized Boltzmann equation \eqref{Lineareq}. Then, there exist positive constants $\la$ and $C$ such that
	\begin{align}\label{L2decayR3}
		\|f(t)\|\leq C(1+t)^{-\frac{3}{4}}\|f_0\|_{L^2_{v,I}(L^1_x\cap L^2_x)},
	\end{align}
	when $\Omega = \R^3$, and
	\begin{align}\label{L2decaytorus}
		\|f(t)\|\leq Ce^{-\la t}\|f_0\|,
	\end{align}
	if $(M_{i0}, J_0, E_0) = (0,0,0)$ when $\Omega = \T^3$, where the constants $\la,C>0$ depend only on $\de_i$, $m_i$ and $\eta$.   Recall that $M_{i0}$, $J_0$ and $E_0$ are defined in \eqref{M1}, \eqref{M2}, \eqref{J} and \eqref{E}.
\end{proposition}
The $L^\infty$ solution is constructed with an additional weight $w_\be$. Denote the function $h=w_\be f$ and the operator $K_{w_\be}h=(K_{1{w_\be}}h,\cdots,K_{n{w_\be}})={w_\be}K\frac{h}{{w_\be}}$, where
\begin{equation*}
w_\be f:=(w_{1\be}f_1,\cdots,w_{n\be}f_n)^T, \quad {w_\be}K:=(w_{1\be}K_1,\cdots,w_{n\be}K_n)^T , \quad and \quad \frac{h}{{w_\be}}:=(\frac{h_1}{w_{1\be}},\cdots,\frac{h_n}{w_{n\be}})^T.
\end{equation*}
We rewrite the linearized Boltzmann equation in terms of $h$ as
\begin{align}\label{LBEh}
	\pa_th+v\cdot \na_x h+\nu h-K_{w_\be}h=0,\qquad h(t=0)=h_0=(h_{10},\cdots,h_{n0})^T=w_\be f_0.
\end{align}
Let $S_\be(t)=(S_{\be 1}(t),\cdots,S_{\be n}(t))^T$ be the solution operator solving \eqref{LBEh}, then the mild solution of \eqref{LBEh} is defined by
\begin{align}\label{mildlinear}
	h = S_\be(t)h_0 = \begin{pmatrix}
		e^{-\nu_1(v)t}h_{10}(x-vt,v)+\int_0^t e^{-\nu_1(v)(t-s)}(K_{1{w_\be}}h)(s,x-v(t-s),v)ds\\
		\vdots
		\\
		e^{-\nu_n(v,I)t}h_{n0}(x-vt,v,I)+\int_0^t e^{-\nu_n(v,I)(t-s)}(K_{n{w_\be}}h)(s,x-v(t-s),v,I)ds
	\end{pmatrix},
\end{align}
or, equivalently, componentwise
\begin{align}\label{mildh}
	h_i(t) = S_{\be i}(t)h_0&=e^{-\nu_i(Z)t}h_{i0}(x-vt,Z) \notag\\
	&\quad + \sum^n_{j=1}\int_0^t e^{-\nu_i(Z)(t-s)}\int_{\CZ_j}(k_{ij}^2-k^1_{ij})(Z,Z_*)\frac{w_{i\be}(Z)}{w_{j\be}(Z_*)}h_j(s,x-v(t-s),Z_*)dZ_* ds\notag\\
	&\quad-\sum^n_{j=1}\int_0^t e^{-\nu_i(Z)(t-s)}\int_{\CZ_i}k_{ij}^3(Z,Z_*)\frac{w_{i\be}(Z)}{w_{i\be}(Z_*)}h_i(s,x-v(t-s),Z_*)dZ_* ds,
\end{align}
for $1\leq i\leq n$ by \eqref{reKij}.
The final ingredient is the following $L^\infty$ decay property for $S_\be(t)h_0$, which is proved in Section \ref{Section5}.
\begin{proposition}\label{lelinearLinfty}
	Assume $\Omega = \R^3$ and let $S_\be(t)h_0=h$ be the solution for \eqref{LBEh}.  Then, for any $\be>6$ there exists a constant $C$ depending on $\de_i$, $m_i$ and $\eta$ such that
	\begin{align}\label{boundh}
		\|S_\be(t)h_0\|_{L^\infty_{v,I}(L^2_x\cap L^\infty_x)}=\|h(t)\|_{L^\infty_{v,I}(L^2_x\cap L^\infty_x)}\leq C(1+t)^{-\frac{3}{4}}\|h_0\|_{L^\infty_{v,I}(L^1_x\cap L^\infty_x)},
	\end{align}
	for any $t\geq 0$. If $\Omega = \T^3$ and $(M_{i0}, J_0, E_0) = (0,0,0)$, then for any $\be>6$ there exist constants $C$ and $\la$ such that
	\begin{align}\label{boundhT}
		\|S_\be(t)h_0\|_{\infty}=\|h(t)\|_{\infty}\leq Ce^{-\la t}\|h_0\|_\infty,
	\end{align}
	for any $t\geq 0$, where $\la,C$ depend on $\de_i$, $m_i$ and $\eta$.
\end{proposition}
With these results at hand we can prove the main result.
\begin{proof}[Proof of Theorem \ref{global}]
	Assume $x\in\Omega = \R^3$. Rewrite the Boltzmann equation \eqref{BE} in terms of $h=w_{\be}f$ as follows:
	\begin{align*}
		\pa_th+v\cdot \na_x h+\nu h=K_{w_\be}h+w\,\Ga\big(\frac{h}{w_\be},\frac{h}{w_\be}\big),   \quad &\dis h(t=0)=h_{0}={w_\be}f_{0},
	\end{align*}
	where $${w_\be}\Ga(\frac{h}{w_\be},\frac{h}{w_\be})=\begin{pmatrix}
		w_{1\be}\Ga_1 (\frac{h_1}{w_{1\be}},\frac{h_1}{w_{1\be}})\\
		\vdots
		\\
		w_{n\be}\Ga_n (\frac{h_n}{w_{n\be}},\frac{h_n}{w_{n\be}})
	\end{pmatrix}.$$
	Recalling the solution operator $S_\be(t)$ defined in \eqref{mildlinear} for the linear problem we have that
	\begin{align*}
		h&=S_\be(t)h_0+\int^t_0\left\{S_\be(t-s){w_\be}\Ga(\frac{h}{w_\be},\frac{h}{w_\be})(s)\right\}(s,x-v(t-s),Z)ds.
	\end{align*}
	We use \eqref{mildlinear} to further write
	\begin{align}\label{rehi}
		h_i(t,x,Z)&=S_{\be i}(t)h_0+\int^t_0e^{-\nu_i(Z)(t-s)}w_{i\be}\Ga_i (\frac{h_i}{w_{i\be}},\frac{h_i}{w_{i\be}})(s,x-v(t-s),Z)ds\notag\\
		&\qquad+\int^t_0\int^t_se^{-\nu_i(Z)(t-s_1)}K_{iw_\be}\left\{S_{\be}(s_1-s)w\Ga(\frac{h}{w},\frac{h}{w})(s)\right\}(s,x-v(t-s),Z)d{s_1}ds,
	\end{align}
	where $1\leq i \leq n$. We need to control the $L^\infty_{v,I}(L^2_x\cap L^\infty_x)$ norm for each term above.
	It holds by \eqref{boundh} that
	\begin{align}\label{Sh0}
		\|S_{\be i}(t)h_0\|_{L^\infty_{v,I}(L^2_x\cap L^\infty_x)}\leq C(1+t)^{-\frac{3}{4}}\|h_0\|_{L^\infty_{v,I}(L^1_x\cap L^\infty_x)}.
	\end{align}
	Moreover, we use \eqref{estGa} and \eqref{estGaL2} to obtain that
	\begin{align}\label{SGa}
		\Big\|\int^t_0&e^{-\nu_i(Z)(t-s)}w_{i\be}\Ga_i (\frac{h_i}{w_{i\be}},\frac{h_i}{w_{i\be}})(s,x-v(t-s),Z)ds\Big\|_{L^\infty_{v,I}(L^2_x\cap L^\infty_x)}\notag\\
		\leq& C\sup_{0\leq t<\infty}\|(1+t)^{\frac{3}{4}}h(t)\|^2_{L^\infty_{v,I}(L^2_x\cap L^\infty_x)}\int^t_0e^{-\nu_i(Z)(t-s)}\nu_i(Z)(1+s)^{-\frac{3}{2}}ds\notag\\
		\leq& C(1+t)^{-\frac{3}{4}}\sup_{0\leq t<\infty}\|(1+t)^{\frac{3}{4}}h(t)\|^2_{L^\infty_{v,I}(L^2_x\cap L^\infty_x)}.
	\end{align}
	For the last term on the right hand side of \eqref{rehi}, it follows from \eqref{reKij} that
	\begin{align}\label{KSGa}
		\int^t_0\int^t_s &e^{-\nu_i(Z)(t-s_1)}K_{iw_{\be}}\left\{S(s_1-s)w_{\be}\Ga(\frac{h}{w_{\be}},\frac{h}{w_{\be}})(s)\right\}(s,x-v(t-s),Z)d{s_1}ds\notag\\
		=&\sum^n_{j=1}\int^t_0\int^t_se^{-\nu_i(Z)(t-s_1)} \int_{\CZ_j}(k^2_{ij}-k^1_{ij})(Z,Z_*)\frac{w_{i\be}(v,I)}{w_{j\be}(v_*,I_*)}\notag\\
		&\qquad\qquad\qquad\times\left\{S_{\be j}(s_1-s)w_{\be}\Ga(\frac{h}{w_{\be}},\frac{h}{w_{\be}})(s)\right\}(s,x-v(t-s),Z_*)dZ_*d{s_1}ds\notag\\
		&+\sum^n_{j=1}\int^t_0\int^t_se^{-\nu_i(Z)(t-s_1)}\int_{\CZ_i}k^3_{ij}(Z,Z_{*})\frac{w_{i\be}(v,I)}{w_{i\be}(v_*,I_*)}\notag\\
		&\qquad\qquad\qquad\times\left\{S_{\be j}(s_1-s)w\Ga(\frac{h}{w_{\be}},\frac{h}{w_{\be}})(s)\right\}(s,x-v(t-s),Z_*)dZ_*d{s_1}ds.
	\end{align}
	Note that from \eqref{mildh}, for any $r\in \R$, we have the property
	\begin{align*}
		\frac{1}{w_{ri}(Z)}&S_{\be i}(t)h_0=e^{-\nu_i(Z)t}\frac{1}{w_{ri}(Z)}h_{i0}(x-vt,Z)\notag\\
		&+\sum^n_{j=1}\int_0^t e^{-\nu_i(Z)(t-s)}\int_{\CZ_j}(k_{ij}^2-k^1_{ij})(Z,Z_*)\frac{w_{i\be}(Z)}{w_{j\be}(Z_*)}\frac{w_{rj}(Z_*)}{w_{ri}(Z)}\frac{1}{w_{rj}(Z_*)}h_j(s,x-v(t-s),Z_*)dZ_* ds\notag\\
		&\qquad-\sum^n_{j=1}\int_0^t e^{-\nu_i(Z)(t-s)}\int_{\CZ_i}k_{ij}^3(Z,Z_*)\frac{w_{i\be}(Z)}{w_{i\be}(Z_*)}\frac{w_{rj}(Z_*)}{w_{ri}(Z)}\frac{1}{w_{rj}(Z_*)}h_i(s,x-v(t-s),Z_*)dZ_* ds,
	\end{align*}
	or, equivalently,
	\begin{align*}
		\frac{1}{w_{r}}S_{\be}(t)h_0=S_{\be-r}(t)(\frac{h_0}{w_r}).
	\end{align*}
	We obtain from the above property, \eqref{nu}, \eqref{boundk}, \eqref{boundh} and \eqref{estGa} that
	\begin{align*}
		\Big|\int^t_0&\int^t_s e^{-\nu_i(Z)(t-s_1)} \int_{\CZ_j}(k^2_{ij}-k^1_{ij})(Z,Z_*)\frac{w_{i\be}(v,I)}{w_{j\be}(v_*,I_*)}\\
		&\hspace{3cm}\times\left\{S_{\be j}(s_1-s)w_{\be}\Ga(\frac{h}{w_{\be}},\frac{h}{w_{\be}})(s)\right\}(s,x-v(t-s),Z_*)dZ_*d{s_1}ds\Big|\notag\\
		=&\Big|\int^t_0\int^t_se^{-\nu_i(Z)(t-s_1)} \int_{\CZ_j}(k^2_{ij}-k^1_{ij})(Z,Z_*)\frac{w_{i\be}(v,I)}{w_{j\be}(v_*,I_*)}w_{1-\eta,j}(Z_*)\notag\\
		&\hspace{3cm}\times\left\{S_{\be-1+\eta, j}(s_1-s)\frac{w_{\be}}{w_{1-\eta}}\Ga(\frac{h}{w_{\be}},\frac{h}{w_{\be}})(s)\right\}(s,x-v(t-s),Z_*)dZ_*d{s_1}ds\Big|\notag\\
		\leq&C\sup_{0\leq t<\infty}\|(1+t)^{\frac{3}{4}}h(t)\|^2_{L^\infty_{v,I}(L^2_x\cap L^\infty_x)}\int^t_0\int^t_se^{-\nu_i(Z)(t-s_1)}\nu_i(Z)\notag\\
		&\hspace{2cm}\times \int_{\CZ_j}|(k^2_{ij}-k^1_{ij})(Z,Z_*)|\frac{w_{i\be}(v,I)}{w_{j\be}(v_*,I_*)}\frac{\nu_j(Z_*)}{\nu_i(Z)}(1+s_1-s)^{-\frac{3}{4}}(1+s)^{-\frac{3}{2}}dZ_*d{s_1}ds\notag\\
		\leq&C\sup_{0\leq t<\infty}\|(1+t)^{\frac{3}{4}}h(t)\|^2_{L^\infty_{v,I}(L^2_x\cap L^\infty_x)}\int^t_0\int^{s_1}_0e^{-\nu_i(Z)(t-s_1)}\nu_i(Z)(1+s_1-s)^{-\frac{3}{4}}(1+s)^{-\frac{3}{2}}dsd{s_1}\notag\\
		\leq&C(1+t)^{-\frac{3}{4}}\sup_{0\leq t<\infty}\|(1+t)^{\frac{3}{4}}h(t)\|^2_{L^\infty_{v,I}(L^2_x\cap L^\infty_x)}.
	\end{align*}
	And similarly,
	\begin{align*}
		\Big|\int^t_0&\int^t_se^{-\nu_i(Z)(t-s_1)}\int_{\CZ_i}k^3_{ij}(Z,Z_{*})\frac{w_{i\be}(v,I)}{w_{i\be}(v_*,I_*)}\\
		&\hspace{3cm}\times\left\{S_{\be j}(s_1-s)w\Ga(\frac{h}{w_{\be}},\frac{h}{w_{\be}})(s)\right\}(s,x-v(t-s),Z_*)dZ_*d{s_1}ds\Big|\notag\\
		\leq&C(1+t)^{-\frac{3}{4}}\sup_{0\leq t<\infty}\|(1+t)^{\frac{3}{4}}h(t)\|^2_{L^\infty_{v,I}(L^2_x\cap L^\infty_x)}.
	\end{align*}
	These two inequalities above, together with \eqref{KSGa}, show that
	\begin{align}\label{KSG}
		\Big|\int^t_0&\int^t_se^{-\nu_i(Z)(t-s_1)}K_{iw_{\be}}\left\{S(s_1-s)w_{\be}\Ga(\frac{h}{w_{\be}},\frac{h}{w_{\be}})(s)\right\}(s,x-v(t-s),Z)d{s_1}ds\Big|\notag\\
		\leq&C(1+t)^{-\frac{3}{4}}\sup_{0\leq t<\infty}\|(1+t)^{\frac{3}{4}}h(t)\|^2_{L^\infty_{v,I}(L^2_x\cap L^\infty_x)}.
	\end{align}
	In the same way as described for \eqref{KSG}, it also holds that
		\begin{align}\label{KSGL2}
		\Big(\int_{\Omega}\Big|\int^t_0 & \int^t_se^{-\nu_i(Z)(t-s_1)}K_{iw_{\be}}\left\{S(s_1-s)w_{\be}\Ga(\frac{h}{w_{\be}},\frac{h}{w_{\be}})(s)\right\}(s,x-v(t-s),Z)d{s_1}ds\Big|^2dx\Big)^\frac{1}{2}\notag\\
		\leq&C(1+t)^{-\frac{3}{4}}\sup_{0\leq t<\infty}\|(1+t)^{\frac{3}{4}}h(t)\|^2_{L^\infty_{v,I}(L^2_x\cap L^\infty_x)}.
	\end{align}
	We collect \eqref{rehi}, \eqref{Sh0}, \eqref{SGa}, \eqref{KSG} and \eqref{KSGL2} to obtain that
	\begin{align*}
		\sup_{0\leq t<\infty}\|(1+t)^{\frac{3}{4}}h_i(t)\|_{L^\infty_{v,I}(L^2_x\cap L^\infty_x)}\leq C\|h_0\|_{L^\infty_{v,I}(L^1_x\cap L^\infty_x)}+C\sup_{0\leq t<\infty}\|(1+t)^{\frac{3}{4}}h(t)\|^2_{L^\infty_{v,I}(L^2_x\cap L^\infty_x)},
	\end{align*}
	which yields \eqref{GE} by taking summation on $i$ and assuming $
	\sum^n_{i=1}\left\|w_{i\be} f_{i0}\right\|_{\infty}\leq \eps
	$ for a sufficiently small $\eps>0$ depending only on $\de_i$, $m_i$ and $\eta$.  In the case $\Omega = \T^3$, we can obtain \eqref{GET} using \eqref{L2decaytorus}, \eqref{boundhT}, and analogous arguments. In this way Theorem \ref{global} is proved.
\end{proof}

\section{Pointwise Bounds for Collision Operators}\label{Section3}
In this section, we study three collision operators $\nu_i$, $K_{ij}$ and $\Ga_i$ induced by the linearization. For later use, we not only need $K$ to be compact, but also prove the decay properties in velocity and internal energy, where the proof makes use of a wide range of changes of variables which would be applied in the estimates for $\nu_i$ and $\Ga_i$.  Let us start the analysis with the operator $K_{ij}$ first. 

\subsection{Decay property of the linearized operator}
The linear operator $K$ behaves differently in the four cases of collisions, so they will be approached separately.  We start with Poly-Poly collision where all the interacting distribution functions have an internal energy variable.  Consequently, we can establish more general bounds that can be used in the other cases.
\begin{lemma}[Poly-Poly Collision]\label{leK}
	For $i,j>p,$ the linearized operator defined in \eqref{DefKij} can be written as
	\begin{align*}
		\dis K_{ij}f=\int_{\R^3\times \R_+}(k^2_{ij}(v,I,v_*,I_*)f_j(v_*,I_*)+k^3_{ij}(v,I,v_*,I_*)f_i(v_*,I_*)-k^1_{ij}(v,I,v_*,I_*)f_j(v_*,I_*) )dv_*dI_*,
	\end{align*}
where $k^{1,2,3}_{ij}$ satisfies the following decay properties:
\begin{align}
	&\int_{\R^3\times\R_+} k^1_{ij}(v,I,v_*,I_*)\frac{w_{i\be}(v,I)}{w_{j\be}(v_*,I_*)}e^\frac{\de|v-v_*|^2}{64}(1+I_*)^{\frac{1}{8}} dv_*dI_*\leq \frac{C}{1+|v|^{\frac{1}{2}}+I^\frac{1}{8}},\label{kpp}\\
	&\int_{\R^3\times\R_+} k^2_{ij}(v,I,v_*,I_*)\frac{w_{i\be}(v,I)}{w_{j\be}(v_*,I_*)}e^\frac{\de|m_iv-m_jv_*|^2}{64}(1+I_*)^{\frac{1}{8}} dv_*dI_*\leq \frac{C}{1+|v|^{\frac{1}{2}}+I^\frac{1}{8}},\label{k2pp}\\
	&\int_{\R^3\times\R_+} k^3_{ij}(v,I,v_*,I_*)\frac{w_{i\be}(v,I)}{w_{i\be}(v_*,I_*)}e^\frac{\de|v-v_*|^2}{64}(1+I_*)^{\frac{1}{8}} dv_*dI_*\leq \frac{C}{1+|v|^{\frac{1}{2}}+I^\frac{1}{8}}.\label{k3pp}
\end{align}
Here $\dis \de=\min_{1\leq i,j\leq n}\{m_i,\frac{(\sqrt{m_i}-\sqrt{m_j})^2}{16(m_i-m_j)^2},\frac{1}{m_i^2}\}.$
\end{lemma}
\begin{proof}
We rewrite \eqref{DefKij} as
\begin{align}\label{Kij}
	K_{ij}f&=\int_{\CZ_j\times \CZ_i\times \CZ_j} \frac{ (M_{j*}M'_iM'_{j*})^{1/2}}{(II')^{\de_i/4-1/2}(I_*I'_*)^{\de_j/4-1/2}}\big(\frac{f'_{j*}}{(M'_{j*})^{1/2}}+\frac{f'_i}{(M'_i)^{1/2}}-\frac{f_{j*}}{(M_{j*})^{1/2}}\big) W_{ij}dZ_*dZ^\prime dZ^\prime_*\notag\\
	&=K^2_{ij}f+K^3_{ij}f-K^1_{ij}f.
\end{align}
First consider $k^1_{ij}(Z,Z_*)$. It is straightforward to obtain that
\begin{align*}
	K^1_{ij}f&=\int_{\CZ_j\times \CZ_i\times \CZ_j} \frac{ (M_{j*}M'_iM'_{j*})^{1/2}}{(II')^{\de_i/4-1/2}(I_*I'_*)^{\de_j/4-1/2}}\frac{f_{j*}}{(M_{j*})^{1/2}} W_{ij}dZ_*dZ^\prime dZ^\prime_*\notag\\
	&=\int_{\CZ_j} k^1_{ij}(Z,Z_*)f_{j}(t,x,Z_*)dZ^\prime_*,  
\end{align*}
with
\begin{align}\label{rek1ij}
	k^1_{ij}(Z,Z_*)&=\int_{\CZ_i\times \CZ_j} \frac{ (M'_iM'_{j*})^{1/2}}{(II')^{\de_i/4-1/2}(I_*I'_*)^{\de_j/4-1/2}}W_{ij}(Z,Z_*|Z',Z'_*)dZ^\prime dZ^\prime_*.  
\end{align}
Denoting $V_{ij}=\frac{m_i v+m_j v_*}{m_i+m_j}$, $V_{ij}^\prime=\frac{m_i v'+m_j v'_*}{m_i+m_j}$, we have
\begin{align*}
	k^1_{ij}(Z,Z_*)&=\int_{\CZ_i\times \CZ_j} \frac{ (M'_iM'_{j*})^{1/2}}{(II')^{\de_i/4-1/2}(I_*I'_*)^{\de_j/4-1/2}}(m_i+m_j)^2m_im_j(I)^{\de_i/2-1}(I_*)^{\de_j/2-1}\si_{ij}\frac{|u|}{|u'|}\notag\\
	&\qquad\times\bm{\de}_3(v+v_*-v^\prime-v^\prime_*)\bm{\de}_1(\frac{|v|^2}{2}+\frac{|v_*|^2}{2}-\frac{|v^\prime|^2}{2}-\frac{|v^\prime_*|^2}{2}+(I-I^\prime)\chi_{\{i>p\}}+(I_*-I^\prime_*)\chi_{\{j>p\}})dZ^\prime dZ^\prime_*\notag\\
	&=\int_{\CZ_i\times \CZ_j} \frac{ (M'_iM'_{j*})^{1/2}}{(II')^{\de_i/4-1/2}(I_*I'_*)^{\de_j/4-1/2}}(I)^{\de_i/2-1}(I_*)^{\de_j/2-1}\si_{ij}\frac{|u|}{|u'|^2}\notag\\
	&\qquad\times\bm{\de}_3(V_{ij}-V'_{ij})\bm{\de}_1(\sqrt{|u|^2-\frac{2\De I}{\mu_{ij}}}-|u^\prime|)dZ^\prime dZ^\prime_*.  
\end{align*}
Letting $\om=\frac{u^\prime}{|u^\prime|}$ and noticing that
\begin{align}\label{changeofvariables}
	dv'dv'_*=|u'|^2dV'_{ij}d|u'|d\om,\end{align} for $p+1\leq i,j\leq n$, it holds that
\begin{align}\label{rek1}
	k^1_{ij}(v,I,v_*,I_*)&=\int_{(\R^3)^2\times (\R_+)^2} \frac{ (M'_iM'_{j*})^{1/2}}{(II')^{\de_i/4-1/2}(I_*I'_*)^{\de_j/4-1/2}}(I)^{\de_i/2-1}(I_*)^{\de_j/2-1}\si_{ij}\frac{|u|}{|u'|^2}\notag\\
	&\qquad\times\bm{\de}_3(V_{ij}-V'_{ij})\bm{\de}_1(\sqrt{|u|^2-\frac{2\De I}{\mu_{ij}}}-|u^\prime|)|u'|^2d|u^\prime|d\om dV^\prime_{ij} dI^\prime dI^\prime_*\notag\\
	&=\int_{\S^2\times (\R_+)^2} \frac{ (M'_iM'_{j*})^{1/2}}{(I')^{\de_i/4-1/2}(I'_*)^{\de_j/4-1/2}}(I)^{\de_i/4-1/2}(I_*)^{\de_j/4-1/2}\si_{ij}|u|d\om  dI^\prime dI^\prime_*,
\end{align}
which, after combining with \eqref{si} and the fact that
\begin{align}\label{changeM}
	\frac{MM_*}{(I)^{\de_i/2-1}(I_*)^{\de_j/2-1}}=\frac{M'M'_*}{(I')^{\de_i/2-1}(I'_*)^{\de_j/2-1}},
\end{align} yields
\begin{align}\label{Gk1}
	k^1_{ij}(v,I,v_*,I_*)&\leq Ce^{-\frac{m_i|v|^2}{4}-\frac{m_j|v_*|^2}{4}-\frac{I}{2}-\frac{I_*}{2}}\notag\\
	&\qquad\times\int_{\S^2\times (\R_+)^2} (I)^{\de_i/4-1/2}(I_*)^{\de_j/4-1/2}\sqrt{|u|^2-\frac{2\De I}{\mu_{ij}}}\frac{(I')^{\de_i/2-1}(I'_*)^{\de_j/2-1}}{E_{ij}^{(\eta+\de_i+\de_j)/2}}d\om  dI^\prime dI^\prime_*\notag\\
	&\leq Ce^{-\frac{m_i|v|^2}{8}-\frac{m_j|v_*|^2}{8}-\frac{I}{4}-\frac{I_*}{4}}(I)^{\de_i/4-1/2}(I_*)^{\de_j/4-1/2}\notag\\
	&\qquad\times\int_{\S^2\times (\R_+)^2} e^{-\frac{m_i|v'|^2}{8}-\frac{m_j|v'_*|^2}{8}-\frac{I'}{4}-\frac{I'_*}{4}}\frac{(I')^{\de_i/2-1}(I'_*)^{\de_j/2-1}}{E_{ij}^{(\eta+\de_i+\de_j-1)/2}}d\om  dI^\prime dI^\prime_*.
\end{align}
By $E_{ij}=E'_{ij}$, one gets that
$$
E_{ij}^{(\de_i+\de_j)/2}\geq (I_*)^{1/2}(I^\prime)^{\de_i/2-1/4}( I^\prime_*)^{\de_j/2-1/4}.
$$
Gathering the previous two inequalities, we to obtain that
\begin{align}\label{k1pp}
	k^1_{ij}(v,I,v_*,I_*)&\leq Ce^{-\frac{m_i|v|^2}{8}-\frac{m_j|v_*|^2}{8}-\frac{I}{4}-\frac{I_*}{4}}(I)^{\de_i/4-1/2}(I_*)^{\de_j/4-1/2}E_{ij}^{(1-\eta)/2}\notag\\
	&\qquad\times\int_{ (\R_+)^2} e^{-\frac{m_i|v'|^2}{8}-\frac{m_j|v'_*|^2}{8}-\frac{I'}{4}-\frac{I'_*}{4}}\frac{(I')^{\de_i/2-1}(I'_*)^{\de_j/2-1}}{(I_*)^{1/2}(I^\prime)^{\de_i/2-1/4}( I^\prime_*)^{\de_j/2-1/4}} dI^\prime dI^\prime_*\notag\\
	&\leq Ce^{-\frac{m_i|v|^2}{16}-\frac{m_j|v_*|^2}{16}-\frac{I}{8}-\frac{I_*}{8}}(I_*)^{\de_j/4-1}.
\end{align}
Using $e^\frac{\de|v-v_*|^2}{64}\leq e^\frac{m_i|v|^2+m_j|v_*|^2}{32}$, note that $\de\leq\min_{1\leq i\leq n}m_i$, for the case $i,j>p$, it holds by \eqref{k1pp} that
\begin{align}\label{kpp1}
	\int_{\R^3\times\R_+}& k^1_{ij}(v,I,v_*,I_*)\frac{w_{i\be}(v,I)}{w_{j\be}(v_*,I_*)}e^\frac{\de|v-v_*|^2}{64}(1+I_*)^{\frac{1}{8}} dv_*dI_*\notag\\
	\leq&C\int_{\R^3\times\R_+}e^{-\frac{m_i|v|^2}{32}-\frac{m_j|v_*|^2}{32}-\frac{I}{16}-\frac{I_*}{16}}(I_*)^{\de_j/4-1}(1+I_*)^{\frac{1}{8}}dv_*dI_* \leq \frac{C}{1+|v|+I^\frac{1}{8}},
\end{align}
which shows \eqref{kpp}.  We continue with the estimation of $k^3_{ij}(Z,Z_*)$.

Let us derive a general representation for $k^3_{ij}$ for all $1\leq i,j\leq n$. Recall from \eqref{Kij} that
\begin{align*}
	K^3_{ij}f&=\int_{\CZ_j\times \CZ_i\times \CZ_j} \frac{ (M_{j*}M'_iM'_{j*})^{1/2}}{(II')^{\de_i/4-1/2}(I_*I'_*)^{\de_j/4-1/2}}\frac{f'_i}{(M'_i)^{1/2}}W_{ij}(Z,Z_*|Z',Z'_*)dZ_*dZ^\prime dZ^\prime_*.
\end{align*}
We interchange $Z_*$ and $Z'$ to get that
\begin{align*}
	K^3_{ij}f&=\int_{\CZ_i\times \CZ_j\times \CZ_j} \frac{ (M'_{j}M_{i*}M'_{j*})^{1/2}}{(II_*)^{\de_i/4-1/2}(I'I'_*)^{\de_j/4-1/2}}\frac{f_{i*}}{(M_{i*})^{1/2}}W_{ij}(Z,Z'|Z_*,Z'_*)dZ_*dZ^\prime dZ^\prime_*,
\end{align*}
which yields
\begin{align}\label{Gk3}
	k^3_{ij}=\int_{(\CZ_j)^2} \frac{ (M'_{j}M'_{j*})^{1/2}}{(II_*)^{\de_i/4-1/2}(I'I'_*)^{\de_j/4-1/2}}W_{ij}(Z,Z'|Z_*,Z'_*)dZ^\prime dZ^\prime_*.
\end{align}
Let us introduce
\begin{align}\label{notationk3}
	&u=v-v_*,\qquad u'=v'-v'_*,\qquad \tilde{u}=v-v',\qquad u_*=v_*-v'_*,\notag\\
	&\text{and}\qquad\De_* I=(I_*-I)\chi_{\{i>p\}}+(I'_*-I')\chi_{\{j>p\}},\qquad \zeta_+=(v_*-v^\prime)\cdot\frac{u}{|u|}.
\end{align}
 Using the definition of $W_{ij}$ in \eqref{DefW} one has that
\begin{align}\label{Wk3}
	W_{ij}(Z,Z'|Z_*,Z'_*)=&\frac{(m_i+m_j)^2}{m_j^2}(I)^{\de_i/2-1}(I')^{\de_j/2-1}\si_{ij}{(|\tilde{u}|,\frac{\tilde{u}\cdot u_*}{|\tilde{u}||u_*|},I,I',I_*,I^\prime_*)}\frac{|\tilde{u}|}{|u_*||u|}\notag\\
	&\times\bm{\de}_3(\frac{m_i}{m_j}u+u')\bm{\de}_1(\zeta_+-\frac{m_i-m_j}{2m_j}|u|-\frac{\De_* I}{m_i|u|}).
\end{align}
Notice the momentum and energy conservation induced by \eqref{Wk3}
\begin{align}\label{conservationk3}
	m_iv+m_jv'&=m_iv_*+m_jv'_*,\qquad\text{and}\notag\\
	m_i\frac{|v|^2}{2}+m_j\frac{|v'|^2}{2}+I\chi_{\{i>p\}}+I'\chi_{\{j>p\}}&=m_i\frac{|v_*|^2}{2}+m_j\frac{|v'_*|^2}{2}+I_*\chi_{\{i>p\}}+I'_*\chi_{\{j>p\}}.
\end{align}
Let us particularize to $i,j>p$. We combine the above identity with the fact that 
\begin{align}\label{cv}
	dv'dv'_*=du'd(v'_*-v)=du'd\zeta_+d\om_*,\qquad \om_*:=v'_*-v-\zeta_+ n,\qquad n:=\frac{u}{|u|},
\end{align} to get that
\begin{align}\label{k3}
	k^3_{ij}=&k^3_{ij}(v,I,v_*,I_*)\notag\\
	=&\int_{\R^3\times\R_+\times(\R^3)^{\perp n}\times (\R_+)^2} (M'_{j}M'_{j*})^{1/2}\frac{(m_i+m_j)^2}{m_j^2}\big(\frac{I}{I_*}\big)^{\de_i/4-1/2}\big(\frac{I'}{I'_*}\big)^{\de_j/4-1/2}\si_{ij}{(|\tilde{u}|,\frac{u'\cdot u_*}{|u'||u_*|},I,I',I_*,I^\prime_*)}\notag\\
	&\qquad\times \frac{|\tilde{u}|}{|u_*||u|} \bm{\de}_3(\frac{m_i}{m_j}u+u')\bm{\de}_1(\zeta_+-\frac{m_i-m_j}{2m_j}|u|-\frac{\De_* I}{m_i|u|})du^\prime d\zeta_+ d\om_* dI^\prime dI^\prime_*\notag\\
	=&\frac{(m_i+m_j)^2}{m_j^2}\big(\frac{I}{I_*}\big)^{\de_i/4-1/2}\notag\\
	&\qquad\times \int_{(\R^3)^{\perp n}\times (\R_+)^2} (M'_{j}M'_{j*})^{1/2}\big(\frac{I'}{I'_*}\big)^{\de_j/4-1/2}\si_{ij}{(|\tilde{u}|,\frac{u'\cdot u_*}{|u'||u_*|},I,I',I_*,I^\prime_*)}\frac{|\tilde{u}|}{|u_*||u|} d\om_* dI^\prime dI^\prime_*,
\end{align} 
where $(\R^3)^{\perp n}$ denotes the plane generated by $\om_*\cdot n=0$.
It follows from \eqref{si} that
 \begin{align*}
 	\si_{ij}{(|\tilde{u}|,\frac{u'\cdot u_*}{|u'||u_*|},I,I',I_*,I^\prime_*)}\leq C\sqrt{|\tilde{u}|^2-\frac{2\De_* I}{\mu_{ij}}}\frac{(I_* )^{\de_i/2-1}(I^\prime_*)^{\de_j/2-1}}{|\tilde{u}|\widetilde{E}_{ij}^{(\eta+\de_i\chi_{\{i>p\}}+\de_j\chi_{\{j>p\}})/2}},
 \end{align*} 
where $\widetilde{E}_{ij}=\frac{\mu_{ij}}{2}|\tilde{u}|^2+I\chi_{\{i>p\}}+I'\chi_{\{j>p\}}.$
Substituting  the above equality into \eqref{k3} we obtain that
\begin{align}\label{2boundk3}
	k^3_{ij}(v,I,v_*,I_*)\leq&C(II_*)^{\de_i/4-1/2}\frac{1}{|u|}\notag\\
	&\times \int_{(\R^3)^{\perp n}\times (\R_+)^2} e^{-\frac{m_j|v'|^2}{4}-\frac{m_j|v'_*|^2}{4}-\frac{I'}{2}-\frac{I'_*}{2}}(I'I^\prime_*)^{\de_j/2-1}\frac{1}{\widetilde{E}_{ij}^{(\eta+\de_i+\de_j)/2}} d\om_* dI^\prime dI^\prime_*.
\end{align}
Using the pointwise estimates
\begin{align}\label{controle}
e^\frac{\de|v-v_*|^2}{64}\leq e^\frac{\de m^2_j|v'-v'_*|^2}{64 m^2_i}\leq e^\frac{m_j|v'|^2+m_j|v'_*|^2}{32},
\end{align}
and
\begin{equation*}
\frac{w^2_{i\be}(v,I)}{w^2_{i\be}(v_*,I_*)}\leq C (1+\big|m_i|v|^2-m_i|v_*|^2+I-I_*\big|^2)^\be\leq C e^\frac{m_j|v'|^2+m_j|v'_*|^2+2I'+2I'_*}{32},
\end{equation*}
one further obtains that
\begin{align*}
	k^3_{ij}(v,I,&v_*,I_*)\frac{w_{i\be}(v,I)}{w_{i\be}(v_*,I_*)}e^\frac{\de|v-v_*|^2}{64}(1+I_*)^{\frac{1}{8}} \leq C(II_*)^{\de_i/4-1/2}\frac{1}{|u|} \\
	&\times\int_{(\R^3)^{\perp n}\times (\R_+)^2} e^{-\frac{m_j|v'|^2}{8}-\frac{m_j|v'_*|^2}{8}-\frac{I'}{4}-\frac{I'_*}{4}}(I'I^\prime_*)^{\de_j/2-1}(1+I_*)^{1/8}\frac{1}{\widetilde{E}_{ij}^{(\eta+\de_i+\de_j)/2}} d\om_* dI^\prime dI^\prime_*.
\end{align*}
Define
\begin{align}\label{notationsk3}
V=\frac{v+v_*}{2},\qquad V_{\shortparallel}:=(V\cdot n)n,\qquad V:=V_\shortparallel+V_\perp,\qquad \cos\tilde{\theta}=n\cdot\frac{v}{|v|}\,,
\end{align} 
and $\zeta_-:=\frac{\De_* I}{m_i|u|}-\frac{m_i-m_j}{2m_j}|u|$.  A direct calculation gives that
\begin{align}\label{computev}
	\frac{1}{8}(|v^\prime|^2+|v^\prime_*|^2)&=\frac{1}{8}(|v_*-\zeta_- n+\om_*|^2+|v-\zeta_+ n+\om_*|^2)\notag\\
	&= \frac{1}{4}\big(|V-\frac{\De_* I}{m_i|u|} n+\om_*|^2+\frac{m_i^2}{4m_j^2}|u|\big) \notag\\
	&=\frac{1}{4}\left( |V_\perp+\om_*|^2+\frac{(|u|-2|v|\cos\tilde{\theta}+2\frac{\De_* I}{m_i|u|})^2}{4}+\frac{m_i^2|u|^2}{4m_j^2} \right).
\end{align} 
The combination of the previouus two relations shows that
\begin{align}\label{3boundk3}
	k^3_{ij}&(v,I,v_*,I_*)\frac{w_{i\be}(v,I)}{w_{i\be}(v_*,I_*)}e^\frac{\de|v-v_*|^2}{64}(1+I_*)^{\frac{1}{8}}\notag\\ \leq&C(II_*)^{\de_i/4-1/2}\frac{1}{|u|} \int_{(\R^3)^{\perp n}\times (\R_+)^2}\exp[-\frac{m_j}{4}\big( |V_\perp+\om_*|^2+\frac{(|u|-2|v|\cos\tilde{\theta}+2\frac{\De_* I}{m_i|u|})^2}{4}+\frac{m_i^2|u|^2}{4m_j^2} \big)]\notag\\
	&\qquad\qquad\qquad\qquad\qquad\times e^{-\frac{I'}{4}-\frac{I'_*}{4}}(I'I^\prime_*)^{\de_j/2-1}(1+I_*)^{1/8}\frac{1}{\widetilde{E}_{ij}^{(\eta+\de_i+\de_j)/2}} d\om_* dI^\prime dI^\prime_*.
\end{align}
Then, we use the fact that \begin{align*}
	\widetilde{E}_{ij}=\frac{\mu_{ij}}{2}|\tilde{u}|^2+I+I'\geq C|\tilde{u}|\sqrt{\frac{\mu_{ij}}{2}|\tilde{u}|^2+I+I'-I_*-I^\prime_*}=C|\tilde{u}||u_*|\geq C|\om_*|^2
\end{align*}
to get that
\begin{align}\label{1boundk3}
	k^3_{ij}&(v,I,v_*,I_*)\frac{w_{i\be}(v,I)}{w_{i\be}(v_*,I_*)}e^\frac{\de|v-v_*|^2}{64}(1+I_*)^{\frac{1}{8}}\notag\\ \leq&C(II_*)^{\de_i/4-1/2}\frac{1}{|u|} \int_{(\R^3)^{\perp n}\times (\R_+)^2}\exp[-\frac{m_j}{4}\big( |V_\perp+\om_*|^2+\frac{(|u|-2|v|\cos\tilde{\theta}+2\frac{\De_* I}{m_i|u|})^2}{4}+\frac{m_i^2|u|^2}{4m_j^2} \big)]\frac{1}{|\om_*|^\eta}\notag\\
	&\qquad\qquad\qquad\qquad\qquad\times e^{-\frac{I'}{4}-\frac{I'_*}{4}}(I'I^\prime_*)^{\de_j/2-1}(1+I_*)^{1/8}\frac{1}{\widetilde{E}_{ij}^{(\de_i+\de_j)/2}} d\om_* dI^\prime dI^\prime_*\notag\\ \leq&\frac{C}{|u|} \int_{ (\R_+)^2}\exp[-\frac{m_j}{4}\big( |V_\perp+\om_*|^2+\frac{(|u|-2|v|\cos\tilde{\theta}+2\frac{\De_* I}{m_i|u|})^2}{4}+\frac{m_i^2|u|^2}{4m_j^2} \big)]\notag\\
	&\qquad\qquad\qquad\qquad\qquad\qquad\times e^{-\frac{I'}{4}-\frac{I'_*}{4}}[(II_*)^{\de_i/4-1/2}(I'I^\prime_*)^{\de_j/2-1}\frac{(1+I_*)^{1/8}}{\widetilde{E}_{ij}^{(\de_i+\de_j)/2}}]  dI^\prime dI^\prime_*.
\end{align}
Apply the identity $\widetilde{E}_{ij}=\frac{\mu_{ij}}{2}|\tilde{u}|^2+I+I'=\frac{\mu_{ij}}{2}|u_*|^2+I_*+I'_*$ to get
\begin{align}\label{boundtildeE}
	\frac{1}{\widetilde{E}_{ij}^{(\de_i+\de_j)/2}}\leq& \frac{1}{(I)^{\de_i/4-1/2}(I_*)^{\de_i/4}(I'I'_*)^{\de_j/4+1/4}}\chi_{\{I\leq 1,I_*\leq 1\}}+\frac{1}{(I_*)^{(\de_i+\de_j)/2}}\chi_{\{I\leq 1,I_*\geq 1\}}\notag\\
	&\qquad+\frac{1}{(I)^{(\de_i+\de_j)/2}}\chi_{\{I\geq 1,I_*\leq 1\}}+\frac{1}{(I)^{(\de_i+\de_j-1)/4}(I_*)^{(\de_i+\de_j+1)/4}}\chi_{\{I\geq 1,I_*\geq 1\}},
\end{align}
which implies that
\begin{align}\label{boundI}
	\int_{ (\R_+)^3}&e^{-\frac{I'}{4}-\frac{I'_*}{4}}(II_*)^{\de_i/4-1/2}(I'I^\prime_*)^{\de_j/2-1}\frac{(1+I_*)^{1/8}}{\widetilde{E}_{ij}^{(\de_i+\de_j)/2}}dI^\prime dI^\prime_*dI_*\notag\\
	\leq&\int_{ (\R_+)^3}e^{-\frac{I'}{4}-\frac{I'_*}{4}}\big\{ \frac{(I'I'_*)^{\de_j/4-3/4}}{(I_*)^{1/2}}\chi_{\{I\leq 1,I_*\leq 1\}}+\frac{(I)^{\de_i/4-1/2}(I'I^\prime_*)^{\de_j/2-1}}{(I_*)^{\de_i/4+\de_j/2+1/2}}\chi_{\{I\leq 1,I_*\geq 1\}}\notag\\
	&+\frac{(I_*)^{\de_i/4-1/2}(I'I^\prime_*)^{\de_j/2-1}(1+I_*)^{1/8}}{(I)^{\de_i/4+\de_j/2+1/2}}\chi_{\{I\geq 1,I_*\leq 1\}}+\frac{(I'I^\prime_*)^{\de_j/2-1}(1+I_*)^{1/8}}{(I)^{\de_j/4+1/4}(I_*)^{\de_j/4+3/4}}\chi_{\{I\geq 1,I_*\geq 1\}}\big\}dI^\prime dI^\prime_*dI_*\notag\\
	\leq& \frac{C}{1+I^{1/2}}.
\end{align}
Now we can bound the integral \eqref{1boundk3}. It follows from \eqref{1boundk3} and \eqref{boundI} that
\begin{align}\label{4boundk3}
	\int_{\R^3\times (\R_+)^3}&k^3_{ij}(v,I,v_*,I_*)\frac{w_{i\be}(v,I)}{w_{i\be}(v_*,I_*)}e^\frac{\de|v-v_*|^2}{64}(1+I_*)^{\frac{1}{8}}dI^\prime dI^\prime_*dI_*dv_*\notag\\ \leq&\int_{ \R^3\times (\R_+)^3}\frac{C}{|u|} e^{-\frac{m_j}{4}\frac{m_i^2|u|^2}{4m_j^2} }e^{-\frac{I'}{4}-\frac{I'_*}{4}}[(II_*)^{\de_i/4-1/2}(I'I^\prime_*)^{\de_j/2-1}\frac{(1+I_*)^{1/8}}{\widetilde{E}_{ij}^{(\de_i+\de_j)/2}}]  dI^\prime dI^\prime_*dI_*dv_*\notag\\ \leq&\frac{C}{1+I^{1/2}}.
\end{align}
We seek to obtain the decay in $v$ as well. For $|u|\geq |v|$ we have from \eqref{1boundk3} and \eqref{boundI} that
\begin{align}\label{5boundk3}
	\int_{\R^3\times (\R_+)^3}&k^3_{ij}(v,I,v_*,I_*)\frac{w_{i\be}(v,I)}{w_{i\be}(v_*,I_*)}e^\frac{\de|v-v_*|^2}{64}(1+I_*)^{\frac{1}{8}}dI^\prime dI^\prime_*dI_*dv_*\notag\\ \leq&\int_{ \R^3\times (\R_+)^3}\frac{C}{|v|} e^{-\frac{m_j}{4}\frac{m_i^2|u|^2}{4m_j^2} }e^{-\frac{I'}{4}-\frac{I'_*}{4}}[(II_*)^{\de_i/4-1/2}(I'I^\prime_*)^{\de_j/2-1}\frac{(1+I_*)^{1/8}}{\widetilde{E}_{ij}^{(\de_i+\de_j)/2}}]  dI^\prime dI^\prime_*dI_*dv_*\notag\\ \leq&\frac{C}{|v|(1+I^{1/2})}.
\end{align}
For $|u|\leq |v|$ we set $r=|u|$ and use $v\cdot(v-v_*)=r|v|\cos\theta$ to get that
\begin{align}\label{6boundk3}
	\int_{\R^3\times (\R_+)^3}&k^3_{ij}(v,I,v_*,I_*)\frac{w_{i\be}(v,I)}{w_{i\be}(v_*,I_*)}e^\frac{\de|v-v_*|^2}{64}(1+I_*)^{\frac{1}{8}}dI^\prime dI^\prime_*dI_*dv_*\notag\\ \leq&\int_{  (\R_+)^3}\Big[\int_0^{|v|}\int_0^\pi r\exp[-m_j[(r-2|v|\cos\theta+2\frac{\De_* I}{m_ir})^2+\frac{m_i^2r^2}{m_j}]/16]\sin\theta d\theta dr\Big]\notag\\
	&\qquad\times e^{-\frac{I'}{4}-\frac{I'_*}{4}}[(II_*)^{\de_i/4-1/2}(I'I^\prime_*)^{\de_j/2-1}\frac{(1+I_*)^{1/8}}{\widetilde{E}_{ij}^{(\de_i+\de_j)/2}}]  dI^\prime dI^\prime_*dI_*\notag\\ \leq&\frac{C}{|v|(1+I^{1/2})}.
\end{align}
Collecting the above three estimates it follows that
\begin{align}\label{kpp3}
	&\int_{\R^3\times (\R_+)^3}k^3_{ij}(v,I,v_*,I_*)\frac{w_{i\be}(v,I)}{w_{i\be}(v_*,I_*)}e^\frac{\de|v-v_*|^2}{64}(1+I_*)^{\frac{1}{8}}dI^\prime dI^\prime_*dI_*dv_*\leq\frac{C}{1+|v|+I^{1/2}},
\end{align}
which implies \eqref{k3pp}.  We now focus on the last case $k^2_{ij}$ by first assuming $m_i\neq m_j$. Recall \eqref{DefKij} and exchange $Z_*$ and $Z'_*$ to obtain that
\begin{align*}
	K^2_{ij}f&=\int_{\CZ_j\times \CZ_i\times \CZ_j} \frac{ (M_{j*}M'_iM'_{j*})^{1/2}}{(II')^{\de_i/4-1/2}(I_*I'_*)^{\de_j/4-1/2}}\frac{f'_{j*}}{(M'_{j*})^{1/2}}W_{ij}(Z,Z_*|Z',Z'_*)dZ_*dZ^\prime dZ^\prime_*\notag\\
	&=\int_{\CZ_j\times \CZ_i\times \CZ_j} \frac{ (M_{j*}M'_iM'_{j*})^{1/2}}{(II')^{\de_i/4-1/2}(I_*I'_*)^{\de_j/4-1/2}}\frac{f_{j*}}{(M_{j*})^{1/2}}W_{ij}(Z,Z'_*|Z',Z_*)dZ_*dZ^\prime dZ^\prime_*,
\end{align*}
which yields
\begin{align}\label{rek2ij}
	k^2_{ij}=k^2_{ij}(Z,Z_*)&=\int_{\CZ_i\times \CZ_j} \frac{ (M'_iM'_{j*})^{1/2}}{(II')^{\de_i/4-1/2}(I_*I'_*)^{\de_j/4-1/2}}W_{ij}(Z,Z'_*|Z',Z_*)dZ^\prime dZ^\prime_*.
\end{align}
Denote
\begin{align*}
\dis u_{ij}&=\frac{m_iv-m_jv_*}{m_i-m_j}, \qquad\dis u'_{ij}=\frac{m_iv'-m_jv'_*}{m_i-m_j},\qquad u=v-v_*,\qquad u'=v'-v'_*,\\
\hat{u}&=v'_* - v,\quad \bar{u}=v_*-v', \qquad  \text{and} \qquad \widehat{\De} I=(I'-I)\chi_{\{i>p\}}+(I_*-I'_*)\chi_{\{j>p\}}.
\end{align*}
From the definition of $W_{ij}$ in \eqref{DefW}, it holds that
\begin{align}\label{Wk2}
	W_{ij}(Z,Z'_*|Z',Z_*)=\frac{(m_i+m_j)^2}{(m_i-m_j)^2}&(I)^{\de_i/2-1}(I'_*)^{\de_j/2-1}\si_{ij}{(|\hat{u}|,\frac{\hat{u}\cdot \bar{u}}{|\hat{u}||\bar{u}|},I,I'_*,I',I_*)}\frac{|\hat{u}|}{|u'||\bar{u}|}\notag\\
	&\times\bm{\de}_3(u_{ij}-u'_{ij})\bm{\de}_1(|u'|-\sqrt{|u|^2-2\frac{m_j-m_i}{m_im_j}\widehat{\De} I}).
\end{align}
Such form holds for all $1\leq i,j\leq n$.  In the following we treat the case of Poly-Poly interactions $i,j>p$.  First, note that 
\begin{align*}
	dv'dv'_*=|u'|^2d|u'|du'_{ij}d\om, 
\end{align*}
where $\om:=\frac{u'}{|u'|}$.  It follows that
\begin{align}\label{rek2}
	k^2_{ij}(Z,Z_*)&=\frac{(m_i+m_j)^2}{(m_i-m_j)^2}\int_{\S^2\times (\R_+)^2}  (M'_iM'_{j*})^{1/2}\left(\frac{I}{I'}\right)^{\de_i/4-1/2}\left(\frac{I'_*}{I_*}\right)^{\de_j/4-1/2}\notag\\
	&\qquad\qquad\qquad\qquad\times\si_{ij}{(|\hat{u}|,\frac{\hat{u}\cdot \bar{u}}{|\hat{u}||\bar{u}|},I,I'_*,I',I_*)}\frac{|\hat{u}||u'|}{|\bar{u}|}d\om dI'dI'_*.
\end{align}
Moreover, define $\hat{\om}:=\frac{\bar{u}}{|\bar{u}|}$ and $\bar{\om}:=\frac{\hat{u}}{|\hat{u}|}$. Similar calculations as in \eqref{rek2} show that
\begin{align}
	k^2_{ij}(Z,Z_*)&=\int_{\S^2\times (\R_+)^2}  (M'_iM'_{j*})^{1/2}\left(\frac{I}{I'}\right)^{\de_i/4-1/2}\left(\frac{I'_*}{I_*}\right)^{\de_j/4-1/2}\si_{ij}{(|\hat{u}|,\frac{\hat{u}\cdot \bar{u}}{|\hat{u}||\bar{u}|},I,I'_*,I',I_*)}\hat{u}d\hat{\om} dI'dI'_*\label{rek21}\\
	&=\int_{\S^2\times (\R_+)^2}  (M'_iM'_{j*})^{1/2}\left(\frac{I'}{I}\right)^{\de_i/4-1/2}\left(\frac{I_*}{I'_*}\right)^{\de_j/4-1/2}\si_{ij}{(|\bar{u}|,\frac{\hat{u}\cdot \bar{u}}{|\hat{u}||\bar{u}|},I',I_*,I,I'_*)}\bar{u}d\bar{\om} dI'dI'_*\label{rek22}.
\end{align}
Notice that, from \eqref{si}, it yields
\begin{align*}
	\si_{ij}{(|\hat{u}|,\frac{\hat{u}\cdot \bar{u}}{|\hat{u}||\bar{u}|},I,I'_*,I',I_*)}\frac{|\hat{u}||u'|}{|\bar{u}|}\leq C\frac{(I^\prime )^{\de_i/2-1}(I_*)^{\de_j/2-1}}{\widehat{E}_{ij}^{(\eta+\de_i\chi_{\{i>p\}}+\de_j\chi_{\{j>p\}})/2}}|u'|
\end{align*}
with
\begin{equation*}
\widehat{E}_{ij}=\frac{\mu_{ij}}{2}|\hat{u}|^2+I\chi_{\{i>p\}}+I'_*\chi_{\{j>p\}}=\frac{\mu_{ij}}{2}|\bar{u}|^2+I'\chi_{\{i>p\}}+I_*\chi_{\{j>p\}}.
\end{equation*}
In order to control the $M'_iM'_{j*}$ in the integral, a straightforward calculation shows that
\begin{align}\label{boundv}
	m_i|v'|^2+m_j|v'_*|^2&=(\sqrt{m_i}-\sqrt{m_j})^2(|u'_{ij}|^2+\frac{m_im_j}{(m_i-m_j)^2}|u'|^2)+2\sqrt{m_im_j}(u'_{ij}-\frac{\sqrt{m_im_j}}{(m_i-m_j)}u')^2\notag\\
	&\geq (\sqrt{m_i}-\sqrt{m_j})^2(|u'_{ij}|^2+\frac{m_im_j}{(m_i-m_j)^2}|u'|^2)\notag\\
	&= (\sqrt{m_i}-\sqrt{m_j})^2|u_{ij}|^2+\frac{m_im_j}{(\sqrt{m_i}+\sqrt{m_j})^2}(|u|^2-2\frac{m_j-m_i}{m_im_j}\widehat{\De} I).
\end{align}
In order to bound $k^2_{ij}$, we split the integral into three cases. First, if $\min\{\hat{u},\bar{u}\}\geq 1$, then $|\hat{u}||\bar{u}|\geq \max\{\hat{u},\bar{u}\}\geq |u|$. One can further get
\begin{align*}
		\si_{ij}{(|\hat{u}|,\frac{\hat{u}\cdot \bar{u}}{|\hat{u}||\bar{u}|},I,I'_*,I',I_*)}\frac{|\hat{u}|}{|u'||\bar{u}|}\leq C\frac{(I^\prime )^{\de_i/2-1}(I_*)^{\de_j/2-1}}{\widehat{E}_{ij}^{(\de_i+\de_j)/2}}\frac{1}{(|\hat{u}||\bar{u}|)^{\eta/2}}\frac{1}{|u'|}\leq C\frac{(I^\prime )^{\de_i/2-1}(I_*)^{\de_j/2-1}}{\widehat{E}_{ij}^{(\de_i+\de_j)/2}}\frac{1}{|u|^{\eta/2}|u'|},
\end{align*}
which, together with \eqref{rek2} and \eqref{boundv}, yields
\begin{align}\label{boundk21}
	k^2_{ij}(Z,Z_*)&\leq C\int_{\S^2\times (\R_+)^2}  (M'_iM'_{j*})^{1/2}\left(II'\right)^{\de_i/4-1/2}\left(I'_*I_*\right)^{\de_j/4-1/2}{\widehat{E}_{ij}^{(\de_i+\de_j)/2}}\frac{1}{|u|^{\eta/2}}\frac{1}{|u'|}d\om dI'dI'_*\notag\\
	&\leq C\int_{\S^2\times (\R_+)^2} \exp[-\frac{(\sqrt{m_i}-\sqrt{m_j})^2}{8}|u_{ij}|^2-\frac{m_im_j}{8(\sqrt{m_i}+\sqrt{m_j})^2}(|u|^2-2\frac{m_j-m_i}{m_im_j}\widehat{\De} I)] \notag\\
	&\qquad\qquad\times e^{-\frac{I'}{2}-\frac{I'_*}{2}}(I)^{\de_i/4-1/2}(I')^{\de_i/2-1}(I_*)^{\de_j/4-1/2}(I'_*)^{\de_j/2-1}\frac{1}{\widehat{E}_{ij}^{(\de_i+\de_j)/2}}\frac{1}{|u|^{\eta/2}}d\om dI'dI'_*.
\end{align}
Second, if $\min\{\hat{u},\bar{u}\}<1$ and $|\hat{u}|=\max\{\hat{u},\bar{u}\}$, then $|\hat{u}|=\max\{\hat{u},\bar{u}\}\geq |u|$ and $|\bar{u}|<1$. It holds that
\begin{align*}
	\si_{ij}{(|\hat{u}|,\frac{\hat{u}\cdot \bar{u}}{|\hat{u}||\bar{u}|},I,I'_*,I',I_*)}|\hat{u}|\leq C\frac{(I^\prime )^{\de_i/2-1}(I_*)^{\de_j/2-1}}{\widehat{E}_{ij}^{(\de_i+\de_j)/2}}\frac{1}{|\hat{u}|^{\eta}}|\bar{u}|\leq C\frac{(I^\prime )^{\de_i/2-1}(I_*)^{\de_j/2-1}}{\widehat{E}_{ij}^{(\de_i+\de_j)/2}}\frac{1}{|u|^{\eta}}.
\end{align*}
We combine the above estimate with \eqref{rek21} and \eqref{boundv} to get that
\begin{align}\label{boundk22}
	k^2_{ij}(Z,Z_*)&\leq C\int_{\S^2\times (\R_+)^2} \exp[-\frac{(\sqrt{m_i}-\sqrt{m_j})^2}{8}|u_{ij}|^2-\frac{m_im_j}{8(\sqrt{m_i}+\sqrt{m_j})^2}(|u|^2-2\frac{m_j-m_i}{m_im_j}\widehat{\De} I)] \notag\\
	&\qquad\qquad\times e^{-\frac{I'}{2}-\frac{I'_*}{2}}(I)^{\de_i/4-1/2}(I')^{\de_i/2-1}(I_*)^{\de_j/4-1/2}(I'_*)^{\de_j/2-1}\frac{1}{\widehat{E}_{ij}^{(\de_i+\de_j)/2}}\frac{1}{|u|^{\eta}}d\hat{\om} dI'dI'_*.
\end{align}
Third, if $\min\{\hat{u},\bar{u}\}<1$ and $|\bar{u}|=\max\{\hat{u},\bar{u}\}$, then $|\bar{u}|=\max\{\hat{u},\bar{u}\}\geq |u|$ and $|\hat{u}|<1$. It holds that
\begin{align*}
	\si_{ij}{(|\bar{u}|,\frac{\hat{u}\cdot \bar{u}}{|\hat{u}||\bar{u}|},I',I_*,I,I'_*)}\bar{u}\leq C\frac{(I )^{\de_i/2-1}(I'_*)^{\de_j/2-1}}{\widehat{E}_{ij}^{(\de_i+\de_j)/2}}\frac{1}{|\bar{u}|^{\eta}}|\hat{u}|\leq C\frac{(I )^{\de_i/2-1}(I'_*)^{\de_j/2-1}}{\widehat{E}_{ij}^{(\de_i+\de_j)/2}}\frac{1}{|u|^{\eta}},
\end{align*}
which, combined with \eqref{rek22} and \eqref{boundv}, it shows that
\begin{align}\label{boundk23}
	k^2_{ij}(Z,Z_*)&\leq C\int_{\S^2\times (\R_+)^2} \exp[-\frac{(\sqrt{m_i}-\sqrt{m_j})^2}{8}|u_{ij}|^2-\frac{m_im_j}{8(\sqrt{m_i}+\sqrt{m_j})^2}(|u|^2-2\frac{m_j-m_i}{m_im_j}(I'-I+I_*-I'_*))] \notag\\
	&\qquad\qquad\times e^{-\frac{I'}{2}-\frac{I'_*}{2}}(I')^{\de_i/2-1}(I)^{\de_i/4-1/2}(I_*)^{\de_j/4-1/2}(I'_*)^{\de_j/2-1}\frac{1}{\widehat{E}_{ij}^{(\de_i+\de_j)/2}}\frac{1}{|u|^{\eta}}d\bar{\om} dI'dI'_*.
\end{align}
Combining \eqref{boundk21}, \eqref{boundk22} and \eqref{boundk23}, and using the fact that
\begin{align*}
	\frac{1}{\widehat{E}_{ij}^{(\de_i+\de_j)/2}}\leq C\Psi(I,I_*,I',I'_*),
\end{align*}
where
\begin{align}\label{DefPsi}
	\Psi(I,I_*,I',I'_*):=(I)^{-\psi_1}(I_*)^{-\psi_2}(I')^{-\psi_3}(I'_*)^{-\psi_4},
\end{align}
and $\psi_i>0$ can be arbitrarily chosen such that $\psi_1+\psi_2+\psi_3+\psi_4=(\de_i+\de_j)/2,$ we obtain that
\begin{align*}
	&k^2_{ij}(Z,Z_*)\leq C\int_{(\R_+)^2} \exp[-\frac{(\sqrt{m_i}-\sqrt{m_j})^2}{8}|u_{ij}|^2-\frac{m_im_j}{8(\sqrt{m_i}+\sqrt{m_j})^2}(|u|^2-2\frac{m_j-m_i}{m_im_j}(I'-I+I_*-I'_*))] \notag\\
	&\qquad\qquad\qquad\quad \times (1+\frac{1}{|u|^{\eta}}) e^{-\frac{I'}{2}-\frac{I'_*}{2}}(I')^{\de_i/2-1}(I)^{\de_i/4-1/2}(I_*)^{\de_j/4-1/2}(I'_*)^{\de_j/2-1}\Psi(I,I_*,I',I'_*)dI'dI'_*,
\end{align*}
which, together with
\begin{align*}
\frac{w^2_{i\be}(v,I)}{w^2_{j\be}(v_*,I_*)}\leq C (1+\big|m_i|v|^2-m_j|v_*|^2+2I-2I_*\big|^2)^\be\leq C e^\frac{m_i|v'|^2+m_j|v'_*|^2+2I'+2I'_*}{32},
\end{align*}
gives that
\begin{align}\label{boundk2}
	k^2_{ij}&(v,I,v_*,I_*)\frac{w_{i\be}(v,I)}{w_{j\be}(v_*,I_*)}e^\frac{\de|m_iv-m_jv_*|^2}{64}(1+I_*)^{1/8}\notag\\
	\leq& C\int_{(\R_+)^2} \exp[-\frac{(\sqrt{m_i}-\sqrt{m_j})^2}{16}|u_{ij}|^2-\frac{m_im_j}{16(\sqrt{m_i}+\sqrt{m_j})^2}(|u|^2-2\frac{m_j-m_i}{m_im_j}(I'-I+I_*-I'_*))]\notag\\
	&\quad\times (1+\frac{1}{|u|^{\eta}}) e^{-\frac{I'}{4}-\frac{I'_*}{4}}(I')^{\de_i/2-1}(I)^{\de_i/4-1/2}(I_*)^{\de_j/4-1/2}(I'_*)^{\de_j/2-1}\Psi(I,I_*,I',I'_*)(1+I_*)^{1/8}dI'dI'_*.
\end{align}
We consider \eqref{boundk2} in six possible cases. 

\smallskip
\noindent{\it Case 1. } $I\geq \ka(|v|^2+1)$  where is a small constant $\ka>0$ that will be determined after the last case is computed, and either $m_j>m_i$, $I\geq I_*$ or $m_j<m_i$, $I\leq I_*$.  In this case we choose $\psi_i$ in \eqref{DefPsi} such that
\begin{align*}
	\Psi=&\frac{1}{(I)^{\de_i/2+\de_j/2}}\chi_{\{I_*\leq1\}}+\frac{1}{(I)^{\de_i/2+\de_j/4-3/4}(I_*)^{\de_j/4+3/4}}\chi_{\{I_*>1\}}.
\end{align*} Notice $\Psi$ is independent of $u$. It follows from \eqref{boundk2} and our choice of $\Psi$ that
\begin{align}\label{k2case1}
	\int_{\R^3\times \R_+}&k^2_{ij}(v,I,v_*,I_*)\frac{w_{i\be}(v,I)}{w_{j\be}(v_*,I_*)}e^\frac{\de|m_iv-m_jv_*|^2}{64}(1+I_*)^{1/8}dv_*dI_*\notag\\
	\leq& C\int_{\R^3\times (\R_+)^3} \exp[-\frac{m_im_j}{16(\sqrt{m_i}+\sqrt{m_j})^2}|u|^2](1+\frac{1}{|u|^{\eta}}) \notag\\
	&\qquad\times e^{-\frac{I'}{4}-\frac{I'_*}{4}}(I')^{\de_i/2-1}(I)^{\de_i/4-1/2}(I_*)^{\de_j/4-1/2}(I'_*)^{\de_j/2-1}\Psi(I,I_*,I',I'_*)(1+I_*)^{1/8}dudI_*dI'dI'_*\notag\\
	\leq& C\int_{(\R_+)^3} e^{-\frac{I'}{2}-\frac{I'_*}{2}}(I')^{\de_i/2-1}(I)^{\de_i/4-1/2}(I_*)^{\de_j/4-1/2}(I'_*)^{\de_j/2-1}\Psi(I,I_*,I',I'_*)(1+I_*)^{1/8}dI_*dI'dI'_*.
\end{align}
The combination of the above two inequalities gives
\begin{align*}
	\int_{\R^3\times \R_+}&k^2_{ij}(v,I,v_*,I_*)\frac{w_{i\be}(v,I)}{w_{j\be}(v_*,I_*)}e^\frac{\de|m_iv-m_jv_*|^2}{64}(1+I_*)^{1/8}dv_*dI_*\notag\\
	\leq& C\int_{(\R_+)^3} e^{-\frac{I'}{2}-\frac{I'_*}{2}}(I')^{\de_i/2-1}(I'_*)^{\de_j/2-1}(I)^{-\de_i/4-\de_j/2-1/2}(I_*)^{\de_j/4-1/2}(1+I_*)^{1/8}\chi_{\{I_*\leq1\}}dI_*dI'dI'_*\notag\\
	&+C\int_{(\R_+)^3} e^{-\frac{I'}{2}-\frac{I'_*}{2}}(I')^{\de_i/2-1}(I'_*)^{\de_j/2-1}(I)^{-\de_i/4-\de_j/4+1/4}(I_*)^{-5/4}(1+I_*)^{1/8}\chi_{\{I_*\geq1\}}dI_*dI'dI'_*\notag\\
	\leq& \frac{C}{I^{1/2}}\leq \frac{C}{1+|v|+I^{1/4}}.
\end{align*}
\noindent{\it Case 2. } $I\geq \ka(|v|^2+1)$, $m_j<m_i$, $I\geq I_*$, or $I_*\geq \ka(|v|^2+1)$, $m_j<m_i$, $I\geq I_*$.  Observing that $\{I_*\geq \ka(|v|^2+1),I\geq I_*\}$ implies $I\geq \ka(|v|^2+1)$, we only need to prove the case $I\geq \ka(|v|^2+1)$, $m_j<m_i$, $I\geq I_*$.
It follows from \eqref{boundk2} and $\dis \frac{1}{|u|^\eta}\leq C_\eta+\frac{C_\eta}{|u|}$ that
\begin{align}\label{k2case2}
	k^2_{ij}&(v,I,v_*,I_*)\frac{w_{i\be}(v,I)}{w_{j\be}(v_*,I_*)}e^\frac{\de|m_iv-m_jv_*|^2}{64}(1+I_*)^{1/8}\notag\\
	\leq& C\int_{(\R_+)^2} \exp[-\frac{(\sqrt{m_i}-\sqrt{m_j})^2}{16}|u_{ij}|^2-\frac{m_im_j}{16(\sqrt{m_i}+\sqrt{m_j})^2}(|u|^2-2\frac{m_j-m_i}{m_im_j}(I'-I+I_*-I'_*))]\notag\\
	&\times (1+\frac{1}{|u|}) e^{-\frac{I'}{4}-\frac{I'_*}{4}}(I')^{\de_i/2-1}(I)^{\de_i/4-1/2}(I_*)^{\de_j/4-1/2}(I'_*)^{\de_j/2-1}\Psi(I,I_*,I',I'_*)(1+I_*)^{1/8}dI'dI'_*\notag\\
	\leq& C\int_{(\R_+)^2} \exp[-\frac{(\sqrt{m_i}-\sqrt{m_j})^2}{16}|u_{ij}|^2]\notag\\
	&\qquad\quad\times  e^{-\frac{I'}{4}-\frac{I'_*}{4}}(I')^{\de_i/2-1}(I)^{\de_i/4-1/2}(I_*)^{\de_j/4-1/2}(I'_*)^{\de_j/2-1}\Psi(I,I_*,I',I'_*)(1+I_*)^{1/8}dI'dI'_*\notag\\
	&+ C\int_{(\R_+)^2} \exp[-\frac{m_im_j}{16(\sqrt{m_i}+\sqrt{m_j})^2}(|u|^2-2\frac{m_j-m_i}{m_im_j}(I'-I+I_*-I'_*))]\frac{1}{|u|}\notag\\
	&\qquad\quad\times  e^{-\frac{I'}{4}-\frac{I'_*}{4}}(I')^{\de_i/2-1}(I)^{\de_i/4-1/2}(I_*)^{\de_j/4-1/2}(I'_*)^{\de_j/2-1}\Psi(I,I_*,I',I'_*)(1+I_*)^{1/8}dI'dI'_*\notag\\
	=&J_1+J_2.
\end{align}
Integrating over $(v_*,I_*)$ one gets that
\begin{align}\label{J1J2}
	\int_{\R^3\times \R_+}&k^2_{ij}(v,I,v_*,I_*)\frac{w_{i\be}(v,I)}{w_{j\be}(v_*,I_*)}e^\frac{\de|m_iv-m_jv_*|^2}{64}(1+I_*)^{1/8}dv_*dI_*\notag\\
	\leq&\int_{\R^3\times \R_+}J_1dv_*dI_*+\int_{\R^3\times \R_+}J_2dv_*dI_*.
\end{align}
For $J_1$, we choose $\Psi=\frac{1}{(I)^{\de_i/2+\de_j/2}}$ so that
\begin{align}\label{J1}
	\int_{\R^3\times \R_+}&J_1dv_*dI_*\notag\\
	\leq&C\int_{\R^3\times(\R_+)^3} \exp[-\frac{(\sqrt{m_i}-\sqrt{m_j})^2}{16}|u_{ij}|^2]e^{-\frac{I'}{4}-\frac{I'_*}{4}}\notag\\
	&\qquad\times  (I')^{\de_i/2-1}(I)^{\de_i/4-1/2}(I_*)^{\de_j/4-1/2}(I'_*)^{\de_j/2-1}\Psi(I,I_*,I',I'_*)(1+I_*)^{1/8}du_{ij}dI_*dI'dI'_* \notag\\
	\leq& C\int_{(\R_+)^2} e^{-\frac{I'}{4}-\frac{I'_*}{4}}(I')^{\de_i/2-1}(I)^{-\de_i/4-\de_j/2-1/2}[\int_0^I(I_*)^{\de_j/4-1/2}(1+I_*)^{1/8}dI_*](I'_*)^{\de_j/2-1}dI'dI'_*\notag\\
	\leq& CI^{-\de_i/4-\de_j/4+1/8} \leq \frac{C}{1+|v|+I^{1/4}}.
\end{align}
In the latter inequality we used the condition $I\geq \ka(|v|^2+1)\geq \ka$.  Regarding the $J_2$ term, we apply the change to spherical coordinates to obtain that
\begin{align}\label{J2case2}
	\int_{\R^3\times \R_+}&J_2dv_*dI_*\notag\\
	\leq&C\int_{(\R_+)^3}\int^\infty_{\sqrt{2\frac{m_j-m_i}{m_im_j}(I'-I+I_*-I'_*)}}\exp[-\frac{m_im_j}{16(\sqrt{m_i}+\sqrt{m_j})^2}(r^2-2\frac{m_j-m_i}{m_im_j}(I'-I+I_*-I'_*))]rdr\notag\\
	&\qquad\times  e^{-\frac{I'}{4}-\frac{I'_*}{4}}(I')^{\de_i/2-1}(I)^{\de_i/4-1/2}(I_*)^{\de_j/4-1/2}(I'_*)^{\de_j/2-1}\Psi(I,I_*,I',I'_*,r)(1+I_*)^{1/8}dI_*dI'dI'_*,
\end{align}
which, together with the same $\Psi$ as in $J_1$ and the fact that
\begin{align*}
&\int^\infty_{\sqrt{2\frac{m_j-m_i}{m_im_j}(I'-I+I_*-I'_*)}}\exp[-\frac{m_im_j}{16(\sqrt{m_i}+\sqrt{m_j})^2}(r^2-2\frac{m_j-m_i}{m_im_j}(I'-I+I_*-I'_*))]rdr\notag\\
&\leq\int^\infty_{0}\exp[-\frac{m_im_j}{16(\sqrt{m_i}+\sqrt{m_j})^2}(r^2-2\frac{m_j-m_i}{m_im_j}(I'-I+I_*-I'_*))]d(r^2-2\frac{m_j-m_i}{m_im_j}(I'-I+I_*-I'_*)) \leq C,
\end{align*}
yields that
\begin{align*}
	\int_{\R^3\times \R_+}&J_2dv_*dI_*\notag\\
	\leq&C\int_{(\R_+)^3}e^{-\frac{I'}{4}-\frac{I'_*}{4}}(I')^{\de_i/2-1}(I)^{\de_i/4-1/2}(I_*)^{\de_j/4-1/2}(I'_*)^{\de_j/2-1}\frac{(1+I_*)^{1/8}}{(I)^{\de_i/2+\de_j/2}}dI_*dI'dI'_*.
\end{align*}
Then, similar arguments as in $J_1$ show that
\begin{align}\label{J2}
	&\int_{\R^3\times \R_+}J_2dv_*dI_*\leq \frac{C}{1+|v|+I^{1/4}}.
\end{align}
It follows from \eqref{J1J2}, \eqref{J1} and \eqref{J2} that
\begin{align*}
	&\int_{\R^3\times \R_+}k^2_{ij}(v,I,v_*,I_*)\frac{w_{i\be}(v,I)}{w_{j\be}(v_*,I_*)}e^\frac{\de|m_iv-m_jv_*|^2}{64}(1+I_*)^{1/8}dv_*dI_*\leq \frac{C}{1+|v|+I^{1/4}}.
\end{align*}

\medskip
\noindent{\it Case 3. } $I\geq \ka(|v|^2+1)$, $m_j>m_i$, $I\leq I_*$.  In this case notice that \eqref{k2case2} still holds with the same $J_1$ and $J_2$.  We choose $\Psi=\frac{1}{(I_*)^{\de_i/2+\de_j/2}}$ for $J_1$ to get that
\begin{align*}
	\int_{\R^3\times \R_+}&J_1dv_*dI_*\notag\\
	\leq&C\int_{\R^3\times(\R_+)^3} \exp[-\frac{(\sqrt{m_i}-\sqrt{m_j})^2}{16}|u_{ij}|^2]e^{-\frac{I'}{4}-\frac{I'_*}{4}}\notag\\
	&\qquad\times  (I')^{\de_i/2-1}(I)^{\de_i/4-1/2}(I_*)^{\de_j/4-1/2}(I'_*)^{\de_j/2-1}\Psi(I,I_*,I',I'_*)(1+I_*)^{1/8}du_{ij}dI_*dI'dI'_* \notag\\
	\leq& C\int_{(\R_+)^2} e^{-\frac{I'}{4}-\frac{I'_*}{4}}(I')^{\de_i/2-1}(I)^{\de_i/4-1/2}[\int_I^\infty (I_*)^{-\de_i/2-\de_j/4-1/2}(1+I_*)^{1/8}dI_*](I'_*)^{\de_j/2-1}dI'dI'_*\notag\\
	\leq& CI^{-\de_i/4-\de_j/4+1/8} \leq \frac{C}{1+|v|+I^{1/4}}.
\end{align*} 
Moreover, similar calculation as in \eqref{J2case2} and the same $\Psi$ give
\begin{align*}
	\int_{\R^3\times \R_+}&J_2dv_*dI_*\notag\\
	\leq&C\int_{(\R_+)^3}e^{-\frac{I'}{4}-\frac{I'_*}{4}}(I')^{\de_i/2-1}(I)^{\de_i/4-1/2}(I_*)^{\de_j/4-1/2}(I'_*)^{\de_j/2-1}\frac{(1+I_*)^{1/8}}{(I_*)^{\de_i/2+\de_j/2}}dI_*dI'dI'_*\notag\\
	\leq& \frac{C}{1+|v|+I^{1/4}}.
\end{align*}
Hence, \eqref{k2pp} holds after using the previous two estimates.

\medskip
\noindent{\it Case 4. } $I_*\geq \ka(|v|^2+1)$, and either $m_j>m_i$, $I\geq I_*$ or $m_j<m_i$, $I\leq I_*$.  Notice that \eqref{k2case1} still holds. We let
\begin{align*}
	\Psi=&\frac{1}{(I)^{\de_i/4}(I_*)^{\de_i/4+\de_j/2}}\chi_{\{I>1\}}+\frac{1}{(I_*)^{\de_i/2+\de_j/2}}\chi_{\{I\leq1\}},
\end{align*}
to obtain that
\begin{align*}
	\int_{\R^3\times \R_+}&k^2_{ij}(v,I,v_*,I_*)\frac{w_{i\be}(v,I)}{w_{j\be}(v_*,I_*)}e^\frac{\de|m_iv-m_jv_*|^2}{64}(1+I_*)^{1/8}dv_*dI_*\notag\\
	& \leq  C\int_{\R_+} (I)^{-1/2}(I_*)^{-\de_i/4-\de_j/4-1/2}(1+I_*)^{1/8}\chi_{\{I\geq1\}}dI_*\notag\\
	&\quad + C\int_{\R_+}(I)^{\de_i/4-1/2}(I_*)^{-\de_i/2-\de_j/4-1/2}(1+I_*)^{1/8}\chi_{\{I\leq1\}}dI_*\notag\\
	&\leq \frac{C}{1+|v|+I^{1/4}}.
\end{align*}

\medskip
\noindent{\it Case 5. } $I_*\geq \ka(|v|^2+1)$, $m_j>m_i$, $I\leq I_*$.  We use \eqref{k2case2} again. For $J_1$ set
\begin{align*}
	\Psi=&\frac{1}{(I_*)^{\de_i/2+\de_j/2}}\chi_{\{I\leq1\}}+\frac{1}{(I)^{\de_i/4}(I_*)^{\de_i/4+\de_j/2}}\chi_{\{I>1\}}
\end{align*}
to get that
\begin{align}\label{boundJ1}
	\int_{\R^3\times \R_+}&J_1dv_*dI_*\notag\\
	\leq& C(I)^{\de_i/4-1/2}\chi_{\{I\leq1\}}[\int_\ka^\infty (I_*)^{-\de_i/4-\de_j/4-1/2}(1+I_*)^{1/8}dI_*](1+|v|^2)^{-\de_i/4}\notag\\
	&+C(I)^{-1/2}\chi_{\{I\geq1\}}[\int_\ka^\infty (I_*)^{-\de_i/8-\de_j/4-1/2}(1+I_*)^{1/8}dI_*](1+|v|^2)^{-\de_i/8}\notag\\
	\leq& \frac{C}{1+|v|^{1/2}+I^{1/4}}.
\end{align}
Similarly, using the same $\Psi$ and the argument in \eqref{J2case2}, it holds that
\begin{align*}
	&\int_{\R^3\times \R_+}J_2dv_*dI_*
	\leq\frac{C}{1+|v|^{1/2}+I^{1/4}}.
\end{align*}
Then one gets \eqref{k2pp} by the previous two estimates.

\smallskip
\noindent{\it Case 6. } $I\leq \ka(|v|^2+1)$ and $I_*\leq \ka(|v|^2+1)$.  In this last case, we determine the small number $\ka>0$ introduce in {\it Case 1}. A direct calculation shows that
\begin{align}\label{reuij}
	&|u_{ij}|^2=\big|\frac{m_i}{m_i-m_j}v-\frac{m_j}{m_i-m_j}v_*\big|^2\notag\\
	=&\frac{(m_i+m_j)^2}{4(m_i-m_j)^2}|(v+v_*)^\perp|^2+\frac{(m_i+m_j)^2}{4(m_i-m_j)^2}\frac{\big||v|^2-|v_*|^2\big|^2}{|v-v_*|^2}+\frac{m_i+m_j}{2(m_i-m_j)}(|v|^2-|v_*|^2)+\frac{1}{4}|v-v_*|^2,
\end{align}
where $(v+v_*)^\perp:=v+v_*-(v+v_*)\cdot\frac{v-v_*}{|v-v_*|}\, \frac{v-v_*}{|v-v_*|}.$ 
We combine the above identity and \eqref{boundk2} to obtain
\begin{align}\label{boundk2c6}
	k^2_{ij}&(v,I,v_*,I_*)\frac{w_{i\be}(v,I)}{w_{j\be}(v_*,I_*)}e^\frac{\de|m_iv-m_jv_*|^2}{64}(1+I_*)^{1/8}\notag\\
	\leq& C\int_{(\R_+)^2} \exp[-\frac{(m_i+m_j)^2}{64(m_i-m_j)^2}\frac{\big||v|^2-|v_*|^2\big|^2}{|v-v_*|^2}-\frac{m_i+m_j}{32(m_i-m_j)}(|v|^2-|v_*|^2)-\frac{1}{64}|v-v_*|^2\notag\\
	&\qquad\qquad\qquad-\frac{m_im_j}{16(\sqrt{m_i}+\sqrt{m_j})^2}(|u|^2-2\frac{m_j-m_i}{m_im_j}(I'-I+I_*-I'_*))]\notag\\
	&\times (1+\frac{1}{|u|^{\eta}}) e^{-\frac{I'}{4}-\frac{I'_*}{4}}(I')^{\de_i/2-1}(I)^{\de_i/4-1/2}(I_*)^{\de_j/4-1/2}(I'_*)^{\de_j/2-1}\Psi(I,I_*,I',I'_*)(1+I_*)^{1/8}dI'dI'_*.
\end{align}
To simplify the notations introduce
\begin{align}\label{defAB}
A_1=\frac{m_i+m_j}{8(m_i-m_j)},\qquad A_2=\frac{1}{8},\qquad B=\frac{m_im_j}{16(\sqrt{m_i}+\sqrt{m_j})^2}.
\end{align}
It is straightforward to get that 
\begin{align*}
	\frac{m_im_j}{16(\sqrt{m_i}+\sqrt{m_j})^2}\times2\frac{m_j-m_i}{m_im_j}(I'-I'_*)-\frac{1}{8}(I'+I'_*)=-\frac{\sqrt{m_i}}{4(\sqrt{m_i}+\sqrt{m_j})}I'-\frac{\sqrt{m_j}}{4(\sqrt{m_i}+\sqrt{m_j})}I'_*,
 \end{align*}
which, together with \eqref{boundk2c6} and our assumption $\max\{I,I_*\}\leq \ka(|v|^2+1)$, yields
\begin{align}\label{boundk2c61}
	k^2_{ij}&(v,I,v_*,I_*)\frac{w_{i\be}(v,I)}{w_{j\be}(v_*,I_*)}e^\frac{\de|m_iv-m_jv_*|^2}{64}(1+I_*)^{1/8}\notag\\
	\leq& C\int_{(\R_+)^2} \exp[-A_1^2\frac{\big||v|^2-|v_*|^2\big|^2}{|v-v_*|^2}-2A_1A_2(|v|^2-|v_*|^2)-(A_2^2+B)|v-v_*|^2+\frac{\ka}{4}|v|^2]\notag\\
	&\times (1+\frac{1}{|u|^{\eta}}) e^{-\frac{I'}{8}-\frac{I'_*}{8}}(I')^{\de_i/2-1}(I)^{\de_i/4-1/2}(I_*)^{\de_j/4-1/2}(I'_*)^{\de_j/2-1}\Psi(I,I_*,I',I'_*)(1+I_*)^{1/8}dI'dI'_*.
\end{align}
Note that 
\begin{align}\label{rev}
|v|^2=\frac{1}{4}|(v+v_*)^\perp|^2+\frac{1}{4}\frac{\big||v|^2-|v_*|^2\big|^2}{|v-v_*|^2}+\frac{1}{2}(|v|^2-|v_*|^2)+\frac{1}{4}|v-v_*|^2.	
\end{align}
Then, we have that
\begin{align*}
	k^2_{ij}&(v,I,v_*,I_*)\frac{w_{i\be}(v,I)}{w_{j\be}(v_*,I_*)}e^\frac{\de|m_iv-m_jv_*|^2}{64}(1+I_*)^{1/8}\notag\\
	\leq& C\int_{(\R_+)^2} \exp[-(A_1^2-\ka)\frac{\big||v|^2-|v_*|^2\big|^2}{|v-v_*|^2}-2(A_1A_2-\ka)(|v|^2-|v_*|^2)-(A_2^2+B-\ka)|v-v_*|^2]\notag\\
	&\times (1+\frac{1}{|u|^{\eta}}) e^{-\frac{I'}{8}-\frac{I'_*}{8}}(I')^{\de_i/2-1}(I)^{\de_i/4-1/2}(I_*)^{\de_j/4-1/2}(I'_*)^{\de_j/2-1}\Psi(I,I_*,I',I'_*)(1+I_*)^{1/8}dI'dI'_*.
\end{align*}
Letting $\dis M=1+\frac{B}{2A^2_2}$ and choosing $\ka>0$ such that
\begin{align*}
&(A_1^2-\ka)\frac{\big||v|^2-|v_*|^2\big|^2}{|v-v_*|^2}+2(A_1A_2-\ka)(|v|^2-|v_*|^2)+(A_2^2+B-\ka)|v-v_*|^2\notag\\
\geq& (A_1^2-\frac{A_1^2}{M}-2\ka)\frac{\big||v|^2-|v_*|^2\big|^2}{|v-v_*|^2}+(A_2^2-MA^2_2+B-2\ka)|v-v_*|^2\notag\\
=& ((1-(1+32B)^{-1})A^2_1-2\ka)\frac{\big||v|^2-|v_*|^2\big|^2}{|v-v_*|^2}+(\frac{B}{2}-\ka)|v-v_*|^2\notag\\
\geq& \frac{1}{2}(1-(1+32B)^{-1})A^2_1\frac{\big||v|^2-|v_*|^2\big|^2}{|v-v_*|^2}+\frac{B}{4}|v-v_*|^2,
\end{align*}
it follows that there exists a constant 
$$c=\frac{1}{2}\min\{(1-(1+32B)^{-1})A^2_1,\frac{B}{2}\}>0,$$
 such that 
\begin{multline}\label{positivity}
	(A_1^2-\ka)\frac{\big||v|^2-|v_*|^2\big|^2}{|v-v_*|^2}+2(A_1A_2-\ka)(|v|^2-|v_*|^2)+(A_2^2+B-\ka)|v-v_*|^2\\
	>c\frac{\big||v|^2-|v_*|^2\big|^2}{|v-v_*|^2}+c|v-v_*|^2,
\end{multline}
which gives that
\begin{align}\label{boundk2pp1}
	k^2_{ij}&(v,I,v_*,I_*)\frac{w_{i\be}(v,I)}{w_{j\be}(v_*,I_*)}e^\frac{\de|m_iv-m_jv_*|^2}{64}(1+I_*)^{1/8}\notag
	\leq C\int_{(\R_+)^2} \exp[-c\frac{\big||v|^2-|v_*|^2\big|^2}{|v-v_*|^2}-c|v-v_*|^2]\notag\\
	&\times (1+\frac{1}{|u|^{\eta}}) e^{-\frac{I'}{8}-\frac{I'_*}{8}}(I')^{\de_i/2-1}(I)^{\de_i/4-1/2}(I_*)^{\de_j/4-1/2}(I'_*)^{\de_j/2-1}\Psi(I,I_*,I',I'_*)(1+I_*)^{1/8}dI'dI'_*.
\end{align} 
Integrate in $(v_*, I_*)$ and use the following estimate (see Lemma 7 in \cite{Guo}) 
\begin{align}\label{propertyv}
\int_{\R^3} &\exp[-c\frac{\big||v|^2-|v_*|^2\big|^2}{|v-v_*|^2}-c|v-v_*|^2](1+\frac{1}{|u|^{\eta}})dv_*\notag \leq \int_{\R^3} \exp[-c\frac{|v\cdot u|^2}{|u|^2}-c|u|^2](1+\frac{1}{|u|^{\eta}})dv_*\notag\\
	\leq &\int_{\R^3} \exp[-c'|u|^2-c'|v\cdot u|](1+\frac{1}{|u|^{\eta}})du
	\notag\\
	\leq &C\min\Big\{1,\int_{\R^2}(1+\frac{1}{|u_\perp|^{\eta}})\exp[-c'|u_
	\perp|^2]du_\perp\int_{\R} \exp[-c'|v||u_\para|]du_\para\Big\}\quad(u_\para=u\cdot \frac{v}{|v|}\, \frac{v}{|v|},\ u_\perp=u-u_\para)\notag\\
	\leq &\frac{C}{1+|v|},
\end{align}
to obtain that
\begin{align*}
	\int_{\R^3\times \R_+}&k^2_{ij}(v,I,v_*,I_*)\frac{w_{i\be}(v,I)}{w_{j\be}(v_*,I_*)}e^\frac{\de|m_iv-m_jv_*|^2}{64}(1+I_*)^{1/8}dv_*dI_*\notag\\
	\leq& \frac{C}{1+|v|}\int_{(\R_+)^3} e^{-\frac{I'}{8}-\frac{I'_*}{8}}(I')^{\de_i/2-1}(I)^{\de_i/4-1/2}(I_*)^{\de_j/4-1/2}(I'_*)^{\de_j/2-1}\Psi(I,I_*,I',I'_*)(1+I_*)^{1/8}dI_*dI'dI'_*.
\end{align*}
Then, we set
\begin{align*}
\Psi(I,I_*,I',I'_*)=& \frac{1}{(I_*)^{1/2}(I')^{\de_j/2+1/4}(I'_*)^{\de_j/2+1/4}}\chi_{\{I\leq 1,I_*\leq 1\}}+\frac{1}{(I_*)^{(\de_i+\de_j)/2}}\chi_{\{I\leq 1,I_*> 1\}}\notag\\
&\qquad+\frac{1}{(I)^{(\de_i+\de_j)/2}}\chi_{\{I> 1,I_*\leq 1\}}+\frac{1}{(I)^{\de_i/2-1/4}(I_*)^{\de_j/2+1/4}}\chi_{\{I> 1,I_*> 1\}}
\end{align*}
to get that
\begin{align*}
	&\int_{\R^3\times \R_+}k^2_{ij}(v,I,v_*,I_*)\frac{w_{i\be}(v,I)}{w_{j\be}(v_*,I_*)}e^\frac{\de|m_iv-m_jv_*|^2}{64}(1+I_*)^{1/8}dv_*dI_*
	\leq \frac{C}{1+|v|+I^{1/2}}.
\end{align*}
Consequenlty, \eqref{k2pp} is proved by collecting the above six cases for $m_i\neq m_j$.  Now, for the situation $m_i=m_j$ we use the structure given in \eqref{k3} to obtain that
\begin{align*}
	k^2_{ij}&(v,I,v_*,I_*)\notag\\
	=&C\big(\frac{I}{I'}\big)^{\de_i/4-1/2}\int_{(\R^3)^{\perp n}\times (\R_+)^2} (M'_{i}M'_{j*})^{1/2}\big(\frac{I'_*}{I_*}\big)^{\de_j/4-1/2}\si_{ij}{(|\hat{u}|,\frac{\hat{u}\cdot \bar{u}}{|\hat{u}||\bar{u}|},I,I'_*,I',I_*)}\frac{|\hat{u}|}{|\bar{u}||u|} d\om' dI^\prime dI^\prime_*,
\end{align*} 
with
\begin{align}\label{k2velocity}
	v'=v+\om'-\frac{\De I}{m_i|u|},\qquad\text{and}\qquad v'_*=v_*+\om'-\frac{\De I}{m_i|u|}.
\end{align}
This yields that
\begin{align}\label{k2mimj}
	k^2_{ij}&(v,I,v_*,I_*)\frac{w_{i\be}(v,I)}{w_{j\be}(v_*,I_*)}e^\frac{\de|v-v_*|^2}{64}(1+I_*)^{\frac{1}{8}}\notag\\ \leq&C(I)^{\de_i/4-1/2}(I_*)^{\de_j/4-1/2}\frac{1}{|u|} \int_{(\R^3)^{\perp n}\times (\R_+)^2}\exp[-\frac{m_j}{4}\big( |V_\perp+\om'|^2+\frac{(|u|-2|v|\cos\tilde{\theta}+2\frac{\De I}{m_i|u|})^2}{4}+\frac{|u|^2}{4} \big)]\notag\\
	&\qquad\qquad\qquad\qquad\qquad\times e^{-\frac{I'}{4}-\frac{I'_*}{4}}(I')^{\de_i/2-1}(I^\prime_*)^{\de_j/2-1}(1+I_*)^{1/8}\frac{1}{\widehat{E}_{ij}^{(\eta+\de_i+\de_j)/2}} d\om' dI^\prime dI^\prime_*,
\end{align}
by similar arguments leading to \eqref{3boundk3}. Then we use the fact
\begin{align*}
	\frac{1}{\widehat{E}_{ij}^{(\de_i+\de_j)/2}}\leq& \frac{1}{(I)^{\de_i/4-1/2}(I_*)^{\de_j/4}(I')^{\de_i/4+1/4}(I'_*)^{\de_j/4+1/4}}\chi_{\{I\leq 1,I_*\leq 1\}}+\frac{1}{(I_*)^{(\de_i+\de_j)/2}}\chi_{\{I\leq 1,I_*\geq 1\}}\notag\\
	&\qquad+\frac{1}{(I)^{(\de_i+\de_j)/2}}\chi_{\{I\geq 1,I_*\leq 1\}}+\frac{1}{(I)^{(\de_i+\de_j-1)/4}(I_*)^{(\de_i+\de_j+1)/4}}\chi_{\{I\geq 1,I_*\geq 1\}},
\end{align*}
and perform similar calculations as in \eqref{boundI}, \eqref{4boundk3}, \eqref{5boundk3}, \eqref{6boundk3} and \eqref{kpp3}, together with
$$
\de|m_iv-m_jv_*|^2\leq \de \max_{1\leq i\leq n}{m_i^2}|v-v_*|^2,\ m_i=m_j,
$$
to obtain \eqref{k2pp}.  In this way the proof of Lemma \ref{leK} is completed.  
\end{proof}

In case of Mono-Poly collision, the internal energy variables reduce to only $I_*$ and $I'_*$. However, such reduction implies non-integrability. For example, on the estimates for $k^2_{ij}$ in the proof of Lemma \ref{leK}, if we let $I$, $I'$ vanish, the polynomial part for $I_*$ and $I'_*$ in \eqref{boundk2} becomes
$$
(I_*)^{\de_j/4-1/2}(I'_*)^{\de_j/2-1}\Psi(I_*,I'_*)(1+I_*)^{1/8}=(I_*)^{\de_j/4-1/2}(I'_*)^{\de_j/2-1}\frac{(1+I_*)^{1/8}}{(I_*)^{\psi_2}(I'_*)^{\de_j/2-\psi_2}},\quad 0\leq\psi_2\leq\frac{\de_j}{2}.
$$
Observe that for the whole range of $\psi_2$, the integrand on $I_*$ is not integrable for large $I_*$. Thus, in this case, we need a variation of the approach to handle this problem.
\begin{lemma}[Mono-Poly Collision]\label{leKmp}
	For $i\leq p$ and $j> p$, the linearized operator defined in \eqref{DefKij} can be written as
	\begin{align*}
			&\dis 	K_{ij}f=	K_{ij}f(v)=\int_{\R^3\times \R_+}(k^2_{ij}(v,v_*,I_*)-k^1_{ij}(v,v_*,I_*))f_j(v_*,I_*)dv_*dI_*+\int_{\R^3}k^3_{ij}(v,v_*)f_i(v_*)dv_*,
	\end{align*}
	with
	\begin{align}
		&\int_{\R^3\times\R_+} k^1_{ij}(v,v_*,I_*)\frac{w_{i\be}(v)}{w_{j\be}(v_*,I_*)}e^\frac{\de|v-v_*|^2}{64}(1+I_*)^{\frac{1}{8}} dv_*dI_*\leq \frac{C}{1+|v|^{\frac{1}{2}}},\label{kmp}\\
		&\int_{\R^3\times\R_+} k^2_{ij}(v,v_*,I_*)\frac{w_{i\be}(v)}{w_{j\be}(v_*,I_*)}e^\frac{\de|m_iv-m_jv_*|^2}{64}(1+I_*)^{\frac{1}{8}} dv_*dI_*\leq \frac{C}{1+|v|^{\frac{1}{2}}},\label{k2mp}\\
		&\int_{\R^3} k^3_{ij}(v,v_*)\frac{w_{i\be}(v)}{w_{i\be}(v_*)}e^\frac{\de|v-v_*|^2}{64} dv_*\leq \frac{C}{1+|v|^{\frac{1}{2}}},\label{k3mp}
	\end{align}
	where $\dis \de=\min_{1\leq i,j\leq n}\{m_i,\frac{(\sqrt{m_i}-\sqrt{m_j})^2}{16(m_i-m_j)^2}\}.$
\end{lemma}
\begin{proof}
	We prove \eqref{kmp} first. Noticing that \eqref{changeofvariables} and \eqref{changeM} still hold, then similar arguments as in \eqref{Gk1} and \eqref{k1pp} give
	\begin{align}\label{k1mp}
		k^1_{ij}(v,v_*,I_*)&=\int_{\S^2\times \R_+} \frac{ (M'_iM'_{j*})^{1/2}}{(I'_*)^{\de_j/4-1/2}}(I_*)^{\de_j/4-1/2}\si_{ij}|u|d\om   dI^\prime_*\notag\\
		&\leq Ce^{-\frac{m_i|v|^2}{8}-\frac{m_j|v_*|^2}{8}-\frac{I_*}{4}}(I_*)^{\de_j/4-1/2}E_{ij}^{(1-\eta)/2}\int_{ \R_+} e^{-\frac{m_i|v'|^2}{8}-\frac{m_j|v'_*|^2}{8}-\frac{I'_*}{4}}\frac{(I'_*)^{\de_j/2-1}}{(I_*)^{1/2}( I^\prime_*)^{\de_j/2-1/2}} dI^\prime_*\notag\\
		&\leq Ce^{-\frac{m_i|v|^2}{16}-\frac{m_j|v_*|^2}{16}-\frac{I_*}{8}}(I_*)^{\de_j/4-1},
	\end{align}
	which yields \eqref{k1mp} by taking integral in $v_*$ and $I_*$.  Let us consider now the term $k^3_{ij}$ recalling that $\widetilde{E}_{ij}=\frac{\mu_{ij}}{2}|\tilde{u}|^2+I$ for Mono-Poly collisions.  Using the notations in \eqref{notationk3}, similar calculations made in \eqref{cv}, \eqref{k3} and \eqref{2boundk3} show that
	\begin{align*}
		k^3_{ij}(v,v_*)\leq&C\frac{1}{|u|}\int_{(\R^3)^{\perp n}\times (\R_+)^2} e^{-\frac{m_j|v'|^2}{4}-\frac{m_j|v'_*|^2}{4}-\frac{I'}{2}-\frac{I'_*}{2}}(I'I^\prime_*)^{\de_j/2-1}\frac{1}{\widetilde{E}_{ij}^{(\eta+\de_j)/2}} d\om_* dI^\prime dI^\prime_*,
	\end{align*}
	which, together with \eqref{controle}, \eqref{computev}, the fact that $\frac{w_{i\be}(v)}{w_{i\be}(v_*)}\leq C e^\frac{\de m_i|u|^2}{64m_j},$ and the arguments of \eqref{3boundk3} and \eqref{1boundk3}, yields
	\begin{align*}
		k^3_{ij}&(v,v_*)\frac{w_{i\be}(v)}{w_{i\be}(v_*)}e^\frac{\de|v-v_*|^2}{64}\notag\\
		 \leq&\frac{C}{|u|} \int_{ (\R_+)^2}\exp[-\frac{m_j}{4}\big(\frac{(|u|-2|v|\cos\tilde{\theta}+2\frac{\De_* I}{m_i|u|})^2}{4}+\frac{m_i^2|u|^2}{4m_j^2} \big)]e^{-\frac{I'}{4}-\frac{I'_*}{4}}\frac{(I'I^\prime_*)^{\de_j/2-1}}{\widetilde{E}_{ij}^{\de_j/2}} dI^\prime dI^\prime_*.
	\end{align*}
	Integrate in $v_*$ and use that 
$$\frac{(I'I^\prime_*)^{\de_j/2-1}}{\widetilde{E}_{ij}^{\de_j/2}}\leq C\frac{(I'I^\prime_*)^{\de_j/2-1}}{(I'I^\prime_*)^{\de_j/4}}\leq C(I'I^\prime_*)^{\de_j/4-1}
$$
	to get that
	\begin{align*}
		&\int_{\R^3}k^3_{ij}(v,v_*)\frac{w_{i\be}(v)}{w_{i\be}(v_*)}e^\frac{\de|v-v_*|^2}{64}dv_* \leq\frac{C}{1+|v|} \int_{ (\R_+)^2}e^{-\frac{I'}{4}-\frac{I'_*}{4}}(I'I^\prime_*)^{\de_j/4-1} dI^\prime dI^\prime_*\leq\frac{C}{1+|v|},
	\end{align*}
which gives \eqref{k3mp}.  In the first inequality above, we applied
	\begin{align}\label{kmp3}
		\int_{ \R^3}\frac{C}{|u|^\ga}\exp[-\frac{m_j}{4}\big(\frac{(|u|-2|v|\cos\tilde{\theta}+2\frac{\De_* I}{m_i|u|})^2}{4}+\frac{m_i^2|u|^2}{4m_j^2} \big)]dv_* \leq\frac{C}{1+|v|}, 
	\end{align}
	for any $\ga>-3$, which can be obtained by the calculations in \eqref{4boundk3}, \eqref{5boundk3} and \eqref{6boundk3}.  Next, considering the term $k^2_{ij}$, if $m_i\neq m_j$ we can use \eqref{Wk2} and similar arguments leading to \eqref{boundk2} to find that
	\begin{align}\label{boundk2pp}
		k^2_{ij}&(v,v_*,I_*)\frac{w_{i\be}(v)}{w_{j\be}(v_*,I_*)}e^\frac{\de|m_iv-m_jv_*|^2}{64}(1+I_*)^{1/8}\notag\\
		&\leq C\int_{\R_+} \exp[-\frac{(\sqrt{m_i}-\sqrt{m_j})^2}{16}|u_{ij}|^2-\frac{m_im_j}{16(\sqrt{m_i}+\sqrt{m_j})^2}(|u|^2-2\frac{m_j-m_i}{m_im_j}(I_*-I'_*))]\notag\\
		&\qquad\times (1+\frac{1}{|u|^{\eta}}) e^{-\frac{I'_*}{4}}(I_*)^{\de_j/4-1/2}(I'_*)^{\de_j/2-1}\Psi_{mp}(I_*,I'_*)(1+I_*)^{1/8}dI'_*,
	\end{align}
	where $\Psi_{mp}(I_*,I'_*):=(I_*)^{-\psi_{1*}}(I'_*)^{-\psi_{2*}}$ with $\psi_{i*}>0$ arbitrarily chosen such that $\psi_{1*}+\psi_{2*}=\de_j/2$.  We consider \eqref{boundk2pp} in four cases.
	
	\smallskip
	\noindent{\it Case 1. } $m_i>m_j$.  In this case, $k^2_{ij}$ enjoys exponential decay in both $I_*$ and $I'_*$ such that
	\begin{align*}
		k^2_{ij}&(v,v_*,I_*)\frac{w_{i\be}(v)}{w_{j\be}(v_*,I_*)}e^\frac{\de|m_iv-m_jv_*|^2}{64}(1+I_*)^{1/8}\notag\\
		&\leq C\int_{\R_+} \exp[-\frac{(\sqrt{m_i}-\sqrt{m_j})^2}{16}|u_{ij}|^2-\frac{m_im_j}{16(\sqrt{m_i}+\sqrt{m_j})^2}(|u|^2-2\frac{m_j-m_i}{m_im_j}I_*)]\notag\\
		&\qquad\times (1+\frac{1}{|u|^{\eta}}) e^{-\frac{I'_*}{8}}(I'_*)^{\de_j/2-1}\Psi_{mp}(I_*,I'_*)dI'_*.
	\end{align*}
	Then, by choosing
	$$
	\Psi_{mp}(I_*,I'_*)=\frac{1}{(I_*)^{1/2}(I'_*)^{\de_j/2-1/2}},
	$$
	and using \eqref{reuij}, \eqref{rev}, \eqref{positivity} with $\ka=0$, and similar calculations given in \eqref{boundk2pp1}, it holds that
	\begin{align*}
		&\int_{\R^3\times \R_+}k^2_{ij}(v,v_*,I_*)\frac{w_{i\be}(v)}{w_{j\be}(v_*,I_*)}e^\frac{\de|m_iv-m_jv_*|^2}{64}(1+I_*)^{1/8}dv_*dI_*
		\leq \frac{C}{1+|v|}.
	\end{align*}
	
	\medskip
	\noindent{\it Case 2. } $m_i<m_j$ and $I_*\leq 1$.  Under this condition, the term $\exp[\frac{m_im_j}{16(\sqrt{m_i}+\sqrt{m_j})^2}2\frac{m_j-m_i}{m_im_j}I_*]$ is bounded. Therefore,
	\begin{align*}
		k^2_{ij}&(v,v_*,I_*)\frac{w_{i\be}(v)}{w_{j\be}(v_*,I_*)}e^\frac{\de|m_iv-m_jv_*|^2}{64}(1+I_*)^{1/8}\notag\\
		& \leq C\int_{\R_+} \exp[-\frac{(\sqrt{m_i}-\sqrt{m_j})^2}{16}|u_{ij}|^2-\frac{m_im_j}{16(\sqrt{m_i}+\sqrt{m_j})^2}|u|^2]\notag\\
		&\qquad \times (1+\frac{1}{|u|^{\eta}}) e^{-\frac{I'_*}{8}}(I_*)^{\de_j/4-1/2}(I'_*)^{\de_j/2-1}\Psi_{mp}(I_*,I'_*)dI'_*.
	\end{align*}
	We use \eqref{positivity} and \eqref{propertyv}, and set
	$$
	\Psi_{mp}(I_*,I'_*)=\frac{1}{(I_*)^{1/2}(I'_*)^{\de_j/2-1/2}},
	$$
	to get that
	\begin{align*}
		&\int_{\R^3\times \R_+}k^2_{ij}(v,v_*,I_*)\frac{w_{i\be}(v)}{w_{j\be}(v_*,I_*)}e^\frac{\de|m_iv-m_jv_*|^2}{64}(1+I_*)^{1/8}dv_*dI_*
		\leq \frac{C}{1+|v|}.
	\end{align*}
\noindent{\it Case 3. } $m_i<m_j$, $I_*\geq 1$ and $I_*\leq \ka(|v|^2+1)$.  Notice that $I_*\leq \ka(|v|^2+1)$, thus \eqref{reuij} and \eqref{defAB} imply that
	\begin{align*}
		k^2_{ij}&(v,v_*,I_*)\frac{w_{i\be}(v)}{w_{j\be}(v_*,I_*)}e^\frac{\de|m_iv-m_jv_*|^2}{64}(1+I_*)^{1/8}\notag\\
		&\leq C\int_{\R_+} \exp[-\frac{(\sqrt{m_i}-\sqrt{m_j})^2}{16}|u_{ij}|^2-\frac{m_im_j}{16(\sqrt{m_i}+\sqrt{m_j})^2}(|u|^2-2\frac{m_j-m_i}{m_im_j}I_*)+\frac{I_*}{8}]\notag\\
		&\qquad\times (1+\frac{1}{|u|^{\eta}}) e^{-\frac{I'_*}{8}-\frac{I_*}{8}}(I_*)^{\de_j/4-1/2}(I'_*)^{\de_j/2-1}\Psi_{mp}(I_*,I'_*)(1+I_*)^{1/8}dI'_*\notag\\
		&\leq C\int_{\R_+} \exp[-A_1^2\frac{\big||v|^2-|v_*|^2\big|^2}{|v-v_*|^2}-2A_1A_2(|v|^2-|v_*|^2)-(A_2^2+B)|v-v_*|^2+\frac{\ka}{4}|v|^2]\notag\\
		&\qquad\times (1+\frac{1}{|u|^{\eta}}) e^{-\frac{I'_*}{8}-\frac{I_*}{8}}(I_*)^{\de_j/4-1/2}(I'_*)^{\de_j/2-1}\Psi_{mp}(I_*,I'_*)(1+I_*)^{1/8}dI'_*.
	\end{align*}
	We perform similar calculations as in \eqref{boundk2pp1}, select a sufficiently small $\ka$, and choose 
	$$
	\Psi_{mp}(I_*,I'_*)=\frac{1}{(I_*)^{1/2}(I'_*)^{\de_j/2-1/2}}
	$$
	to deduce that
	\begin{align*}
		\int_{\R^3\times \R_+}&k^2_{ij}(v,v_*,I_*)\frac{w_{i\be}(v)}{w_{j\be}(v_*,I_*)}e^\frac{\de|m_iv-m_jv_*|^2}{64}(1+I_*)^{1/8}dv_*dI_*\notag\\
		\leq& C\int_{\R^3} \exp[-c\frac{\big||v|^2-|v_*|^2\big|^2}{|v-v_*|^2}-c|v-v_*|^2](1+\frac{1}{|u|^{\eta}})dv_*\notag\\
		&\qquad \times  \int_{(\R_+)^2}e^{-\frac{I_*}{8}}e^{-\frac{I'_*}{8}}(I_*)^{\de_j/4-1}(I'_*)^{-1/2}(1+I_*)^{1/8}dI'_*dI_*\notag\\
		\leq& \frac{C}{1+|v|}.
	\end{align*}
	\noindent{\it Case 4. } $m_i<m_j$, $I_*\geq 1$ and $I_*\geq \ka(|v|^2+1)$.  In this case set $\Psi_{mp}(I_*,I'_*)=\frac{1}{(I_*)^{\de_j/2}}$.  Then, it holds by \eqref{boundk2pp} and $I_*\geq \ka(|v|^2+1)$ that
	\begin{align*}
		\int_{\R^3\times \R_+}&k^2_{ij}(v,v_*,I_*)\frac{w_{i\be}(v)}{w_{j\be}(v_*,I_*)}e^\frac{\de|m_iv-m_jv_*|^2}{64}(1+I_*)^{1/8}dv_*dI_*\notag\\
		&\leq C\int_{\R^3\times \R_+}\int_{\R_+} \exp[-\frac{(\sqrt{m_i}-\sqrt{m_j})^2}{16}|u_{ij}|^2-\frac{m_im_j}{16(\sqrt{m_i}+\sqrt{m_j})^2}(|u|^2-2\frac{m_j-m_i}{m_im_j}(I_*-I'_*))]\notag\\
		&\qquad\times (1+\frac{1}{|u|^{\eta}}) e^{-\frac{I'_*}{4}}(I_*)^{-\de_j/4-1/2}(I'_*)^{\de_j/2-1}(1+I_*)^{1/8}dv_*dI_*dI'_*\notag\\
		\leq\frac{C}{1+|v|}&\int_{\R^3\times \R_+}\exp[-\frac{(\sqrt{m_i}-\sqrt{m_j})^2}{16}|u_{ij}|^2]\int_{\R_+}\exp[-\frac{m_im_j}{16(\sqrt{m_i}+\sqrt{m_j})^2}(|u|^2-2\frac{m_j-m_i}{m_im_j}(I_*-I'_*))]dI_*\notag\\
		&\qquad\times (1+\frac{1}{|u|^{\eta}}) e^{-\frac{I'_*}{4}}(I'_*)^{\de_j/2-1}dv_*dI'_*.
	\end{align*}
	One can see from \eqref{Wk2} that $|u|^2-2\frac{m_j-m_i}{m_im_j}(I_*-I'_*)=|u'|^2\geq 0$, which, together with $I_*\geq 0$, shows that
	\begin{align*}
		\int_{\R_+}\exp[-\frac{m_im_j}{16(\sqrt{m_i}+\sqrt{m_j})^2}&(|u|^2-2\frac{m_j-m_i}{m_im_j}(I_*-I'_*))]dI_*\notag\\
		=&\int^{\frac{|u|^2}{A_4}+I'_*}_0\exp[-A_3(|u|^2-A_4(I_*-I'_*))]dI_* \leq C,
	\end{align*}
	where we denoted $A_3=\frac{m_im_j}{16(\sqrt{m_i}+\sqrt{m_j})^2}$ and $A_4=2\frac{m_j-m_i}{m_im_j}$.
	It follows from the previous two inequalities that
	\begin{align*}
		\int_{\R^3\times \R_+}&k^2_{ij}(v,v_*,I_*)\frac{w_{i\be}(v)}{w_{j\be}(v_*,I_*)}e^\frac{\de|m_iv-m_jv_*|^2}{64}(1+I_*)^{1/8}dv_*dI_*\notag\\
		\leq& \frac{C}{1+|v|} \int_{\R^3}\exp[-\frac{(\sqrt{m_i}-\sqrt{m_j})^2}{16}|u_{ij}|^2](1+\frac{1}{|u|^{\eta}}) dv_*\notag\\
		\leq&\frac{C}{1+|v|}\big\{\int_{|u|\leq 1}(1+\frac{1}{|u|^{\eta}}) dv_*+ \int_{|u|\geq 1}\exp[-\frac{(\sqrt{m_i}-\sqrt{m_j})^2}{16}|u_{ij}|^2] dv_*\big\}\notag \leq \frac{C}{1+|v|}.
	\end{align*}
	Therefore, \eqref{k2mp} holds by the above four cases for $m_i\neq m_j$.  If $m_i=m_j$ the structure is similar to $k^3_{ij}$. From similar arguments leading to \eqref{k2mimj}, one can obtain that
	\begin{align*}
		k^2_{ij}&(v,v_*,I_*)\frac{w_{i\be}(v)}{w_{j\be}(v_*,I_*)}e^\frac{\de|v-v_*|^2}{64}(1+I_*)^{\frac{1}{8}}\notag\\ &\leq C(I_*)^{\de_j/4-1/2}\frac{1}{|u|} \int_{(\R^3)^{\perp n}\times \R_+}\exp[-\frac{m_j}{4}\big( |V_\perp+\om'|^2+\frac{(|u|-2|v|\cos\tilde{\theta}+2\frac{\De I}{m_i|u|})^2}{4}+\frac{|u|^2}{4} \big)]\notag\\
		&\hspace{5cm}\times e^{-\frac{I'_*}{4}}(I^\prime_*)^{\de_j/2-1}(1+I_*)^{1/8}\frac{1}{\widehat{E}_{ij}^{(\eta+\de_j)/2}} d\om' dI^\prime_*,
	\end{align*}
	with the notations introduced in \eqref{k2velocity} and \eqref{notationsk3}.
	Then \eqref{k2mp} holds by using
	\begin{align*}
		\frac{1}{\widehat{E}_{ij}^{\de_j/2}}\leq& \frac{1}{(I_*)^{\de_j/4}(I'_*)^{\de_j/4}}\chi_{\{I_*\leq 1\}}+\frac{1}{(I_*)^{\de_j/2}}\chi_{\{I_*> 1\}},
	\end{align*}
	and integrating in $(v_*,I_*)$.
\end{proof}

The non-integrability in Poly-Mono collision occurs mainly in the estimates on $k^3_{ij}$, where the algebraic component for $I$ and $I_*$ in \eqref{1boundk3} becomes
$
(II_*)^{\de_i/4-1/2}\frac{(1+I_*)^{1/8}}{\widetilde{E}^{\de_i/2}_{ij}}.
$
If we use the same approach as in Lemma \ref{leK}, such component is bounded by 
$$
(II_*)^{\de_i/4-1/2}\frac{(1+I_*)^{1/8}}{(I)^{\psi}(I_*)^{\de_i/2-\psi}},\qquad 0\leq\psi\leq \frac{\de_i}{2},
$$
leading to an non-integrable integrand on $I_*$ over $[0,\infty)$.  Instead, we will use the dependency on $|u|$ in $\widetilde{E}^{\de_i/2}_{ij}$ to obtain integrability.   This procedure, however, will cause a slower algebraic decay in $v$ compared to the monatomic gas case, see \cite{Guo}.
\begin{lemma}[Poly-Mono Collision]\label{leKpm}
	For $i> p$ and $j\leq p$, the linearized operator defined in \eqref{DefKij} can be written as
	\begin{align*}
		&\dis 	K_{ij}f=	K_{ij}f(v,I)=\int_{\R^3}(k^2_{ij}(v,I,v_*)-k^1_{ij}(v,I,v_*))f_j(v_*)dv_*+\int_{\R^3\times \R_+}k^3_{ij}(v,I,v_*,I_*)f_i(v_*,I_*)dv_*dI_*,
	\end{align*}
	with
\begin{align}
	\int_{\R^3}& k^1_{ij}(v,I,v_*)\frac{w_{i\be}(v,I)}{w_{j\be}(v_*)}e^\frac{\de|v-v_*|^2}{64} dv_*\leq \frac{C}{1+|v|^{\frac{1}{2}}+I^\frac{1}{8}},\label{kpm}\\
	\int_{\R^3} &k^2_{ij}(v,I,v_*)\frac{w_{i\be}(v,I)}{w_{j\be}(v_*)}e^\frac{\de|m_iv-m_jv_*|^2}{64} dv_*\leq \frac{C}{1+|v|^{\frac{1}{2}}+I^\frac{1}{8}},\label{k2pm}\\
	\int_{\R^3\times\R_+} &k^3_{ij}(v,I,v_*,I_*)\frac{w_{i\be}(v,I)}{w_{i\be}(v_*,I_*)}e^\frac{\de|v-v_*|^2}{64}(1+I_*)^{\frac{1}{8}} dv_*dI_*\leq \frac{C}{1+|v|^{\frac{1}{2}}+I^\frac{1}{8}},\label{k3pm}
\end{align}
	where $\dis \de=\min_{1\leq i,j\leq n}\{m_i,\frac{(\sqrt{m_i}-\sqrt{m_j})^2}{16(m_i-m_j)^2}\}.$
\end{lemma}
\begin{proof}
	Applying \eqref{changeofvariables} and \eqref{changeM} to the general bound for $k^1_{ij}$ in \eqref{Gk1} leads to
	\begin{align*}
		k^1_{ij}&(v,I,v_*)=\int_{\S^2\times \R_+} \frac{ (M'_iM'_{j*})^{1/2}}{(I')^{\de_i/4-1/2}}(I)^{\de_i/4-1/2}\si_{ij}|u|d\om  dI^\prime \notag\\
		&\leq Ce^{-\frac{m_i|v|^2}{8}-\frac{m_j|v_*|^2}{8}-\frac{I}{4}}(I)^{\de_i/4-1/2}\int_{\R_+} e^{-\frac{m_i|v'|^2}{8}-\frac{m_j|v'_*|^2}{8}-\frac{I'}{4}}\frac{(I')^{\de_i/2-1}\chi_{\{0\leq I'\leq E_{ij}\}}}{E_{ij}^{(\eta+\de_i-1)/2}} dI^\prime \notag\\
		&\leq Ce^{-\frac{m_i|v|^2}{8}-\frac{m_j|v_*|^2}{8}-\frac{I}{4}}(I)^{\de_i/4-1/2}E_{ij}^{(1-\eta)/2}\int_{0\leq I'\leq E_{ij}}\frac{(I')^{\de_i/2-1}}{E_{ij}^{\de_i/2}} dI^\prime\leq Ce^{-\frac{m_i|v|^2}{16}-\frac{m_j|v_*|^2}{16}-\frac{I}{8}},
	\end{align*}
	which yields \eqref{kpm} by taking integral in $v_*$.  Next, we bound $k^3_{ij}$ from the general representation \eqref{Gk3} with notations defined in \eqref{notationk3}, \eqref{Wk3} and \eqref{conservationk3}.
	 Using similar arguments in \eqref{cv}, \eqref{k3}, \eqref{2boundk3}, \eqref{controle}, together with \eqref{computev} and the fact that $$\frac{w^2_{i\be}(v,I)}{w^2_{i\be}(v_*,I_*)}\leq C \Big(1+\big|m_i|v|^2-m_i|v_*|^2+I-I_*\big|^2\Big)^\be\leq C e^\frac{m_j|v'|^2+m_j|v'_*|^2}{16},$$ we obtain that
	\begin{align*}
		\int_{\R^3\times (\R_+)^3}&k^3_{ij}(v,I,v_*,I_*)\frac{w_{i\be}(v,I)}{w_{i\be}(v_*,I_*)}e^\frac{\de|v-v_*|^2}{64}(1+I_*)^{\frac{1}{8}}dI^\prime dI^\prime_*dI_*dv_*\notag\\ \leq&\int_{ \R^3\times (\R_+)^3}\frac{C}{|u|}\exp[-\frac{m_j}{4}\big(\frac{(|u|-2|v|\cos\tilde{\theta}+2\frac{\De_* I}{m_i|u|})^2}{4}+\frac{m_i^2|u|^2}{4m_j^2} \big)](II_*)^{\de_i/4-1/2}\frac{(1+I_*)^{1/8}}{\widetilde{E}_{ij}^{\de_i/2}}dI_*dv_*.
	\end{align*}
	We split the integration domain in two pieces. First for $I_*\leq |v|+1$, it holds by $$(II_*)^{\de_i/4-1/2}\frac{(1+I_*)^{1/8}}{\widetilde{E}_{ij}^{\de_i/2}}\leq C(II_*)^{-1/2}(1+I_*)^{1/8}\chi_{\{I\geq1\}}+C(I_*)^{-3/4}|u|^{-1/2}(1+I_*)^{1/8}\chi_{\{I\leq1\}}$$
	that
	\begin{align}\label{k3largeI*}
		\int_{\R^3\times\R_+}k^3_{ij}(v,I,&v_*,I_*)\frac{w_{i\be}(v,I)}{w_{i\be}(v_*,I_*)}e^\frac{\de|v-v_*|^2}{64}(1+I_*)^{1/8}dv_*dI_*\notag\\
		\leq&\frac{1}{(1+|v|)I^{1/2}}\chi_{\{I\geq1\}}\int^{|v|+1}_0(I_*)^{-1/2}(1+I_*)^{1/8}dI_* \notag\\
	&\qquad+\frac{1}{(1+|v|)}\chi_{\{I\leq1\}}\int^{|v|+1}_0(I_*)^{-3/4}(1+I_*)^{1/8}dI_*\leq\frac{1}{1+|v|^{1/2}+I^{1/2}}.
	\end{align}
In the first inequality above, the decay in $v$ is obtained using the calculations in \eqref{4boundk3}, \eqref{5boundk3}, \eqref{6boundk3}.  Whereas, for $I_*\geq |v|+1$, it follows from 
	\begin{align}\label{controltildeEpm}
		(II_*)^{\de_i/4-1/2}\frac{(1+I_*)^{1/8}}{\widetilde{E}_{ij}^{\de_i/2}}\leq C(I)^{-1/4}(I_*)^{-3/4}(1+I_*)^{1/8}\chi_{\{I\geq1\}}+C(I_*)^{-1}(1+I_*)^{1/8}\chi_{\{I\leq1\}}
		\end{align}
	that
	\begin{align}\label{k3largeI*pm}
		&\int_{\R^3\times\R_+}k^3_{ij}(v,I,v_*,I_*)\frac{w_{i\be}(v,I)}{w_{i\be}(v_*,I_*)}e^\frac{\de|v-v_*|^2}{64}(1+I_*)^{1/8}dv_*dI_*\notag\\
		\leq C&(I)^{-1/4}\chi_{\{I\geq1\}}\int_{ \R^3\times\R_+}\frac{1}{|u|}\exp[-\frac{m_j}{4}\big(\frac{(|u|-2|v|\cos\tilde{\theta}+2\frac{\De_* I}{m_i|u|})^2}{4}+\frac{m_i^2|u|^2}{4m_j^2} \big)](I_*)^{-3/4}(1+I_*)^{1/8}dv_*dI_*\notag\\
		&+C\chi_{\{I\leq1\}}\int_{ \R^3\times\R_+}\frac{1}{|u|}\exp[-\frac{m_j}{4}\big(\frac{(|u|-2|v|\cos\tilde{\theta}+2\frac{\De_* I}{m_i|u|})^2}{4}+\frac{m_i^2|u|^2}{4m_j^2} \big)](I_*)^{-1}(1+I_*)^{1/8}dv_*dI_*.
	\end{align}
	We perform the change of variable $dI_*=\frac{m_i}{2}|u|d\Phi$ with $\Phi:=|u|-2|v|\cos\tilde{\theta}+2\frac{\De_* I}{m_i|u|}$ to obtain that
	\begin{align}\label{k3smallI*}
		\int_{\R^3\times\R_+}&k^3_{ij}(v,I,v_*,I_*)\frac{w_{i\be}(v,I)}{w_{i\be}(v_*,I_*)}e^\frac{\de|v-v_*|^2}{64}(1+I_*)^{1/8}dv_*dI_*\notag\\
		\leq&\frac{C}{1+I^{1/4}}\int_{ \R^3\times\R_+}\exp[-\frac{m_j}{4}\big(\frac{\Phi^2}{4}+\frac{m_i^2|u|^2}{4m_j^2} \big)](I_*)^{-3/4}(1+I_*)^{1/8}dud\Phi\notag\\
		\leq&\frac{C}{1+I^{1/4}}\frac{1}{(1+|v|)^{1/2}}\int_{ \R^3\times\R_+}\exp[-\frac{m_j}{4}\big(\frac{\Phi^2}{4}+\frac{m_i^2|u|^2}{4m_j^2} \big)]dud\Phi\leq\frac{C}{1+|v|^{1/2}+I^{1/4}}.
	\end{align}
	We combine \eqref{k3largeI*} and \eqref{k3smallI*} to get that
	\begin{align*}
		\int_{\R^3\times\R_+}&k^3_{ij}(v,I,v_*,I_*)\frac{w_{i\be}(v,I)}{w_{i\be}(v_*,I_*)}e^\frac{\de|v-v_*|^2}{64}(1+I_*)^{1/8}dv_*dI_*\notag\\
		\leq&\frac{C}{1+I^{1/4}}\int_{ \R^3\times\R_+}\exp[-\frac{m_j}{4}\big(\frac{\Phi^2}{4}+\frac{m_i^2|u|^2}{4m_j^2} \big)](I_*)^{-3/4}(1+I_*)^{1/8}dud\Phi\notag\\
		\leq&\frac{C}{1+I^{1/4}}\frac{1}{(1+|v|)^{1/2}}\int_{ \R^3\times\R_+}\exp[-\frac{m_j}{4}\big(\frac{\Phi^2}{4}+\frac{m_i^2|u|^2}{4m_j^2} \big)]dud\Phi\leq\frac{C}{1+|v|^{1/2}+I^{1/4}},
	\end{align*}
which, together with \eqref{k3largeI*}, yields \eqref{k3pm}.  Next, for $k^2_{ij}$ in the case $m_i\neq m_j$ we perform similar calculations leading to \eqref{boundk2} to obtain that
	\begin{align}\label{kpm2}
		k^2_{ij}(v,I,&v_*)\frac{w_{i\be}(v,I)}{w_{j\be}(v_*)}e^\frac{\de|m_iv-m_jv_*|^2}{64}\notag\\
		\leq& C\int_{\R_+} \exp[-\frac{(\sqrt{m_i}-\sqrt{m_j})^2}{16}|u_{ij}|^2-\frac{m_im_j}{16(\sqrt{m_i}+\sqrt{m_j})^2}(|u|^2-2\frac{m_j-m_i}{m_im_j}(I'-I))]\notag\\
		&\qquad\times (1+\frac{1}{|u|^{\eta}}) e^{-\frac{I'}{4}}(I')^{\de_i/2-1}(I)^{\de_i/4-1/2}\Psi_{pm}(I,I',|u|)dI',
	\end{align}
	where $\Psi_{pm}(I,I',|u|):=(I)^{-\psi'_{1}}(I')^{-\psi'_{2}}|u|^{-2\psi'_{3}}$ with $\psi'_{i}>0$ arbitrarily chosen such that $\psi'_{1}+\psi'_{2}+\psi'_{3}=\de_i/2$.  Now, in the case $m_j>m_i$, $k^2_{ij}$ decays in $I$ and $I'$ exponentially since
	\begin{align*}
		k^2_{ij}(v,I,&v_*)\frac{w_{i\be}(v,I)}{w_{j\be}(v_*)}e^\frac{\de|m_iv-m_jv_*|^2}{64}\notag\\
		&\leq C\int_{\R_+} \exp[-\frac{(\sqrt{m_i}-\sqrt{m_j})^2}{16}|u_{ij}|^2-\frac{m_im_j}{16(\sqrt{m_i}+\sqrt{m_j})^2}(|u|^2+2\frac{m_j-m_i}{m_im_j}I)]\notag\\
		&\qquad\times (1+\frac{1}{|u|^{\eta}})(I)^{\de_i/4-1/2} e^{-\frac{I'}{8}}(I')^{\de_i/2-1}\Psi_{pm}(I,I',|u|)dI'.
	\end{align*}
	Thus, choosing $\Psi_{pm}(I,I',|u|)=(I')^{-\de_i/2+1/2}|u|^{-1}$ we obtain that
	\begin{align*}
		\int_{\R^3}k^2_{ij}&(v,I,v_*)\frac{w_{i\be}(v,I)}{w_{j\be}(v_*)}e^\frac{\de|m_iv-m_jv_*|^2}{64}dv_*\notag\\
		\leq& \frac{C}{1+I} \int_{\R^3}\exp[-\frac{(\sqrt{m_i}-\sqrt{m_j})^2}{16}|u_{ij}|^2-\frac{m_im_j}{16(\sqrt{m_i}+\sqrt{m_j})^2}|u|^2](1+\frac{1}{|u|^2})dv_*.
	\end{align*}
	Then similar arguments as in \eqref{boundk2pp1} and \eqref{propertyv} yield
	\begin{align*}
		&\int_{\R^3}k^2_{ij}(v,I,v_*)\frac{w_{i\be}(v,I)}{w_{j\be}(v_*)}e^\frac{\de|m_iv-m_jv_*|^2}{64}dv_*
		\leq \frac{C}{1+|v|+I}.
	\end{align*}
In the case $m_j>m_i$ and $I\geq \ka(|v|^2+1)$ we set $\Psi_{pm}(I,I',|u|)=(I)^{-\de_i/2}$ and use \eqref{kpm2} to obtain that
	\begin{align*}
		\int_{\R^3}k^2_{ij}&(v,I,v_*)\frac{w_{i\be}(v,I)}{w_{j\be}(v_*)}e^\frac{\de|m_iv-m_jv_*|^2}{64}dv_*\notag\\
		&\leq C\int_{\R^3} \exp[-\frac{(\sqrt{m_i}-\sqrt{m_j})^2}{16}|u_{ij}|^2](1+\frac{1}{|u|^{\eta}})(I)^{-\de_i/4-1/2} dv_*\notag\\
		&\leq\frac{C}{1+|v|+I}\Big\{\int_{|u|\leq 1}(1+\frac{1}{|u|^{\eta}}) dv_*+\int_{|u|\geq1} \exp[-\frac{(\sqrt{m_i}-\sqrt{m_j})^2}{16}|u_{ij}|^2]dv_*\Big\} \leq \frac{C}{1+|v|+I}.
	\end{align*}
If $m_j>m_i$ and $I\leq \ka(|v|^2+1)$, by choosing $\ka$ to be small, we use the arguments in \eqref{boundk2c61} and \eqref{positivity} give that
	\begin{align*}
		k^2_{ij}(v,I,&v_*)\frac{w_{i\be}(v,I)}{w_{j\be}(v_*)}e^\frac{\de|m_iv-m_jv_*|^2}{64}\notag\\
		& \leq C\int_{\R_+} \exp[-c\frac{\big||v|^2-|v_*|^2\big|^2}{|v-v_*|^2}-c|v-v_*|^2] (1+\frac{1}{|u|^{\eta}}) e^{-\frac{I'}{8}}(I')^{\de_i/2-1}(I)^{\de_i/4-1/2}\Psi_{pm}(I,I',|u|)dI'.
	\end{align*} 
Choose
\begin{equation*}
\Psi_{pm}(I,I',|u|)=(I)^{-\de_i/2}\chi_{\{I>1\}}+(I')^{-\de_i/2-1/2}|u|^{-1}\chi_{\{I\leq1\}},
\end{equation*}
to obtain that
	\begin{align*}
		&\int_{\R^3}k^2_{ij}(v,I,v_*)\frac{w_{i\be}(v,I)}{w_{j\be}(v_*)}e^\frac{\de|m_iv-m_jv_*|^2}{64}dv_*\leq \frac{C}{1+|v|+I}.
	\end{align*}
Overall, given the aforementioned arguments, \eqref{k2pm} holds for $m_j\neq m_j$.  Finally, in the case $m_i=m_j$, recalling the notations in \eqref{k2velocity} and \eqref{notationsk3}, the similar structure to $k^3_{ij}$ yields
	\begin{align*}
		&k^2_{ij}(v,I,v_*)\frac{w_{i\be}(v,I)}{w_{j\be}(v_*)}e^\frac{\de|v-v_*|^2}{64}\leq \frac{C}{|u|}(I)^{\de_i/4-1/2} \\
		&\times\int_{(\R^3)^{\perp n}\times \R_+}\exp[-\frac{m_j}{4}\big( |V_\perp+\om'|^2+\frac{(|u|-2|v|\cos\tilde{\theta}+2\frac{\De I}{m_i|u|})^2}{4}+\frac{|u|^2}{4} \big)]e^{-\frac{I'}{4}}(I')^{\de_i/2-1}\frac{1}{\widehat{E}_{ij}^{(\eta+\de_i)/2}} d\om' dI^\prime.
	\end{align*}
	Therefore, \eqref{k2mp} holds using the estimate
	\begin{align*}
		\frac{1}{\widehat{E}_{ij}^{\de_i/2}}\leq& \frac{1}{(I')^{\de_j/4-1/8}|u|^{1/4}}\chi_{\{I\leq 1\}}+\frac{1}{(I)^{\de_j/2}}\chi_{\{I\geq 1\}},
	\end{align*}
and the arguments given in \eqref{4boundk3}, \eqref{5boundk3} and \eqref{6boundk3}.
\end{proof}
At this point the argument for the case Mono-Mono collisions is shorter since several ingredients of the proof are contained in the previous cases.
\begin{lemma}[Mono-Mono Collision]\label{leKmm}
	For $i,j\leq p$, the linearized operator $K_{ij}$ can be written as
	\begin{align*}
		&\dis 	K_{ij}f=	K_{ij}f(v)=\int_{\R^3}(k^2_{ij}(v,v_*)f_j(v_*)+k^3_{ij}(v,v_*)f_i(v_*)-k^1_{ij}(v,v_*)f_j(v_*) )dv_*,
	\end{align*}
	with
	\begin{align}
		&\int_{\R^3} k^1_{ij}(v,v_*)\frac{w_{i\be}(v)}{w_{j\be}(v_*)}e^\frac{\de|v-v_*|^2}{64} dv_*\leq \frac{C}{1+|v|^{\frac{1}{2}}},\label{kmm}\\
		&\int_{\R^3} k^2_{ij}(v,v_*)\frac{w_{i\be}(v)}{w_{j\be}(v_*)}e^\frac{\de|m_iv-m_jv_*|^2}{64} dv_*\leq \frac{C}{1+|v|^{\frac{1}{2}}},\label{k2mm}\\
		&\int_{\R^3} k^3_{ij}(v,v_*)\frac{w_{i\be}(v)}{w_{i\be}(v_*)}e^\frac{\de|v-v_*|^2}{64} dv_*\leq \frac{C}{1+|v|^{\frac{1}{2}}},\label{k3mm}
	\end{align}
	where $\dis \de=\min_{1\leq i,j\leq n}\{m_i,\frac{(\sqrt{m_i}-\sqrt{m_j})^2}{16(m_i-m_j)^2}\}.$
\end{lemma}
\begin{proof}
	We apply change of variable \eqref{changeofvariables} and the property \eqref{changeM} to the bound of $k^1_{ij}$ in \eqref{Gk1} to obtain that
	\begin{align*}
		k^1_{ij}(v_*)&=\int_{\S^2} (M'_iM'_{j*})^{1/2}\si_{ij}|u|d\om\notag\\
		&\leq Ce^{-\frac{m_i|v|^2}{8}-\frac{m_j|v_*|^2}{8}}\int_{\S^2} e^{-\frac{m_i|v'|^2}{8}-\frac{m_j|v'_*|^2}{8}}E_{ij}^{(1-\eta)/2} d\om\leq Ce^{-\frac{m_i|v|^2}{16}-\frac{m_j|v_*|^2}{16}}, 
	\end{align*}
	which gives \eqref{kmm} by taking the integral in $v_*$. For the term $k^3_{ij}$ use similar calculations as in \eqref{cv}, \eqref{k3}, and \eqref{2boundk3}, together with the fact that $\frac{w_{i\be}(v)}{w_{i\be}(v_*)}\leq C e^\frac{\de|v-v_*|^2}{64}\leq C e^\frac{m_j|v'|^2+m_j|v'_*|^2}{16}$, to obtain that
	\begin{align*}
		&\int_{\R^3}k^3_{ij}(v,v_*)\frac{w_{i\be}(v)}{w_{i\be}(v_*)}e^\frac{\de|v-v_*|^2}{64}dv_*
		\leq\int_{ \R^3\times\R_+}\frac{C}{|u|}e^{-\frac{m_j}{4}\left(\frac{(|u|-2|v|\cos\tilde{\theta})^2}{4}+\frac{m_i^2|u|^2}{4m_j^2} \right)}dv_*\leq \frac{1}{1+|v|},
	\end{align*}
	which shows \eqref{k3mm}.  Next, for the term  $k^2_{ij}$ if $m_i\neq m_j$ we perform similar calculations as in \eqref{boundk2}, \eqref{boundk2c61}, and \eqref{positivity}, to obtain that
	\begin{align*}
		\int_{\R^3}k^2_{ij}&(v,v_*)\frac{w_{i\be}(v)}{w_{j\be}(v_*)}e^\frac{\de|m_iv-m_jv_*|^2}{64}dv_*\notag\\
		&\leq C\int_{\R^3} \exp[-\frac{(\sqrt{m_i}-\sqrt{m_j})^2}{16}|u_{ij}|^2-\frac{m_im_j}{16(\sqrt{m_i}+\sqrt{m_j})^2}|u|^2](1+\frac{1}{|u|^{\eta}})dv_*\notag\\
		&\leq C\int_{\R^3} \exp[-c\frac{\big||v|^2-|v_*|^2\big|^2}{|v-v_*|^2}-c|v-v_*|^2](1+\frac{1}{|u|^{\eta}})dv_* \leq \frac{C}{1+|v|}.
	\end{align*}
	Therefore, \eqref{k2mm} is proved for $m_i\neq m_j$.  Finally, if $m_i=m_j$, recalling \eqref{k2mimj}, we have that
	\begin{align*}
		\int_{\R^3}k^2_{ij}&(v,v_*)\frac{w_{i\be}(v)}{w_{j\be}(v_*)}e^\frac{\de|v-v_*|^2}{64}dv_*\notag\\
		 \leq&C\int_{\R^3}\frac{1}{|u|} \int_{(\R^3)^{\perp n}\times \R_+}\exp[-\frac{m_j}{4}\big( |V_\perp+\om'|^2+\frac{(|u|-2|v|\cos\tilde{\theta})^2}{4}+\frac{|u|^2}{4} \big)]\frac{1}{\widehat{E}_{ij}^{\eta/2}} d\om' dv_*\leq \frac{C}{1+|v|},
	\end{align*}
	where $\om'$, $V_\perp$ and $\cos\tilde{\theta}$ are defined in \eqref{k2velocity} and \eqref{notationsk3}.
\end{proof}
These bounds on $\int k^{1,2,3}(Z,Z_*)dZ_*$ lead to Lemma \ref{leKij}.
\begin{proof}[Proof of Lemma \ref{leKij}]
Estimate \eqref{boundk} is a direct consequence of Lemma \ref{leK} - Lemma \ref{leKmm}. Moreover, \eqref{reKij} holds from \eqref{rek1ij}, \eqref{Gk3} and \eqref{rek2ij}.
\end{proof}

\subsection{Estimates on collision frequency and nonlinear operator}
In this section, we establish the bounds on collision frequency $\nu_i$ and nonlinear operator $\Ga_i$, which correspond to Lemma \ref{lenu} and Lemma \ref{lebound} respectively.

\begin{proof}[Proof of Lemma \ref{lenu}]
	Recall the notations for $Z_i$ and $\CZ_i$ in \eqref{DefZ}
	and rewrite
	\begin{align}\label{renu}
		\nu_i(Z)=\sum^n_{j=1}\int_{\CZ_j\times \CZ_i\times \CZ_j} W_{ij}(Z,Z_*|Z',Z'_*)\frac{M_{j*}}{(I)^{\de_i/2-1}(I_*)^{\de_j/2-1}}dZ_*dZ^\prime dZ^\prime_*:=\sum^n_{j=1}\nu_{ij}(Z_i).
	\end{align}
	It follows from \eqref{DefW}, \eqref{si} and similar calculations as in \eqref{rek1} that
	\begin{align*}
		\nu_{ij}(Z_i)\leq&  C\int_{\CZ_j\times \CZ_i\times \CZ_j} \sqrt{|u|^2-\frac{2\De I}{\mu_{ij}}}\frac{(I^\prime )^{\de_i/2-1}(I^\prime_*)^{\de_j/2-1}}{E_{ij}^{(\eta+\de_i\chi_{\{i>p\}}+\de_j\chi_{\{j>p\}})/2}}\frac{1}{|u'|^2}M_{j*}\notag\\
		&\qquad\times\bm{\de}_3(V_{ij}-V'_{ij})\bm{\de}_1(\sqrt{|u|^2-\frac{2\De I}{\mu_{ij}}}-|u^\prime|)dZ_*dZ^\prime dZ^\prime_*,
	\end{align*}
	where $V_{ij}=\frac{m_i v+m_j v_*}{m_i+m_j}$, $V_{ij}^\prime=\frac{m_i v'+m_j v'_*}{m_i+m_j}$ and $u'=v'-v'_*.$ We discuss individually the integral \eqref{renu} in the four collisional interactions. For Poly-Poly collision where $i,j>p$, we define $\om=\frac{u^\prime}{|u^\prime|}$ and use \eqref{changeofvariables} to get
	\begin{align}\label{nupp}
		\nu_{ij}(Z_i)\leq& C\int_{\S^2\times (\R_+)^2\times \R^3\times \R_+}e^{-\frac{m_j}{2}|v_*|^2-I_*} (I_*)^{\de_j/2-1}\sqrt{|u|^2-\frac{2\De I}{\mu_{ij}}}\frac{(I^\prime )^{\de_i/2-1}(I^\prime_*)^{\de_j/2-1}}{E_{ij}^{(\eta+\de_i+\de_j)/2}}d\om  dI^\prime dI^\prime_*dv_*dI_*\notag\\
		\leq& C\int_{\R^3\times \R_+}e^{-\frac{m_j}{2}|v_*|^2-I_*} (I_*)^{\de_j/2-1}E^{1/2-(\eta+\de_i+\de_j)/2}_{ij}\notag\\
		&\qquad\times \int_{0\leq I^\prime\leq E_{ij}}(I^\prime )^{\de_i/2-1}dI^\prime \int_{0\leq I^\prime_*\leq E_{ij}}(I^\prime_*)^{\de_j/2-1}dI^\prime_* dv_*dI_*\notag\\
		\leq& C\int_{\R^3\times \R_+}e^{-\frac{m_j}{2}|v_*|^2-\frac{I_*}{2}} (|u|^2+I+I_*)^{(1-\eta)/2}dv_*dI_*.
	\end{align}
	By the fact $|v-v_*|^2+I+I_*\leq (1+|v_*|^2)(1+I_*)(1+|v|^2+I)$, it holds that
	\begin{align*}
		\nu_{ij}(Z_i)
		\leq& C(1+|v|^2+I)^{1-\eta}\int_{\R^3\times \R_+}e^{-\frac{m_j}{2}|v_*|^2-\frac{I_*}{2}} (1+|v_*|^2)^{1-\eta}(1+I_*)^{1-\eta}dv_*dI_*\leq C(1+|v|^2+I)^{(1-\eta)/2}.
	\end{align*}
	Similarly, for Mono-Poly collision one has that
	\begin{align}\label{nump}
		\nu_{ij}(Z_i)\leq& C\int_{\S^2\times \R_+\times \R^3\times \R_+}e^{-\frac{m_j}{2}|v_*|^2-I_*} (I_*)^{\de_j/2-1}\sqrt{|u|^2-\frac{2\De I}{\mu_{ij}}}\frac{(I^\prime_*)^{\de_j/2-1}}{E_{ij}^{(\eta+\de_j)/2}}d\om   dI^\prime_*dv_*dI_*\notag\\
		\leq& C\int_{\R^3\times \R_+}e^{-\frac{m_j}{2}|v_*|^2-I_*} (I_*)^{\de_j/2-1}E^{1/2-(\eta+\de_j)/2}_{ij} \int_{0\leq I^\prime\leq E_{ij}}(I^\prime_*)^{\de_j/2-1}dI^\prime_* dv_*dI_*\notag\\
		\leq& C\int_{\R^3\times \R_+}e^{-\frac{m_j}{2}|v_*|^2-\frac{I_*}{2}} (|u|^2+I_*)^{(1-\eta)/2}dv_*dI_*\notag\\
		\leq& C(1+|v|^2)^{(1-\eta)/2}.
	\end{align}
	Moreover, for Poly-Mono collision, it holds that
	\begin{align}\label{nupm}
		\nu_{ij}(Z_i)\leq& C\int_{\S^2\times \R_+\times \R^3}e^{-\frac{m_j}{2}|v_*|^2} \sqrt{|u|^2-\frac{2\De I}{\mu_{ij}}}\frac{(I^\prime )^{\de_i/2-1}}{E_{ij}^{(\eta+\de_i)/2}}d\om  dI^\prime dv_*\notag\\
		\leq& C\int_{\R^3}e^{-\frac{m_j}{2}|v_*|^2} E^{1/2-(\eta+\de_i)/2}_{ij}\int_{0\leq I^\prime\leq E_{ij}}(I^\prime )^{\de_i/2-1}dI^\prime dv_*\notag\\
		\leq& C\int_{\R^3}e^{-\frac{m_j}{2}|v_*|^2} (|u|^2+I)^{(1-\eta)/2}dv_*\notag\\
		\leq& C(1+|v|^2+I)^{(1-\eta)/2}.
	\end{align}
	For Mono-Mono collisions, it is direct that
	\begin{align}\label{numm}
		\nu_{ij}(Z_i)\leq& C\int_{\S^2\times \R^3}e^{-\frac{m_j}{2}|v_*|^2}|u|\frac{1}{|u|^{\eta}}d\om  dv_*
		\leq C(1+|v|^2)^{(1-\eta)/2}.
	\end{align}
	Hence, we collect \eqref{nupp}, \eqref{nump}, \eqref{nupm} and \eqref{numm} to obtain
	\begin{align}\label{nuupper}
		\nu_{ij}(Z_i)\leq C(1+|v|^2+I\chi_{\{i>p\}})^{(1-\eta)/2}.
	\end{align}
	The lower bound can be proved in a similar way. We still use \eqref{DefW} and \eqref{si} to get
	\begin{align*}
		\nu_{ij}(Z_i)\geq& c\int_{\CZ_j\times \CZ_i\times \CZ_j} \sqrt{|u|^2-\frac{2\De I}{\mu_{ij}}}\frac{(I^\prime )^{\de_i/2-1}(I^\prime_*)^{\de_j/2-1}}{E_{ij}^{(\eta+\de_i\chi_{\{i>p\}}+\de_j\chi_{\{j>p\}})/2}}\frac{1}{|u'|^2}M_{j*}\notag\\
		&\qquad\times\bm{\de}_3(V_{ij}-V'_{ij})\bm{\de}_1(\sqrt{|u|^2-\frac{2\De I}{\mu_{ij}}}-|u^\prime|)dZ_*dZ^\prime dZ^\prime_*,
	\end{align*}
	For Poly-Poly collision, by \eqref{changeofvariables} and $\chi_{\{I^\prime+I^\prime_*<E_{ij}\}}\geq \chi_{\{I^\prime<E_{ij}/4\}}\chi_{\{I^\prime_*<E_{ij}/4\}}$, one has
	\begin{align*}
		\nu_{ij}(Z_i)\geq& c\int_{\S^2\times (\R_+)^2\times \R^3\times \R_+}e^{-\frac{m_j}{2}|v_*|^2-I_*} (I_*)^{\de_j/2-1}\sqrt{|u|^2-\frac{2\De I}{\mu_{ij}}}\frac{(I^\prime )^{\de_i/2-1}(I^\prime_*)^{\de_j/2-1}}{E_{ij}^{(\eta+\de_i+\de_j)/2}}d\om  dI^\prime dI^\prime_*dv_*dI_*\notag\\
		\geq& c\int_{\R^3\times \R_+}e^{-\frac{m_j}{2}|v_*|^2-I_*} (I_*)^{\de_j/2-1}E^{1/2-(\eta+\de_i+\de_j)/2}_{ij}\notag\\
		&\qquad\times \int_{0\leq I^\prime\leq E_{ij}/4}(I^\prime )^{\de_i/2-1}dI^\prime \int_{0\leq I^\prime_*\leq E_{ij}/4}(I^\prime_*)^{\de_j/2-1}dI^\prime_* dv_*dI_*\notag\\
		\geq& c\int_{\R^3\times \R_+}e^{-\frac{m_j}{2}|v_*|^2-\frac{I_*}{2}} (|u|^2+I+I_*)^{(1-\eta)/2}dv_*dI_*.
	\end{align*}
	We further apply $$(|u|^2+I+I_*)^{(1-\eta)/2}\geq c(\big||v|-|v_*|\big|+\sqrt{I})^{1-\eta}$$ and $$1\geq \chi_{\{|v|\geq 1\}}\chi_{\{|v_*|\leq 1/2\}}+\chi_{\{|v|< 1\}}\chi_{\{|v_*|> 2\}}$$ to deduce that
	\begin{align}\label{nupplo}
		\nu_{ij}(Z_i)&\geq c\int_{\R^3} e^{-|v_*|^2/2}(\big||v|-|v_*|\big|+\sqrt{I})^{1-\eta}\left(\chi_{\{|v|\geq 1\}}\chi_{\{|v_*|\leq 1/2\}}+\chi_{\{|v|< 1\}}\chi_{\{|v_*|> 2\}}\right) dv_*\notag\\
		&\geq c\int_{\R^3} e^{-|v_*|^2/2}(|v|+\sqrt{I})^{1-\eta}dv_*+c\int_{\R^3}e^{-|v_*|^2/2}(1+\sqrt{I})^{1-\eta}\chi_{\{|v|< 1\}}dv_*\notag\\
		&\geq c(1+|v|^2+I)^{(1-\eta)/2}.
	\end{align}
	The other three types of collisions can be estimated in a very similar way that
	\begin{align*}
		\nu_{ij}(Z_i)&\geq c(1+|v|^2+I\chi_{\{i>p\}})^{(1-\eta)/2},
	\end{align*}
	which, together with \eqref{renu} and \eqref{nuupper}, gives \eqref{nu}.
\end{proof}

Next, we need to prove the $L^\infty_{v,I}$ bound for the nonlinear operator $\Ga_i$. We use the definition of $\Ga_i$ in \eqref{DefLGa} to write
\begin{align}\label{defGa+-}
	\Ga_i(f,f)
	=&\sum_{j=1}^{n}\int_{\CZ_j\times \CZ_i\times \CZ_j}W_{ij}(Z,Z_*|Z',Z'_*)\notag\\
	&\qquad\qquad\times\big((\frac{M_i^\prime M_{j*}^\prime}{M_i})^{1/2}\frac{f_i^\prime f_{j*}^\prime}{( I^\prime_*)^{\de_i/2-1}( I^\prime_*)^{\de_j/2-1}}-(M_{j*})^{1/2}\frac{f_if_{j*}}{(I)^{\de_i/2-1}( I_*)^{\de_j/2-1}}\big)dZ_*dZ^\prime dZ^\prime_*\notag\\
	:=&\sum_{j=1}^{n}(\Ga^+_{ij}(f,f)-\Ga^-_{ij}(f,f))\notag\\
	:=&\Ga^+_i(f,f)-\Ga^-_i(f,f).
\end{align}

\begin{lemma}\label{Gabound}
	For any $f_i=f_i(Z_i)$, it holds that
	\begin{align}
		\left|w_i\Ga^{\pm}_i(f,f) \right|&\leq C\nu_i\sum^n_{j=1}\|(w_jf_j)(t)\|^2_{L^\infty_{v,I}}.\label{LinftyGa}
	\end{align}
\end{lemma}
\begin{proof}
Recalling $$V_{ij}=\frac{m_i v+m_j v_*}{m_i+m_j},\ V_{ij}^\prime=\frac{m_i v'+m_j v'_*}{m_i+m_j},\ u'=v'-v'_*,\ \De I=(I'-I)\chi_{\{i>p\}}+(I'_*-I_*)\chi_{\{j>p\}},$$ we use \eqref{DefW} to get
\begin{align*}
\Ga^-_{ij}(f,f)=&\int_{\CZ_j\times \CZ_i\times \CZ_j}W_{ij}(Z,Z_*|Z^\prime,Z^\prime_*)(M_{j*})^{1/2}\frac{f_if_{j*}}{(I)^{\de_i/2-1}(I_*)^{\de_j/2-1}}dv_*dv^\prime dv^\prime_* dI_*dI' dI^\prime_*\notag\\
= &C\int_{\CZ_j\times \CZ_i\times \CZ_j}(I_*)^{\de_j/4-1/2}e^{-\frac{m_j}{4}|v_*|^2-\frac{I_*}{2}} \si_{ij}{(|v-v_*|,\cos\theta,I,I_*,I^\prime,I^\prime_*)}\frac{|u|}{|u^\prime|^2}f_if_{j*}\notag\\
&\times\bm{\de}_3(V_{ij}-V'_{ij})\bm{\de}_1(\sqrt{|u|^2-\frac{2\De I}{\mu_{ij}}}-|u^\prime|)dZ_*dZ^\prime dZ^\prime_*,
\end{align*} 
and
\begin{align*}
	\Ga^+_{ij}(f,f)=&\int_{\CZ_j\times \CZ_i\times \CZ_j}W_{ij}(Z,Z_*|Z^\prime,Z^\prime_*)(\frac{M_i^\prime M_{j*}^\prime}{M_i})^{1/2}\frac{f_i^\prime f_{j*}^\prime}{( I^\prime_*)^{\de_i/2-1}( I^\prime_*)^{\de_j/2-1}}dv_*dv^\prime dv^\prime_* dI_*dI' dI^\prime_*\notag\\
	= &C\int_{\CZ_j\times \CZ_i\times \CZ_j}\frac{(I)^{\de_i/4-1/2}( I_*)^{\de_j/2-1}}{(I^\prime_*)^{\de_i/4-1/2}( I^\prime_*)^{\de_j/4-1/2}}e^{-\frac{m_j}{4}|v_*|^2-\frac{I_*}{2}} \si_{ij}{(|v-v_*|,\cos\theta,I,I_*,I^\prime,I^\prime_*)}\notag\\
	&\times\frac{|u|}{|u^\prime|^2}f_if_{j*}\bm{\de}_3(V_{ij}-V'_{ij})\bm{\de}_1(\sqrt{|u|^2-\frac{2\De I}{\mu_{ij}}}-|u^\prime|)dZ_*dZ^\prime dZ^\prime_*.
\end{align*} 
We consider $\Ga^\pm_{ij}$ individually for the four kinds of collisions. First for Poly-Poly collision, where $i,j>p$, apply the change of variables $(v^\prime,v^\prime_*)\rightarrow(u^\prime,V_{ij}^\prime)$ with \eqref{changeofvariables} and use \eqref{si} to obtain that
\begin{align}\label{Ga-1}
	\big|\Ga^-_{ij}(f,f)\big|\leq &C\int_{\S^2\times (\R_+)^2\times\R^3\times\R_+}(I_*)^{\de_j/4-1/2}e^{-\frac{m_j}{4}|v_*|^2-\frac{I_*}{2}} \si_{ij}{(|v-v_*|,\cos\theta,I,I_*,I^\prime,I^\prime_*)}|u||f_if_{j*}|d\om  dI^\prime dI^\prime_*dv_*dI_*\notag\\
	\leq &C\int_{\S^2\times (\R_+)^2\times\R^3\times\R_+}(I_*)^{\de_j/4-1/2}e^{-\frac{m_j}{4}|v_*|^2-\frac{I_*}{2}} \notag\\
	&\qquad\qquad\qquad\qquad\qquad\times\sqrt{|u|^2-\frac{2\De I}{\mu_{ij}}}\frac{(I^\prime )^{\de_i/2-1}(I^\prime_*)^{\de_j/2-1}}{E_{ij}^{(\eta+\de_i+\de_j)/2}}|f_if_{j*}|d\om  dI^\prime dI^\prime_*dv_*dI_*.
\end{align} 
Notice $E_{ij}=\frac{\mu_{ij}}{2}|u|^2+I+I_*=\frac{\mu_{ij}}{2}|u'|^2+I'+I'_*$ implies $I',I'_*\leq E_{ij}$. Then, we integrate in $I'$ and $I'_*$ to get that
\begin{align}\label{Ga-}
	|w_{i\be}\Ga^-_{ij}(f,f)|
	\leq &C\|w_{i\be}f_i\|_{L^\infty_{v,I}}\|w_{j\be}f_{j*}\|_{L^\infty_{v,I}}\int_{\R^3\times\R_+}(I_*)^{\de_j/4-1/2}e^{-\frac{m_j}{4}|v_*|^2-\frac{I_*}{2}}(|u|^2+I+I_*)^{(1-\eta)/2}dv_*dI_*\notag\\
	\leq &C(1+|v|^2+I)^{(1-\eta)/2}\|w_{i\be}f_i\|_{L^\infty_{v,I}}\|w_{j\be}f_{j*}\|_{L^\infty_{v,I}}.
\end{align} 
Similarly, we still apply the change of variables for $\Ga^+_{ij}$ to arrive at
\begin{align*}
	\big|\Ga^+_{ij}(f,f)\big|\leq &C\int_{\S^2\times (\R_+)^2\times\R^3\times\R_+}(I_*)^{\de_j/4-1/2}e^{-\frac{m_j}{4}|v_*|^2-\frac{I_*}{2}} \si_{ij}{(|v-v_*|,\cos\theta,I,I_*,I^\prime,I^\prime_*)}|u||f_if_{j*}|d\om  dI^\prime dI^\prime_*dv_*dI_*\notag\\
	\leq &C\int_{\S^2\times (\R_+)^2\times\R^3\times\R_+}(I)^{\de_i/4-1/2}( I_*)^{\de_j/2-1}(I^\prime_*)^{\de_i/4-1/2}( I^\prime_*)^{\de_j/4-1/2}e^{-\frac{m_j}{4}|v_*|^2-\frac{I_*}{2}}\notag\\
	&\qquad\qquad\times E^{1/2}_{ij}\frac{1}{E_{ij}^{(\eta+\de_i+\de_j)/2}}|f_if_{j*}|d\om  dI^\prime dI^\prime_*dv_*dI_*.
\end{align*} 
Integrate in $I',I'_*$ on $\{I',I'_*\leq E_{ij}\}$ to get that
\begin{align*}
	\big|\Ga^+_{ij}(f,f)\big|\leq &C\int_{\R^3\times\R_+}(I_*)^{\de_j/4-1/2}e^{-\frac{m_j}{4}|v_*|^2-\frac{I_*}{2}} \si_{ij}{(|v-v_*|,\cos\theta,I,I_*,I^\prime,I^\prime_*)}|u|f_if_{j*}d\om  dI^\prime dI^\prime_*dv_*dI_*\notag\\
	\leq &C\int_{\S^2\times (\R_+)^2\times\R^3\times\R_+}(I)^{\de_i/4-1/2}( I_*)^{\de_j/2-1}\notag\\
	&\qquad\qquad\qquad\qquad\times E_{ij}^{\de_i/4+1/2}E_{ij}^{\de_j/4+1/2}e^{-\frac{m_j}{4}|v_*|^2-\frac{I_*}{2}}E^{1/2}_{ij}\frac{1}{E_{ij}^{(\eta+\de_i+\de_j)/2}}|f_if_{j*}|dv_*dI_*.
\end{align*}
Using $(I)^{\de_i/4-1/2}(I_*)^{\de_j/4-1/2}\leq CE_{ij}^{\de_i/4+\de_j/4-1}$, one further gets that
\begin{align}\label{Ga+1}
	\big|\Ga^+_{ij}(f,f)\big|
	\leq &C\int_{\R^3\times\R_+}( I_*)^{\de_j/4-1/2}e^{-\frac{m_j}{4}|v_*|^2-\frac{I_*}{2}}(|u|^2+I+I_*)^{(1-\eta)/2}|f_if_{j*}|dv_*dI_*.
\end{align}
Combining the above inequality with the fact that
$$
w_{i\be}(v,I)\leq C(1+m_i|v|^2+2I)^{\frac{\be}{2}}\leq C(1+m_i|v^\prime|^2+2I^\prime+m_j|v^\prime_*|^2+2I^\prime_*)^\frac{\be}{2}\leq Cw_{i\be}(v^\prime,I^\prime)w_{j\be}(v^\prime_*,I^\prime_*),
$$
it holds that
\begin{align}\label{Ga+}
	|w_{i\be}\Ga^+_{ij}(f,f)|
	\leq &C\|w_{i\be}f_i\|_{L^\infty_{v,I}}\|w_{j\be}f_{j*}\|_{L^\infty_{v,I}}\int_{\R^3\times\R_+}(I_*)^{\de_j/4-1/2}e^{-\frac{m_j}{4}|v_*|^2-\frac{I_*}{2}}(|u|^2+I+I_*)^{(1-\eta)/2}dv_*dI_*\notag\\
	\leq &C(1+|v|^2+I)^{(1-\eta)/2}\|w_{i\be}f_i\|_{L^\infty_{v,I}}\|w_{j\be}f_{j*}\|_{L^\infty_{v,I}}.
\end{align} 
We collect \eqref{Ga-} and \eqref{Ga+} to conclude that
\begin{align}\label{Gapp}
	|w_{i\be}\Ga^{\pm}_{ij}(f,f)|\leq &C(1+|v|^2+I)^{(1-\eta)/2}\|w_{i\be}f_i\|_{L^\infty_{v,I}}\|w_{j\be}f_{j*}\|_{L^\infty_{v,I}}\leq C\nu_i(v,I)\sum^n_{j=1}\|(w_jf_j)(t)\|^2_{L^\infty_{v,I}}.
\end{align} 
The other cases can be calculated in a similar way. For Mono-Poly collision where $i\leq p$ and $j>p$, similar arguments as in \eqref{Ga-1}, \eqref{Ga-}, \eqref{Ga+1} and \eqref{Ga+} show that
\begin{align*}
	|w_{i\be}\Ga^-_{ij}(f,f)|\leq &C\|w_{i\be}f_i\|_{L^\infty_{v,I}}\|w_{j\be}f_{j*}\|_{L^\infty_{v,I}}\notag\\
	&\times\int_{\S^2\times \R_+\times\R^3\times\R_+}(I_*)^{\de_j/4-1/2}e^{-\frac{m_j}{4}|v_*|^2-\frac{I_*}{2}} \sqrt{|u|^2-\frac{2\De I}{\mu_{ij}}}\frac{(I^\prime_*)^{\de_j/2-1}}{E_{ij}^{(\eta+\de_j)/2}}d\om  dI^\prime_*dv_*dI_*\notag\\
	\leq &C\|w_{i\be}f_i\|_{L^\infty_{v,I}}\|w_{j\be}f_{j*}\|_{L^\infty_{v,I}}\int_{\R^3\times\R_+}(I_*)^{\de_j/4-1/2}e^{-\frac{m_j}{4}|v_*|^2-\frac{I_*}{2}}(|u|^2+I_*)^{(1-\eta)/2}dv_*dI_*\notag\\
	\leq &C(1+|v|^2)^{(1-\eta)/2}\|w_{i\be}f_i\|_{L^\infty_{v,I}}\|w_{j\be}f_{j*}\|_{L^\infty_{v,I}},
\end{align*} 
and
\begin{align*}
	|w_{i\be}\Ga^+_{ij}(f,f)|\leq &C\|w_{i\be}f_i\|_{L^\infty_{v,I}}\|w_{j\be}f_{j*}\|_{L^\infty_{v,I}}\notag\\
	&\times\int_{\S^2\times \R_+\times\R^3\times\R_+}( I_*)^{\de_j/2-1}E_{ij}^{\de_j/4+1/2}e^{-\frac{m_j}{4}|v_*|^2-\frac{I_*}{2}}E^{1/2}_{ij}\frac{1}{E_{ij}^{(\eta+\de_j)/2}}d\om  dI^\prime_*dv_*dI_*\notag\\
	\leq &C\|w_{i\be}f_i\|_{L^\infty_{v,I}}\|w_{j\be}f_{j*}\|_{L^\infty_{v,I}}\int_{\R^3\times\R_+}(I_*)^{\de_j/4-1/2}e^{-\frac{m_j}{4}|v_*|^2-\frac{I_*}{2}}(|u|^2+I_*)^{(1-\eta)/2}dv_*dI_*\notag\\
	\leq &C(1+|v|^2)^{(1-\eta)/2}\|w_{i\be}f_i\|_{L^\infty_{v,I}}\|w_{j\be}f_{j*}\|_{L^\infty_{v,I}}.
\end{align*}
It follows from the above two estimates that
\begin{align}\label{Gamp}
	|w_{i\be}\Ga^{\pm}_{ij}(f,f)|\leq &C(1+|v|^2)^{(1-\eta)/2}\|w_{i\be}f_i\|_{L^\infty_{v,I}}\|w_{j\be}f_{j*}\|_{L^\infty_{v,I}}\leq C\nu_i(v)\sum^n_{j=1}\|(w_jf_j)(t)\|^2_{L^\infty_{v,I}}.
\end{align} 
For Poly-Mono collision, as in \eqref{Ga-1}, \eqref{Ga-}, \eqref{Ga+1} and \eqref{Ga+}, one can deduce that
\begin{align*}
	|w_{i\be}\Ga^-_{ij}(f,f)|\leq &C\|w_{i\be}f_i\|_{L^\infty_{v,I}}\|w_{j\be}f_{j*}\|_{L^\infty_{v,I}}\int_{\S^2\times \R_+\times\R^3}e^{-\frac{m_j}{4}|v_*|^2} \sqrt{|u|^2-\frac{2\De I}{\mu_{ij}}}\frac{(I^\prime )^{\de_i/2-1}}{E_{ij}^{(\eta+\de_i)/2}}d\om  dI^\prime dv_*\notag\\
	\leq &C\|w_{i\be}f_i\|_{L^\infty_{v,I}}\|w_{j\be}f_{j*}\|_{L^\infty_{v,I}}\int_{\R^3}e^{-\frac{m_j}{4}|v_*|^2}(|u|^2+I)^{(1-\eta)/2}dv_*\notag\\
	\leq &C(1+|v|^2+I)^{(1-\eta)/2}\|w_{i\be}f_i\|_{L^\infty_{v,I}}\|w_{j\be}f_{j*}\|_{L^\infty_{v,I}},
\end{align*} 
and
\begin{align*}
	|w_{i\be}\Ga^+_{ij}(f,f)|\leq &C\|w_{i\be}f_i\|_{L^\infty_{v,I}}\|w_{j\be}f_{j*}\|_{L^\infty_{v,I}}\int_{\S^2\times \R_+\times\R^3}(I)^{\de_i/4-1/2}E_{ij}^{\de_i/4+1/2}e^{-\frac{m_j}{4}|v_*|^2}E^{1/2}_{ij}\frac{1}{E_{ij}^{(\eta+\de_i)/2}}dv_*\notag\\
	\leq &C\|w_{i\be}f_i\|_{L^\infty_{v,I}}\|w_{j\be}f_{j*}\|_{L^\infty_{v,I}}\int_{\R^3}e^{-\frac{m_j}{4}|v_*|^2}(|u|^2+I)^{(1-\eta)/2}dv_*\notag\\
	\leq &C(1+|v|^2+I)^{(1-\eta)/2}\|w_{i\be}f_i\|_{L^\infty_{v,I}}\|w_{j\be}f_{j*}\|_{L^\infty_{v,I}},
\end{align*}
which yield
\begin{align}\label{Gapm}
	|w_{i\be}\Ga^{\pm}_{ij}(f,f)|\leq &C(1+|v|^2+I)^{(1-\eta)/2}\|w_{i\be}f_i\|_{L^\infty_{v,I}}\|w_{j\be}f_{j*}\|_{L^\infty_{v,I}}\leq C\nu_i(v,I)\sum^n_{j=1}\|(w_jf_j)(t)\|^2_{L^\infty_{v,I}}.
\end{align} 
The Mono-Mono collision is more direct. We have
\begin{align*}
	|w_{i\be}\Ga^-_{ij}(f,f)|
	\leq &C\|w_{i\be}f_i\|_{L^\infty_{v,I}}\|w_{j\be}f_{j*}\|_{L^\infty_{v,I}}\int_{\R^3}e^{-\frac{m_j}{4}|v_*|^2}|u|^{1-\eta}dv_*\notag\\
	\leq &C(1+|v|^2)^{(1-\eta)/2}\|w_{i\be}f_i\|_{L^\infty_{v,I}}\|w_{j\be}f_{j*}\|_{L^\infty_{v,I}},
\end{align*} 
and
\begin{align*}
	|w_{i\be}\Ga^+_{ij}(f,f)|\leq &C\|w_{i\be}f_i\|_{L^\infty_{v,I}}\|w_{j\be}f_{j*}\|_{L^\infty_{v,I}}\int_{\S^2\times\R^3}e^{-\frac{m_j}{4}|v_*|^2}E^{1/2}_{ij}\frac{1}{E_{ij}^{\eta/2}}dv_*\notag\\
	\leq &C(1+|v|^2)^{(1-\eta)/2}\|w_{i\be}f_i\|_{L^\infty_{v,I}}\|w_{j\be}f_{j*}\|_{L^\infty_{v,I}},
\end{align*}
which together give that
\begin{align}\label{Gamm}
	|w_{i\be}\Ga^{\pm}_{ij}(f,f)|\leq &C(1+|v|^2)^{(1-\eta)/2}\|w_{i\be}f_i\|_{L^\infty_{v,I}}\|w_{j\be}f_{j*}\|_{L^\infty_{v,I}}\leq C\nu_i(v)\sum^n_{j=1}\|(w_jf_j)(t)\|^2_{L^\infty_{v,I}}.
\end{align} 
Hence, \eqref{LinftyGa} holds by \eqref{Gapp}, \eqref{Gamp}, \eqref{Gapm} and \eqref{Gamm}.
\end{proof}

The aforementioned lemma is the basis to prove Lemma \ref{lebound}.
\begin{proof}[Proof of Lemma \ref{lebound}]
For $f=f(t,x,v,I)$, the form of $\Ga_i$ in \eqref{defGa+-}, together with H\"older's inequality, shows that
\begin{align*}
	\|w_i\Ga_i(f,f)\|_{L^2_x}\leq w_i\Ga^+_i(\|f\|_{L^2_x},\|f\|_{L^\infty_x})+w_i\Ga^-_i(\|f\|_{L^2_x},\|f\|_{L^\infty_x}),
\end{align*}
and
\begin{align*}
	\|w_i\Ga_i(f,f)\|_{L^1_x}\leq w_i\Ga^+_i(\|f\|_{L^2_x},\|f\|_{L^2_x})+w_i\Ga^-_i(\|f\|_{L^2_x},\|f\|_{L^2_x}),
\end{align*}
which, combined with \eqref{LinftyGa}, yield \eqref{estGa}, \eqref{estGaL2} and \eqref{estGaL1}.
\end{proof}

\section{Linear $L^2$ Decay}\label{Section4}
In this section, we study the $L^2$ decay of solutions to the linear Boltzmann equation
\begin{eqnarray}\label{LBE}
	\pa_tf+v\cdot \na_x f+L f=0,   \quad &\dis f(t=0)=f_0.
\end{eqnarray}
Our approach is motivated by \cite{Duan} in the case of the classical Boltzmann equation in $\R^3$, see also \cite{Duan-Stain1,Guo}.  Thus, we carry out a Fourier transform based method in the gaseous mixture version to prove the decay estimates in both the microscopic and macroscopic parts corresponding to a hyperbolic-parabolic system.

To this end define the density, bulk velocity, and temperature, respectively, as
\begin{align*}
	a_i=&\left\{
	\begin{array}{rl}
		&\dis \int_{\R^3} f_i(t,x,v)dv,\quad i\leq p,
		\\[3mm]
		&\dis \int_{\R^3\times \R_+} f_i(t,x,v,I)dvdI,\quad i> p,
	\end{array} \right.\notag\\
b=&\sum_{i=1}^p\int_{\R^3} m_ivf_i(t,x,v)dv+\sum_{i=p+1}^n\int_{\R^3\times\R_+} m_ivf_i(t,x,v,I)dvdI=(b_{v_1},b_{v_2},b_{v_3}),\notag\\
c=&\sum_{i=1}^p\int_{\R^3} (m_i|v|^2-3)f_i(t,x,v)dv+\sum_{i=p+1}^n\int_{\R^3\times\R_+} (m_i|v|^2-3-I-\de_i)f_i(t,x,v,I)dvdI.
\end{align*}
In terms of the basis of $\ker L$ in \eqref{Defphi}, the macroscopic projection $P$ is given by
\begin{align*}
	Pf=\sum_{i=1}^n a_i\phi_i+b_{v_1} \phi_{n+1}+b_{v_2} \phi_{n+2}+b_{v_3} \phi_{n+3}+c\phi_{n+4}=\begin{pmatrix}
		P_1 f\\
		\vdots
		\\
		P_n f
\end{pmatrix}.
\end{align*}
Correspondingly, we denote
\begin{align*}
	(I-P)f=\begin{pmatrix}
		f_1-P_1 f\\
		\vdots
		\\
		f_n-P_n f
	\end{pmatrix}=\begin{pmatrix}
	(I-P)_1 f\\
	\vdots
	\\
	(I-P)_n f
\end{pmatrix}.
\end{align*}
From Lemma \ref{lenu} and Lemma \ref{leK}, the following proposition follows.
\begin{proposition}
	There exists a constant $\la_0>0$ depending on $\de_i$, $m_i$ and $\eta$ such that 
	\begin{align}\label{proco}
		(Lf,f)\geq C\la_0\|(I-P)f\|^2_\nu\geq \la_0\|(I-P)f\|_{L^2_{v,I}}^2,
	\end{align}	
where the weighted $L^2$ norm is defined by
\begin{align*}
\|(I-P)f\|^2_\nu=\sum_{i=1}^p\int_{\R^3}\nu_i(v) |f_i(t,x,v)|^2dv+\sum_{i=p+1}^n\int_{\R^3\times\R_+} \nu_i(v,I)|f_i(t,x,v,I)|^2dvdI.
\end{align*}
\end{proposition}

The following lemma is important to bound the $L^2$ norm of $Lf$.
\begin{lemma}\label{L2K}
	Let $K_{ij}$ be defined in \eqref{DefKij}. There exists a constant $C$ depending on $\de_i$, $m_i$ and $\eta$ such that
	\begin{align*}
		\|K_{ij}f(t,x)\|_{L^2_{v,I}}\leq C\|f(t,x)\|_{L^2_{v,I}},
	\end{align*}
	for $1\leq i,j\leq n.$
\end{lemma}
\begin{proof}
	Recall the general form of $K_{ij}$ in \eqref{reKij} and the derivation of $k^{1,2,3}_{ij}$ in \eqref{rek1ij}, \eqref{rek2ij} and \eqref{Gk3}. We use \eqref{ProW1} and rename $Z'$ and $Z'_*$ to get
	\begin{align*}
		k^1_{ij}(Z,Z_*)&=\int_{\CZ_i\times \CZ_j}\frac{1}{(M_iM_{j*})^{1/2}} \frac{ (M'_iM'_{j*})^{1/2}}{(II')^{\de_i/4-1/2}(I_*I'_*)^{\de_j/4-1/2}}W_{ij}(Z,Z_*|Z',Z'_*)dZ^\prime dZ^\prime_*\notag\\
		&=\int_{\CZ_i\times \CZ_j}\frac{1}{(M_iM_*)^{1/2}} \frac{ (M_iM_{j*}M'_iM'_{j*})^{1/2}}{(II')^{\de_i/4-1/2}(I_*I'_*)^{\de_j/4-1/2}}W_{ji}(Z_*,Z|Z',Z'_*)dZ^\prime dZ^\prime_*=k^1_{ji}(Z_*,Z).
	\end{align*}
	It holds by \eqref{ProW1}, \eqref{ProW2} and exchanging $Z'$ and $Z'_*$ that
	\begin{align*}
		k^2_{ij}(Z,Z_*)&=\int_{\CZ_i\times \CZ_j}\frac{1}{(M_iM_{j*})^{1/2}} \frac{ (M'_iM'_{j*})^{1/2}}{(II')^{\de_i/4-1/2}(I_*I'_*)^{\de_j/4-1/2}}W_{ij}(Z,Z'_*|Z',Z_*)dZ^\prime dZ^\prime_*\notag\\
		&=\int_{\CZ_i\times \CZ_j}\frac{1}{(M_iM_*)^{1/2}} \frac{ (M_iM_{j*}M'_iM'_{j*})^{1/2}}{(II')^{\de_i/4-1/2}(I_*I'_*)^{\de_j/4-1/2}}W_{ji}(Z'_*,Z|Z_*,Z')dZ^\prime dZ^\prime_*\notag\\
		&=\int_{\CZ_i\times \CZ_j}\frac{1}{(M_iM_*)^{1/2}} \frac{ (M_iM_{j*}M'_iM'_{j*})^{1/2}}{(II')^{\de_i/4-1/2}(I_*I'_*)^{\de_j/4-1/2}}W_{ji}(Z_*,Z'|Z'_*,Z)dZ^\prime dZ^\prime_*\notag\\
		&=\int_{\CZ_i\times \CZ_j}\frac{1}{(M_iM_*)^{1/2}} \frac{ (M_iM_{j*}M'_iM'_{j*})^{1/2}}{(II')^{\de_i/4-1/2}(I_*I'_*)^{\de_j/4-1/2}}W_{ji}(Z_*,Z'_*|Z',Z)dZ^\prime dZ^\prime_*=k^2_{ji}(Z_*,Z).
	\end{align*}
	For $k^3_{ij}$, noting that $Z,Z_*\in \CZ_i,$ one has from \eqref{ProW2} and renaming $Z'$ and $Z'_*$ that
	\begin{align*}
		k^3_{ij}(Z,Z_{*})&=\int_{\CZ_j\times \CZ_j}\frac{1}{(M_iM_{j*})^{1/2}} \frac{ (M'_iM'_{j*})^{1/2}}{(II')^{\de_i/4-1/2}(I_*I'_*)^{\de_j/4-1/2}}W_{ij}(Z,Z'|Z_*,Z'_*)dZ^\prime dZ^\prime_*\notag\\
		&=\int_{\CZ_j\times \CZ_j}\frac{1}{(M_iM_*)^{1/2}} \frac{ (M_iM_{j*}M'_iM'_{j*})^{1/2}}{(II')^{\de_i/4-1/2}(I_*I'_*)^{\de_j/4-1/2}}W_{ji}(Z_*,Z'|Z,Z'_*)dZ^\prime dZ^\prime_*=k^3_{ij}(Z_{*},Z).
	\end{align*}
For $i,j\leq p$ or $i,j>p$ which both indicate $\CZ_i=\CZ_j$, denoting $k(Z,Z_*)=k^2_{ij}(Z,Z_*)+k^3_{ij}(Z,Z_*)-k^1_{ij}(Z,Z_*),$ and using the symmetric properties deduced above, we have from \eqref{kmm}, \eqref{k2mm}, \eqref{k3mm} for Mono-Mono collisions and \eqref{kpp}, \eqref{k2pp}, \eqref{k3pp} for Poly-Poly collisions that
	\begin{align*}
		\int_{\CZ_i}|K_{ij}f(t,x,Z)|^2dZ&=\int_{\CZ_i}\left|\int_{\CZ_i}k(Z,Z_*)f(t,x,Z_*)dZ_*\right|^2dZ\notag\\
		&\leq C\int_{\CZ_i}\left(\int_{\CZ_i}k(Z,Z_*)dZ_*\right)\left(\int_{\CZ_i}k(Z,Z_*)\left|f(t,x,Z_*)\right|^2dZ_*\right)dvdI\notag\\
		&\leq C\int_{\CZ_i\times \CZ_i}k(Z,Z_*)\left|f(t,x,Z_*)\right|^2dZ_*dZ\notag\\
		&\leq C\|f(t,x)\|_{L^2_{v,I}}^2.
	\end{align*}
	All other cases are similarly deduced by using the above symmetric relations.
\end{proof}

In one hand, denoting the Fourier transform in $x$ of $f$ by $\hat{f}$, it is direct from \eqref{LBE} and \eqref{proco} that
\begin{align}\label{co}
	\frac{1}{2}\pa_t\|\hat{f}(t)\|^2_{L^2_{k,v,I}}+\la_0\|(I-P)\hat{f}(t)\|_{L^2_{k,v,I}}^2\leq0, 
\end{align}
for $\la_0>0$ given in \eqref{proco}. 

In the other hand, substituting $f=Pf+(I-P)f$ into \eqref{LBE}, we obtain that
\begin{align*}
	\pa_tP_if+v\cdot \na_x P_if+L(I-P)_if=-\pa_t(I-P)_if-v\cdot \na_x (I-P)_if.
\end{align*}
Taking inner product with $$\sqrt{M_i},\ v\sqrt{M_i},\ (m_i|v|^2-3)\sqrt{M_i}$$ for $i\leq p$ and $$(m_i|v|^2-3-2I-\de_i)\sqrt{M_i}$$ for $i> p$, one has that
\begin{align}\label{eqabc}
	\begin{cases}
	&\dis c_i\pa_ta_i+c_i\na_x\cdot b=0,\\[1mm]
	&\dis c_i\pa_t b+\frac{c_i}{m_i}\na_x a_i+\frac{2c_i}{m_i}\na_x c=-(v\cdot \na_x (I-P)_if,v\sqrt{M_i})\\[1mm]
	&\dis 6c_i\pa_t c+2c_i\na_x \cdot b=-(v\cdot \na_x (I-P)_if,(m_i|v|^2-3)\sqrt{M_i}),\quad i\leq p\\[1mm]
	&\dis (6+2\de_i)c_i\pa_t c+2c_i\na_x \cdot b=-(v\cdot \na_x (I-P)_if,(m_i|v|^2-3-2I-\de_i)\sqrt{M_i}),\quad i> p.
	\end{cases}
\end{align} 
Furthermore, we take inner product with $m_iv_jv_l\sqrt{M_i}$ and $m_iv|v|^2\sqrt{M_i}$ with $j\neq l$ to get
\begin{align}\label{eqhighmo}
	\begin{cases}
		&\dis c_i\pa_{x_j}b_{v_l}+c_i\pa_{x_l}b_{v_j}=-((\pa_t+v\cdot \na_x) (I-P)_if-L_i(I-P)f,m_iv_jv_l\sqrt{M_i}),\quad j\neq l,\\[1mm]
	&\dis 5c_i\pa_t b+c_i\na_x(\frac{5}{m_i}a+\frac{20}{m_i}c)=-((\pa_t+v\cdot \na_x) (I-P)_if-L_i(I-P)f,m_iv|v|^2\sqrt{M_i}).
\end{cases}
\end{align}
Using the momentum equations deduced above, we now prove Proposition \ref{leL2}, which provides the $L^2$ decay for the linearized equation.
\begin{proof}[Proof of Proposition \ref{leL2}]
	Let $\hat{a}$, $\hat{b}=(\hat{b}_{v_1},\hat{b}_{v_2},\hat{b}_{v_3})$, $\hat{c}$ and $\hat{f}$ denote the Fourier transform in $x$ of $a,b,c$ and $f$ respectively. Take Fourier transform on \eqref{eqabc} and \eqref{eqhighmo} to obtain the system in terms of $\hat{a}$, $\hat{b}$, $\hat{c}$ and $\hat{f}$
	\begin{align}\label{abcf}
		\begin{cases}
			&\dis c_i\pa_t\hat{a}_i+c_i \rmi k\cdot \hat{b}=0,\\[1mm]
			&\dis c_i\pa_t \hat{b}+\frac{c_i}{m_i}\rmi k \hat{a}_i+\frac{2c_i}{m_i}\rmi k \hat{c}=-(\rmi k\cdot v (I-P)_i\hat{f},v\sqrt{M_i})\\[1mm]
			&\dis 6c_i\pa_t \hat{c}+2c_i\rmi k\cdot \hat{b}=-(\rmi k\cdot v (I-P)_i\hat{f},(m_i|v|^2-3)\sqrt{M_i}),\quad i\leq p\\[1mm]
			&\dis (6+2\de_i)c_i\pa_t \hat{c}+2c_i\rmi k\cdot \hat{b}=-(\rmi k\cdot v (I-P)_i\hat{f},(m_i|v|^2-3-2I-\de_i)\sqrt{M_i}),\quad i> p\\[1mm]
			&\dis c_i\rmi k_j \hat{b}_{v_l}+c_i\rmi k_l\hat{b}_{v_j}=-((\pa_t+\rmi k\cdot v) (I-P)_i\hat{f}-L_i(I-P)\hat{f},m_iv_jv_l\sqrt{M_i}),\quad j\neq l,\\[1mm]
			&\dis \frac{10c_i}{m_i}\rmi k\hat{c}=-((\pa_t+\rmi k\cdot v) (I-P)_i\hat{f}-L_i(I-P)\hat{f},m_iv|v|^2\sqrt{M_i}-5v\sqrt{M_i}),
		\end{cases}
	\end{align}
	where $k\in \Z^3$ for $x\in \T^3$, and $k\in \R^3$ since $x\in \R^3$, and $1\leq i\leq n$. For convenience, we denote $p_{jli}(v,I)=\frac{1}{c_i}m_iv_iv_j\sqrt{M_i}$, $p_i(v,I)=\frac{m_i}{c_i}(m_i|v|^2-5)v\sqrt{M_i}.$ From the sixth equation in \eqref{abcf}, we integrate by parts to get that
	\begin{align}\label{c1}
		10|k|^2|\hat{c}|^2&=\langle\rmi k\hat{c},10\rmi k\hat{c}\rangle\notag\\
		&=-\langle\rmi k\hat{c},\pa_t ((I-P)_i\hat{f},p_i(v,I))\rangle-\langle\rmi k\hat{c},-(\mathrm{i}k\cdot v(I-P)_i\hat{f},p_i(v,I))-(L_i(I-P)\hat{f},p_i(v,I))\rangle\notag\\
		&=-\pa_{t}\langle\rmi k\hat{c}, ((I-P)_i\hat{f},p_i(v,I))\rangle+\langle\rmi k\pa_t\hat{c}, ((I-P)_i\hat{f},p_i(v,I))\rangle\notag\\
		&\qquad-\left\{\langle\rmi k\hat{c},-(\mathrm{i}k\cdot v(I-P)_i\hat{f},p_i(v,I))-(L_i(I-P)\hat{f},p_i(v,I))\rangle\right\}.
	\end{align}
	It holds by Cauchy-Schwarz inequality that
	\begin{align}\label{c2}
		\left|\langle\rmi k\hat{c},-(\mathrm{i}k\cdot v(I-P)_i\hat{f},p_i(v,I))\rangle\right|\leq \ep|k|^2|\hat{c}|^2+\frac{C}{\ep}(1+|k|^2)\|(I-P)\hat{f}\|_{L^2_{v,I}}^2,
	\end{align} 
	where $\eps>0$ will be chosen later and $C>0$ depends only on $\de_i$, $m_i$ and $\eta$.
	Since $p_i$ exponentially decays in $v$ for $i\leq p$ and in both $v$ and $I$ for $i> p$, one has from Lemma \ref{L2K} that
	\begin{align}\label{c3}
		\left|\langle\rmi k\hat{c},-(L_i(I-P)\hat{f},p_i(v,I))\rangle \right|\leq \ep|k|^2|\hat{c}|^2+\frac{C}{\ep}(1+|k|^2)\|(I-P)\hat{f}\|_{L^2_{v,I}}^2,
	\end{align} 
	where $C>0$ depends on $\de_i$, $m_i$ and $\eta$.
	Then, we use the third equation in \eqref{abcf} to get
	\begin{align*}
		&\langle\rmi k\pa_t\hat{c}, ((I-P)_i\hat{f},p_i(v,I))\rangle
		=-\langle\rmi k\big\{\frac{1}{3}\mathrm{i}k\cdot \hat{b}+\frac{1}{6}\mathrm{i}k\cdot ((I-P)_i\hat{f},\frac{(m_i|v|^2-3)}{c_i}\sqrt{M_i})\big\},((I-P)_i\hat{f},p_i(v,I))\rangle,	
	\end{align*}
	for $i\leq p$.
	A direct application of Cauchy-Schwarz inequality yields
	\begin{align}\label{c4}
		|\langle\rmi k_i\pa_t\hat{c}, ((I-P)_i\hat{f},p_i(v,I))\rangle|\leq \ep_1|k\cdot \hat{b}|^2+\frac{C}{\ep_1}(1+|k|^2)\|(I-P)\hat{f}\|_{L^2_{v,I}}^2,
	\end{align}
	for $i\leq p$, where $\eps_1>0$ will be chosen later and $C>0$ depends only on $\de_i$, $m_i$ and $\eta$. Similarly, \eqref{c4} holds for $i> p$ by the fourth equation in \eqref{abcf}. It following from \eqref{c1}, \eqref{c2}, \eqref{c3}, \eqref{c4} and choosing $\ep$ to be small enough such that 
	\begin{align}\label{estc}
		\pa_{t}\rmre\sum^3_{i=1}\langle\rmi k\hat{c}, ((I-P)_i\hat{f},p_i(v,I))\rangle+\la|k|^2|\hat{c}|^2 \leq \ep_1|k\cdot \hat{b}|^2+\frac{C}{\ep_1}(1+|k|^2)\|(I-P)\hat{f}\|_{L^2_{v,I}}^2,
	\end{align}
	where $\la>0$ is a generic constant and $C>0$ depends on $\de_i$, $m_i$ and $\eta$.
	For the estimate of $\hat{b}$, we use the fifth equation in \eqref{abcf} to obtain
	\begin{align}\label{b1}
		\sum^3_{l,j=1}|\rmi k_l\hat{b}_{v_j}+\rmi k_j\hat{b}_{v_l}|^2&=\sum^3_{i,j=1}\langle\rmi k_l\hat{b}_{v_j}+\rmi k_j\hat{b}_{v_l},-\pa_t  ((I-P)_i\hat{f},p_{lji}(v,I)) \rangle\notag\\
		&\qquad\qquad-\sum^3_{i,j=1}\langle(\rmi k\cdot v(I-P)_i\hat{f},p_{lji}(v,I))-(L_i(I-P)\hat{f},p_{lji}(v,I))\rangle.
	\end{align}
	Similar arguments as in \eqref{c2} and \eqref{c3} yield
	\begin{align}\label{b2}
		\Big|-\sum^3_{i,j=1}&\langle(\rmi k\cdot v(I-P)_i\hat{f},p_{lji}(v,I))-(L_i(I-P)\hat{f},p_{lji}(v,I))\rangle\Big|\notag\\
		&\leq \ep\,|k|^2|\hat{b}|^2+\frac{C}{\ep}(1+|k|^2)\|(I-P)\hat{f}\|_{L^2_{v,I}}^2.
	\end{align}
Integration by parts in the variable $t$ gives
	\begin{align}\label{b3}
		\langle\rmi k_l\hat{b}_{v_j}& + \rmi k_j\hat{b}_{v_l},-\pa_t  ((I-P)_i\hat{f},p_{lji}(v,I))\rangle\notag\\
		=&-\pa_t 	\langle\rmi k_l\hat{b}_{v_j}+\rmi k_j\hat{b}_{v_l},((I-P)_i\hat{f},p_{lji}(v,I)) \rangle+	\langle\pa_t (\rmi k_l\hat{b}_{v_j}+\rmi k_j\hat{b}_{v_l}),((I-P)_i\hat{f},p_{lji}(v,I)) \rangle.
	\end{align}
	For the second term on the right hand side, one uses the second equation in \eqref{abcf} and Cauchy-Schwarz inequality to conclude that
	\begin{align}\label{b4}
		\Big|\langle\pa_t (\rmi k_l\hat{b}_{v_j}&+\rmi k_j\hat{b}_{v_l}), ((I-P)_i\hat{f},p_{lji}(v,I))\rangle\Big|\notag\\
		&\leq C|k|\left|\langle \frac{1}{m_i} \mathrm{i}k_i(\hat{a}+2\hat{c})+\frac{1}{c_i}\mathrm{i}k\cdot((I-P)_i\hat{f},v_lv\sqrt{M}), ((I-P)_i\hat{f},p_{lji}(v,I))\rangle\right|\notag\\
		&\qquad+C|k|\left|\langle \frac{1}{m_i} \mathrm{i}k_j(\hat{a}+2\hat{c})+\frac{1}{c_i}\mathrm{i}k\cdot((I-P)_i\hat{f},v_jv\sqrt{M}), ((I-P)_i\hat{f},p_{lji}(v,I))\rangle\right|\notag\\
		&\leq \ep_2|k|^2|\hat{a}|^2+\frac{C}{\ep_2}|k|^2|\hat{c}|^2+\frac{C}{\ep_2}|k|^2\|(I-P)\hat{f}\|_{L^2_{v,I}}^2,
	\end{align}
	where $\eps_2>0$ and the constant $C>0$ depends on $\de_i$, $m_i$ and $\eta$.
	Combining \eqref{b1}, \eqref{b2}, \eqref{b3}, \eqref{b4} with the fact that
		\begin{align*}
		\sum^3_{l,j=1}|\rmi k_l\hat{b}_{v_j}+\rmi k_j\hat{b}_{v_l}|^2=2|k|^2|\hat{b}|^2+2|k\cdot \hat{b}|^2,
	\end{align*}
	and then choosing $\ep>0$ sufficiently small, it holds that
	\begin{align}\label{estb}
		\pa_t \rmre	&\sum^3_{l,j=1}\langle\rmi k_l\hat{b}_{v_j}+\rmi k_j\hat{b}_{v_l}, ((I-P)_i\hat{f},p_{lji}(v,I))\rangle+\la|k|^2|\hat{b}|^2 \notag\\
		&\leq \ep_2|k|^2|\hat{a}|^2+\frac{C}{\ep_2}|k|^2|\hat{c}|^2+\frac{C}{\ep_2}|k|^2\|(I-P)\hat{f}\|_{L^2_{v,I}}^2.
	\end{align}
	We now focus on $\hat{a}$. The second equation in \eqref{abcf} shows that
	\begin{align}\label{a1}
		|k|^2|\hat{a}|^2&=\langle\rmi k\hat{a},\rmi k\hat{a}\rangle
		=-\langle\rmi k\hat{a},m_i\partial_{t}\hat{b}+2\mathrm{i}k\hat{c}+\frac{m_i}{c_i}\mathrm{i}k\cdot v((I-P)_i\hat{f},v\sqrt{M_i})\rangle\notag\\
		&=-m_i\pa_{t}\langle\rmi k\hat{a},\hat{b}\rangle+m_i\langle\rmi k\pa_{t}\hat{a},\hat{b}\rangle-\langle\rmi k\hat{a},2\mathrm{i}k\hat{c}+\frac{m_i}{c_i}\mathrm{i}k\cdot v((I-P)_i\hat{f},v\sqrt{M_i})\rangle.
	\end{align}
	An application of Cauchy-Schwarz inequality yields 
	\begin{align}\label{a2}
		\left|-\langle\rmi k\hat{a},2\mathrm{i}k\hat{c}+\frac{m_i}{c_i}\mathrm{i}k\cdot v((I-P)_i\hat{f},v\sqrt{M_i})\rangle\right|\leq \ep|k|^2|\hat{a}|^2+\frac{C}{\ep}|k|^2|\hat{c}|^2+\frac{C}{\ep}|k|^2\|(I-P)\hat{f}\|_{L^2_{v,I}}^2,\quad \ep>0.
	\end{align}
	From the first equation in \eqref{abcf}, we obtain
	\begin{align}\label{a3}
		|\langle\rmi k\pa_{t}\hat{a},\hat{b}\rangle|=|\langle\rmi k(\mathrm{i}k\cdot \hat{b}),\hat{b}\rangle|=|k\cdot\hat{b}|^2.
	\end{align}
	By \eqref{a1}, \eqref{a2} and \eqref{a3}, we choose a sufficiently small $\ep$ to get that
	\begin{align}\label{esta}
		\pa_{t}\rmre \langle\rmi k\hat{a},\hat{b}\rangle+\la|k|^2|\hat{a}|^2\leq |k\cdot\hat{b}|^2+C|k|^2|\hat{c}|^2+C|k|^2\|(I-P)\hat{f}\|_{L^2_{v,I}}^2.
	\end{align}
By taking suitable linear combinations of \eqref{estb} and \eqref{esta} such that the $|k\cdot\hat{b}|^2$ on the right hand side of \eqref{esta} can be absorbed by $\la|k|^2|\hat{b}|^2$ on the left hand side of \eqref{estb}, it holds that
\begin{align*}
	\pa_{t}\{\rmre	\sum^3_{l,j=1}&\langle\rmi k_l\hat{b}_{v_j}+\rmi k_j\hat{b}_{v_l}, ((I-P)_i\hat{f},p_{lji}(v,I))\rangle+\rmre \langle\rmi k\hat{a},\hat{b}\rangle\}+\la|k|^2|\hat{a}|^2+\la|k|^2|\hat{b}|^2\notag\\
	&\leq C\ep_2|k|^2|\hat{a}|^2+\frac{C}{\ep_2}|k|^2|\hat{c}|^2+\frac{C}{\ep_2}|k|^2\|(I-P)\hat{f}\|_{L^2_{v,I}}^2,
\end{align*}
with $C>0$ depending on $\de_i$, $m_i$ and $\eta$. Then, we choose $\ep_2>0$ sufficiently small such that
\begin{align*}
	\pa_{t}\{\rmre	\sum^3_{l,j=1}&\langle\rmi k_l\hat{b}_{v_j}+\rmi k_j\hat{b}_{v_l}, ((I-P)_i\hat{f},p_{lji}(v,I))\rangle+\rmre \langle\rmi k\hat{a},\hat{b}\rangle\}+\la|k|^2|\hat{a}|^2+\la|k|^2|\hat{b}|^2\notag\\
	\leq&C|k|^2|\hat{c}|^2+C|k|^2\|(I-P)\hat{f}\|_{L^2_{v,I}}^2.
\end{align*}
We multiply previous inequality with a small number and take summation with \eqref{estc} such that the $C|k|^2|\hat{c}|^2$ term above can be absorbed by $\la|k|^2|\hat{c}|^2$ on the left hand side of \eqref{estc}. We obtain that
	\begin{align*}
	\pa_t \sum^3_{l,j=1}\langle\rmi k_l\hat{b}_{v_j}&+\rmi k_j\hat{b}_{v_l}, ((I-P)_i\hat{f},p_{jli}(v,I))\rangle+\langle\rmi k\hat{a},\hat{b}\rangle+\langle\rmi k\hat{c}, ((I-P)_i\hat{f},p_i(v,I))\rangle\notag\\
	&+\la|k|^2(|\hat{a}|^2+|\hat{b}|^2+|\hat{c}|^2) \leq \ep_1|k\cdot \hat{b}|^2+\frac{C}{\ep_1}(1+|k|^2)\|(I-P)\hat{f}\|_{L^2_{v,I}}^2,
\end{align*}
	By choosing $\ep_1>0$ sufficiently small, it holds that
	\begin{align}\label{estabc}
		&\pa_t \rmre\,  \CE^{int}(\hat{f}_i)+\la\frac{|k|^2}{1+|k|^2}(|\hat{a}|^2+|\hat{b}|^2+|\hat{c}|^2)\leq C\|(I-P)\hat{f}\|_{L^2_{v,I}}^2,
	\end{align}
	where
	\begin{align*}
		\CE^{int}(\hat{f}_i):=&\sum^3_{l,j=1}\langle\rmi \frac{k_l\hat{b}_{v_j}+\rmi k_j\hat{b}_{v_l}}{1+|k|^2}, ((I-P)_i\hat{f},p_{jli}(v,I))\rangle+\langle\rmi \frac{k}{1+|k|^2}\hat{a},\hat{b}\rangle+\langle\rmi \frac{k}{1+|k|^2}\hat{c}, ((I-P)_i\hat{f},p_i(v,I))\rangle.
	\end{align*}
	Consequently,
\begin{align}\label{EL2}
	\left|\rmre\,\CE^{int}(\hat{f}_i)\right|\leq C\|\hat{f}\|_{L^2_{v,I}}.
	\end{align}
	Then, we can choose $\ep$ sufficiently small such that
	\begin{align*}
		\CE(\hat{f}):=\|\hat{f}\|_{L^2_{v,I}}^2+\ep\rmre\,\CE^{int}(\hat{f}_i)\sim\|\hat{f}\|_{L^2_{v,I}}^2,
	\end{align*}
	which, combined with \eqref{estabc} and \eqref{co}, yields that
	\begin{align}\label{estCE}
	\pa_{t}\CE(\hat{f}(t))+\la\frac{|k|^2}{1+|k|^2}\CE(\hat{f}(t))\leq 0,
\end{align}
where $\la>0$ depends on $\de_i$, $m_i$ and $\eta$.
Then it holds that
\begin{align*}
\CE(\hat{f}(t))\leq e^{-\la \frac{|k|^2}{1+|k|^2}t}\CE(\hat{f}(0)).
\end{align*}
Suppose $x\in \R^3$. By the equivalence relation \eqref{EL2} and the above inequality, we have that
\begin{align*}
	\|f(t)\|^2&\leq C\int_{\R^3}e^{-\la \frac{|k|^2}{1+|k|^2}t}\|\hat{f}(0)\|_{L^2_{v,I}}^2 dk\notag\\
	&\leq C\int_{|k|\leq 1}e^{-\la \frac{|k|^2}{1+|k|^2}t}\|\hat{f}(0)\|_{L^2_{v,I}}^2 dk+C\int_{|k|> 1}e^{-\la \frac{|k|^2}{1+|k|^2}t}\|\hat{f}(0)\|_{L^2_{v,I}}^2 dk\notag\\
	&\leq C\|f(0)\|_{L^2_{v,I}L^1_x}^2 \int_{|k|\leq 1}e^{-\la |k|^2t}dk+Ce^{-\frac{\la}{2} t}\int_{|k|> 1}\|\hat{f}(0)\|_{L^2_{v,I}}^2 dk\notag\\
	&\leq C(1+t)^{-\frac{3}{2}}\|f(0)\|_{L^2_{v,I}L^1_x}^2 +Ce^{-\frac{\la}{2} t}\|f(0)\|^2,
\end{align*}
which shows \eqref{L2decayR3}. If $x\in \T^3$, then the conservation laws imply 
	\begin{align*}
		(\int_{\T^3}a(t,x)dx,\int_{\T^3}b(t,x)dx,\int_{\T^3}c(t,x)dx)=0,
	\end{align*}
provided that $(M_0, J_0, E_0) = (0,0,0)$, which, combined with \eqref{estCE}, gives
	\begin{align*}
		\pa_{t}\CE(\hat{f}(t))+\frac{\la}{2}\CE(\hat{f}(t))\leq 0.
	\end{align*}
	 Then we obtain \eqref{L2decaytorus} by renaming $\la/2$ to be $\la$.
\end{proof}

\section{Linear $L^\infty$ Decay}\label{Section5}
In this section we consider the linear problem \eqref{LBE} in the weighted $L^\infty$ framework via the Duhamel's iteration technique as well as its interplay with $L^2$ time decay. Such iteration approach is carried out in \cite{Guo} to study the boundary problem of Boltzmann equation, see also \cite{UY} for a multi-iteration method on $\R^3$ and \cite{BG} for Mono-Mono mixture on torus. Here, after tracking an additional collisional event, due to the different structures of $K^{1,2,3}_{ij}$ we are forced to discuss eight cases of double collisions (Poly/Mono-Poly/Mono-Poly/Mono), choose the appropriate weight functions, and design the partition of integral regions for each case. 
\begin{proof}[Proof of Proposition \ref{lelinearLinfty}]
	Recall from \eqref{mildh} that
		\begin{align*}
		h_i(t,x,Z)&=e^{-\nu_i(Z)t}h_{i0}(x-vt,Z)\notag\\
		&\qquad+\sum^n_{j=1}\int_0^t e^{-\nu_i(Z)(t-s)}\int_{\CZ_j}(k_{ij}^2-k^1_{ij})(Z,Z_*)\frac{w_{i\be}(Z)}{w_{j\be}(Z_*)}h_j(s,x-v(t-s),Z_*)dZ_* ds\notag\\
		&\qquad-\sum^n_{j=1}\int_0^t e^{-\nu_i(Z)(t-s)}\int_{\CZ_i}k_{ij}^3(Z,Z_*)\frac{w_{i\be}(Z)}{w_{i\be}(Z_*)}h_i(s,x-v(t-s),Z_*)dZ_* ds.
	\end{align*}
Denoting $x_1=x-v(t-s),$ one more iteration on $h_j$ and $h_i$ on the right hand side shows
	\begin{align}\label{hi}
	h_i(t,x,Z)&=I_1+\sum^n_{j=1}(I_{2j}+I_{3j})+\sum^n_{j,l=1}(I_{4jl}+I_{5jl}+I_{6jl}+I_{7jl}),
	\end{align}
	where
	\begin{align}\label{Ijl}
	I_1&=e^{-\nu_i(Z)t}h_{i0}(x-vt,Z),\notag\\
	I_{2j}&=\int_0^t e^{-\nu_i(Z)(t-s)}\int_{\CZ_j}(k_{ij}^2-k^1_{ij})(Z,Z_*)\frac{w_{i\be}(Z)}{w_{i\be}(Z_*)}e^{-\nu_i(Z_*)s}h_{i0}(x_1-v_* s,Z_*) dZ_* ds,\notag\\
	I_{3j}&=-\int_0^t e^{-\nu_i(Z)(t-s)}\int_{\CZ_i}k_{ij}^3(Z,Z_*)\frac{w_{i\be}(Z)}{w_{i\be}(Z_*)}e^{-\nu_i(Z_*)s}h_{i0}(x_1-v_* s,Z_*)dZ_* ds,\notag\\
	I_{4jl}&=\int_0^t e^{-\nu_i(Z)(t-s)}\int_{\CZ_j}(k_{ij}^2-k^1_{ij})(Z,Z_*)\frac{w_{i\be}(Z)}{w_{j\be}(Z_*)}\notag\\
	&\ \qquad\times\int_0^s e^{-\nu_i(Z_*)(s-s_1)}\int_{\CZ_l}(k_{jl}^2-k^1_{jl})(Z_*,Z_{**})\frac{w_{j\be}(Z_*)}{w_{l\be}(Z_{**})}h_l(s_1,x_1-v_{**}(s-s_1),Z_{**})dZ_{**} ds_1 dZ_*ds, \notag\\
	I_{5jl}&=-\int_0^t e^{-\nu_i(Z)(t-s)}\int_{\CZ_j}(k_{ij}^2-k^1_{ij})(Z,Z_*)\frac{w_{i\be}(Z)}{w_{j\be}(Z_*)}\notag\\
	&\ \qquad\times\int_0^s e^{-\nu_i(Z_*)(s-s_1)}\int_{\CZ_j}k_{jl}^3(Z_*,Z_{**})\frac{w_{j\be}(Z_*)}{w_{j\be}(Z_{**})}h_j(s_1,x_1-v_{**}(s-s_1),Z_{**})dZ_{**} ds_1 dZ_*ds,\notag\\
	I_{6jl}&=-\int_0^t e^{-\nu_i(Z)(t-s)}\int_{\CZ_i}k^3_{ij}(Z,Z_*)\frac{w_{i\be}(Z)}{w_{i\be}(Z_*)}\notag\\
	&\ \qquad\times\int_0^s e^{-\nu_i(Z_*)(s-s_1)}\int_{\CZ_l}(k_{il}^2-k_{il}^1)(Z_*,Z_{**})\frac{w_{i\be}(Z_*)}{w_{l\be}(Z_{**})}h_l(s_1,x_1-v_{**}(s-s_1),Z_{**})dZ_{**}  ds_1 dZ_*ds,\notag\\
	I_{7jl}&=\int_0^t e^{-\nu_i(Z)(t-s)}\int_{\CZ_i}k^3_{ij}(Z,Z_*)\frac{w_{i\be}(Z)}{w_{i\be}(Z_*)}\notag\\
	&\ \qquad\times\int_0^s e^{-\nu_i(Z_*)(s-s_1)}\int_{\CZ_i}k_{il}^3(Z_*,Z_{**})\frac{w_{i\be}(Z_*)}{w_{i\be}(Z_{**})}h_i(s_1,x_1-v_{**}(s-s_1),Z_{**})dZ_{**} ds_1 dZ_*ds.
\end{align}
We apply $L^2_x\cap L^\infty_x$ norm on \eqref{hi} to get 
\begin{align}\label{L2hi}
	\|h_i(t,Z)\|_{L^2_x\cap L^\infty_x}=&\|I_1\|_{L^2_x\cap L^\infty_x}+\sum^n_{j=1}(\|I_{2j}\|_{L^2_x\cap L^\infty_x}+\|I_{3j}\|_{L^2_x\cap L^\infty_x})\notag\\
	&+\sum^n_{j,l=1}(\|I_{4jl}\|_{L^2_x\cap L^\infty_x}+\|I_{5jl}\|_{L^2_x\cap L^\infty_x}+\|I_{6jl}\|_{L^2_x\cap L^\infty_x}+\|I_{7jl}\|_{L^2_x\cap L^\infty_x}).
\end{align}
We now bound all terms above. It is direct to get from \eqref{nu} that
\begin{align}\label{I1}
	\|I_1\|_{L^2_x\cap L^\infty_x}\leq Ce^{-\nu_0 t}\|h_{i0}\|_{L^\infty_{v,I}(L^1_x\cap L^\infty_x)}.
\end{align}
The boundedness of $\int k^{1,2,3}_{ij}$ shown in Lemma \ref{leK}, Lemma \ref{leKpm}, Lemma \ref{leKmp} and Lemma \ref{leKmm} gives
\begin{align}\label{I23}
	\|I_{2j}\|_{L^2_x\cap L^\infty_x}+\|I_{3j}\|_{L^2_x\cap L^\infty_x}\leq Ce^{-\nu_0 t}t\|h_{i0}\|_{L^\infty_{v,I}(L^1_x\cap L^\infty_x)}.
\end{align}
For $I_{4jl}$, we only need to control the $L^2_x\cap L^\infty_x$ norm of
\begin{align}\label{DefI4jl}
	I^{\al_1,\al_2}_{4jl}:=&\Big|\int_0^t e^{-\nu_i(Z)(t-s)}\int_{\CZ_j}k_{ij}^{\al_1}(Z,Z_*)\frac{w_{i\be}(Z)}{w_{j\be}(Z_*)}\notag\\
	&\quad\times\int_0^s e^{-\nu_i(Z_*)(s-s_1)}\int_{\CZ_l}k_{jl}^{\al_2}(Z_*,Z_{**})\frac{w_{j\be}(Z_*)}{w_{l\be}(Z_{**})}h_l(s_1,x_1-v_{**}(s-s_1),Z_{**})dZ_{**}  ds_1 dZ_*ds\Big|,
\end{align}
with $\al_1,\al_2=1,2$.
We split the above integral into five cases.

\medskip
\noindent{\it Case 1. } When
\begin{align*}
	\begin{cases}
		|v|>N,\quad \it{if}\ i\leq p\\
		|v|>N\ \it{or}\ I>N,\quad \it{if}\ i>p,
	\end{cases}
\end{align*}
it follows from Lemma \ref{leK} to Lemma \ref{leKmm} that
\begin{align}\label{case1}
	\|I^{\al_1,\al_2}_{4jl}\|_{L^2_x\cap L^\infty_x}\leq&C(1+t)^{-\frac{3}{4}}\sup_{0\leq t<\infty}\|(1+t)^{\frac{3}{4}}h_l(t)\|_{L^\infty_{v,I}(L^2_x\cap L^\infty_x)}\notag\\
	&\qquad\qquad\times \int_{\CZ_j}k_{ij}^{\al_1}(Z,Z_*)\frac{w_{i\be}(Z)}{w_{j\be}(Z_*)}\int_{\CZ_l}k_{jl}^{\al_2}(Z_*,Z_{**})\frac{w_{j\be}(Z_*)}{w_{l\be}(Z_{**})}dZ_{**}dZ_*\notag\\
	\leq&C(1+t)^{-\frac{3}{4}}\sup_{0\leq t<\infty}\|(1+t)^{\frac{3}{4}}h_l(t)\|_{L^\infty_{v,I}(L^2_x\cap L^\infty_x)} \int_{\CZ_j}k_{ij}^{\al_1}(Z,Z_*)\frac{w_{i\be}(Z)}{w_{j\be}(Z_*)}dZ_*\notag\\
	\leq&\frac{C(1+t)^{-\frac{3}{4}}}{N^{1/8}}\sup_{0\leq t<\infty}\|(1+t)^{\frac{3}{4}}h_l(t)\|_{L^\infty_{v,I}(L^2_x\cap L^\infty_x)}.
\end{align}
In the first line above, we observe that the $L^2_x\cap L^\infty_x$ only acts on the spacial variable in $h_l(s_1,x_1-v_{**}(s-s_1),Z_{**})$ on the right hand side of \eqref{DefI4jl}. Thus, we use the property
\begin{align*}
	\int_0^t e^{-\nu_i(Z)(t-s)}\int_0^s e^{-\nu_i(Z_*)(s-s_1)}(1+s_1)^{-\frac{3}{4}}ds_1 ds\leq C(1+t)^{-\frac{3}{4}}.
\end{align*}
\noindent{\it Case 2. }
\begin{align*}
	\begin{cases}
		|v-v_*|>N,\quad \it{if}\ \al_1=1,\\
		|m_i v-m_j v_*|>N,\quad \it{if}\ \al_1=2,
	\end{cases}\quad \text{or} \qquad	\begin{cases}
	|v_*-v_{**}|>N,\quad \it{if}\ \al_2=1,\\
	|m_jv_*-m_lv_{**}|>N,\quad \it{if}\ \al_2=2.
	\end{cases}
\end{align*}
It follows from Lemma \ref{leK} to Lemma \ref{leKmm} that
\begin{align}\label{case2}
	\|I^{\al_1,\al_2}_{4jl}\|_{L^2_x\cap L^\infty_x}\leq&C\sup_{0\leq t<\infty}\|(1+t)^{-\frac{3}{4}}h_l(t)\|_{L^\infty_{v,I}(L^2_x\cap L^\infty_x)}\notag\\
	&\times \int_{\CZ_j}k_{ij}^{\al_1}(Z,Z_*)\frac{w_{i\be}(Z)}{w_{j\be}(Z_*)}\int_{\CZ_l}k_{jl}^{\al_2}(Z_*,Z_{**})\frac{w_{j\be}(Z_*)}{w_{l\be}(Z_{**})}dZ_{**}dZ_*\notag\\
	\leq&C(1+t)^{-\frac{3}{4}}\sup_{0\leq t<\infty}\|(1+t)^{\frac{3}{4}}h_l(t)\|_{L^\infty_{v,I}(L^2_x\cap L^\infty_x)}\notag\\
	&\times \int_{\CZ_j}k_{ij}^{\al_1}(Z,Z_*)\frac{w_{i\be}(Z)}{w_{j\be}(Z_*)}\{e^\frac{\de|v-v_*|^2}{64}e^\frac{-\de|v-v_*|^2}{64}\chi_{\{\al_1=1\}}+e^\frac{\de|m_iv-m_jv_*|^2}{64}e^\frac{-\de|m_iv-m_jv_*|^2}{64}\chi_{\{\al_1=2\}}\}\notag\\
	&\quad\times\int_{\CZ_l}k_{jl}^{\al_2}(Z_*,Z_{**})\frac{w_{j\be}(Z_*)}{w_{l\be}(Z_{**})}\notag\\
	&\qquad\times\{e^\frac{\de|v_*-v_{**}|^2}{64}e^\frac{-\de|v_*-v_{**}|^2}{64}\chi_{\{\al_2=1\}}+e^\frac{\de|m_jv_*-m_lv_{**}|^2}{64}e^\frac{-\de|m_jv_*-m_lv_{**}|^2}{64}\chi_{\{\al_2=2\}}\}dZ_{**}dZ_*\notag\\
	\leq&\frac{C(1+t)^{-\frac{3}{4}}}{N^{1/8}}\sup_{0\leq t<\infty}\|(1+t)^{\frac{3}{4}}h_l(t)\|_{L^\infty_{v,I}(L^2_x\cap L^\infty_x)}.
\end{align}
\noindent{\it Case 3. } When
\begin{align*}
		\begin{cases}
		I_*>N\ \it{or}\ I_{**}>N,\quad \it{if}\ j,l>p\\
		I_*>N,\quad \it{if}\ j>p,\ l\leq p\\
		I_{**}>N,\quad \it{if}\ j\leq p,\ l>p,
	\end{cases}
\end{align*}
it holds by Lemma \ref{leK} to Lemma \ref{leKmm} that
\begin{align*}
	\|I^{\al_1,\al_2}_{4jl}\|_{L^2_x\cap L^\infty_x}\leq&C\sup_{0\leq t<\infty}\|(1+t)^{-\frac{3}{4}}h_l(t)\|_{L^\infty_{v,I}(L^2_x\cap L^\infty_x)}\notag\\
	&\qquad\qquad\times \int_{\CZ_j}k_{ij}^{\al_1}(Z,Z_*)\frac{w_{i\be}(Z)}{w_{j\be}(Z_*)}\frac{(1+I_*)^{1/8}}{(1+I_*)^{1/8}}\int_{\CZ_l}k_{jl}^{\al_2}(Z_*,Z_{**})\frac{w_{j\be}(Z_*)}{w_{l\be}(Z_{**})}dZ_{**}dZ_*\notag\\
	\leq&C\sup_{0\leq t<\infty}\|(1+t)^{\frac{3}{4}}h_l(t)\|_{L^\infty_{v,I}(L^2_x\cap L^\infty_x)}\notag\\
	&\qquad\qquad\times\int_{\CZ_j}\{k_{ij}^{\al_1}(Z,Z_*)\frac{w_{i\be}(Z)}{w_{j\be}(Z_*)}(1+I_*)^{1/8}\}\frac{\chi_{\{I_*>N\}}}{(1+I_*)^{1/8}}dZ_*\notag\\
	\leq&\frac{C(1+t)^{-\frac{3}{4}}}{N^{1/8}}\sup_{0\leq t<\infty}\|(1+t)^{\frac{3}{4}}h_l(t)\|_{L^\infty_{v,I}(L^2_x\cap L^\infty_x)},
\end{align*}
for $j>p$, $l\leq p$, 
\begin{align*}
	\|I^{\al_1,\al_2}_{4jl}\|_{L^2_x\cap L^\infty_x}\leq&C\sup_{0\leq t<\infty}\|(1+t)^{-\frac{3}{4}}h_l(t)\|_{L^\infty_{v,I}(L^2_x\cap L^\infty_x)}\notag\\
	&\qquad\qquad\times \int_{\CZ_j}k_{ij}^{\al_1}(Z,Z_*)\frac{w_{i\be}(Z)}{w_{j\be}(Z_*)}\int_{\CZ_l}k_{jl}^{\al_2}(Z_*,Z_{**})\frac{w_{j\be}(Z_*)}{w_{l\be}(Z_{**})}\frac{(1+I_{**})^{1/8}}{(1+I_{**})^{1/8}}dZ_{**}dZ_*\notag\\
	\leq&\frac{C(1+t)^{-\frac{3}{4}}}{N^{1/8}}\sup_{0\leq t<\infty}\|(1+t)^{\frac{3}{4}}h_l(t)\|_{L^\infty_{v,I}(L^2_x\cap L^\infty_x)},
\end{align*}
for $l>p$, $j\leq p$, and
\begin{align*}
	\|I^{\al_1,\al_2}_{4jl}\|_{L^2_x\cap L^\infty_x}\leq&C\sup_{0\leq t<\infty}\|(1+t)^{-\frac{3}{4}}h_l(t)\|_{L^\infty_{v,I}(L^2_x\cap L^\infty_x)} \int_{\CZ_j}k_{ij}^{\al_1}(Z,Z_*)\frac{w_{i\be}(Z)}{w_{j\be}(Z_*)}\frac{(1+I_*)^{1/8}}{(1+I_*)^{1/8}}\notag\\
	&\qquad\qquad\qquad\qquad\times\int_{\CZ_l}k_{jl}^{\al_2}(Z_*,Z_{**})\frac{w_{j\be}(Z_*)}{w_{l\be}(Z_{**})}\frac{(1+I_{**})^{1/8}}{(1+I_{**})^{1/8}}dZ_{**}dZ_*\notag\\
	\leq&\frac{C(1+t)^{-\frac{3}{4}}}{N^{1/8}}\sup_{0\leq t<\infty}\|(1+t)^{\frac{3}{4}}h_l(t)\|_{L^\infty_{v,I}(L^2_x\cap L^\infty_x)},
\end{align*}
for $l,j>p$. It follows from the aforementioned three estimates that
\begin{align}\label{case3}
	\|I^{\al_1,\al_2}_{4jl}\|_{L^2_x\cap L^\infty_x}\leq\frac{C(1+t)^{-\frac{3}{4}}}{N^{1/8}}\sup_{0\leq t<\infty}\|(1+t)^{\frac{3}{4}}h_l(t)\|_{L^\infty_{v,I}(L^2_x\cap L^\infty_x)}.
\end{align}

\medskip
\noindent{\it Case 4. } When $s-s_1\leq\la$, where $0<\la<\frac{1}{2}$ will be determined later.
We can control $I^{\al_1,\al_2}_{4jl}$ by
\begin{align}\label{case4}
	\|I^{\al_1,\al_2}_{4jl}\|_{L^2_x\cap L^\infty_x}\leq&\Big|\int_0^t e^{-\nu_i(Z)(t-s)}\int_{\CZ_j}k_{ij}^{\al_1}(Z,Z_*)\frac{w_{i\be}(Z)}{w_{j\be}(Z_*)}\notag\\
	&\quad\times\int_{s-\la}^s e^{-\nu_i(Z_*)(s-s_1)}\int_{\CZ_l}k_{jl}^{\al_2}(Z_*,Z_{**})\frac{w_{j\be}(Z_*)}{w_{l\be}(Z_{**})}\|h_l(s_1,Z_{**})\|_{L^2_x\cap L^\infty_x}dZ_{**}  ds_1 dZ_*ds\Big|\notag\\
	\leq&C\la(1+t)^{-\frac{3}{4}}\sup_{0\leq t<\infty}\|(1+t)^{\frac{3}{4}}h_l(t)\|_{L^\infty_{v,I}(L^2_x\cap L^\infty_x)}.
\end{align}

\medskip
\noindent{\it Case 5. } We can write the remainder case as
\begin{align}\label{vbound}
	\begin{cases}
		|v|\leq N,\quad \it{if}\ i\leq p\\
		|v|\leq N\ \it{and}\ I\leq N,\quad \it{if}\ i>p,
	\end{cases}
\end{align}
\begin{align}\label{v*bound}
	\begin{cases}
		|v-v_*|\leq N,\quad \it{if}\ \al_1=1,\\
		|m_i v-m_j v_*|\leq N,\quad \it{if}\ \al_1=2,\\
		|v_*-v_{**}|\leq N,\quad \it{if}\ \al_2=1,\\
		|m_jv_*-m_lv_{**}|\leq N,\quad \it{if}\ \al_2=2,
	\end{cases}\quad \text{and} \quad
	\begin{cases}
		I_*\leq N\ \it{and}\ I_{**}\leq N,\quad \it{if}\ j,l>p\\
		I_*\leq N,\quad \it{if}\ j>p,\ l\leq p\\
		I_{**}\leq N,\quad \it{if}\ j\leq p,\ l>p,
	\end{cases}
\end{align}
and $s-s_1>\la$. Notice that
\begin{align*}
|v_*|&\leq C|v|+C\min\{|v-v_*|,|m_iv-m_jv_*|\} \qquad\text{and} \\
|v_{**}|\leq C|v|+C&\min\{|v-v_*|,|m_iv-m_jv_*|\}+C\min\{|v_*-v_{**}|,|m_jv_*-m_lv_{**}|\}.
\end{align*}
For all possible $i,j,l,\al_1,\al_2$, then \eqref{vbound} and \eqref{v*bound} imply $Z,Z_*,Z_{**}$ are in a bounded region with
\begin{align*}
	|v|+|v_*|+|v_{**}|+I\chi_{\{i>p\}}+I_*\chi_{\{j>p\}}+I_{**}\chi_{\{l>p\}}\leq CN.
\end{align*}
Define the region $$\CZ^B_i=\begin{cases}
		\{v\big|\ |v|\leq CN\},\quad \it{if}\ i\leq p,\\
		\{(v,I)\big|\ |v|+I\leq CN\},\quad \it{if}\ i>p.
\end{cases}$$
For any $1\leq i,j\leq n$ and $\al=1,2$, we can choose a smooth function $k^{\al}_{Nij}=k^{\al}_{Nij}(Z,Z_*)$ with compact support such that
\begin{align}\label{approx}
	\sup_{Z\in \CZ^B_i}\int_{\CZ^B_{*j}}|k_{ij}^{\al}(Z,Z_*)\frac{w_{i\be}(Z)}{w_{j\be}(Z_*)}-k^{\al}_{Nij}(Z,Z_*)|dZ_* \leq \frac{C}{N}.
\end{align}
Using Fubini's theorem, a direct calculation shows that
\begin{align}\label{1Ial}
	\|I^{\al_1,\al_2}_{4jl}\|_{L^2_x\cap L^\infty_x}&\leq \int_0^t\int_0^s\int_{\CZ^B_{*j}\times \CZ^B_{**l}} e^{-\nu_0(t-s)}e^{-\nu_0(s-s_1)}|k_{ij}^{\al_1}(Z,Z_*)\frac{w_{i\be}(Z)}{w_{j\be}(Z_*)}-k^{\al_1}_{Nij}(Z,Z_*)| \notag\\
	&\qquad\quad\times k_{jl}^{\al_2}(Z_*,Z_{**})\frac{w_{j\be}(Z_*)}{w_{l\be}(Z_{**})}\|h(s_1)\|_{L^\infty_{v,I}(L^2_x\cap L^\infty_x)}dZ_{**}dZ_*ds_1ds\notag\\
	&\qquad+\int_0^t\int_0^s\int_{\CZ^B_{*j}\times \CZ^B_{**l}} e^{-\nu_0(t-s)}e^{-\nu_0(s-s_1)}k_{ij}^{\al_1}(Z,Z_*)\frac{w_{i\be}(Z)}{w_{j\be}(Z_*)}\notag\\
	&\qquad\quad\times |k_{ij}^{\al_2}(Z_*,Z_{**})\frac{w_{j\be}(Z_*)}{w_{l\be}(Z_{**})}-k^{\al_2}_{Njl}(Z_*,Z_{**})|\|h(s_1)\|_{L^\infty_{v,I}(L^2_x\cap L^\infty_x)}dv_{**}dZ_{**}dZ_*ds_1ds\notag\\
	&\qquad+\int_0^t\int_0^s\int_{\CZ^B_{*j}\times \CZ^B_{**l}} e^{-\nu_0(t-s)}e^{-\nu_0(s-s_1)}k^{\al_1}_{Nij}(Z,Z_*) \notag\\
	&\qquad\quad\times k^{\al_2}_{Njl}(Z_*,Z_{**})|h(s_1,x_1-v_*(s-s_1),Z_{**})|dZ_{**}dZ_*ds_1ds\notag\\
	&\leq \frac{C(1+t)^{-\frac{3}{4}}}{N}\sup_{0\leq t<\infty}\|(1+t)^{\frac{3}{4}}h_l(t)\|_{L^\infty_{v,I}(L^2_x\cap L^\infty_x)}+\|J\|_{L^2_x\cap L^\infty_x},
\end{align}
with
\begin{align*}
	J=\int_0^t\int_0^s\int_{\CZ^B_{*j}\times \CZ^B_{**l}} e^{-\nu_0(t-s)}e^{-\nu_0(s-s_1)}&k^{\al_1}_{Nij}(Z,Z_*) k^{\al_2}_{Njl}(Z_*,Z_{**})\notag\\
	&\times|h(s_1,x_1-v_*(s-s_1),Z_{**})|dZ_{**}dZ_*ds_1ds.
\end{align*}
We use change of variables $y=x_1-v_*(s-s_1)$ to obtain that
\begin{align*}
|J|\leq&	C_N \int_0^t\int_0^s\int_{\CZ^B_{*j}\times \CZ^B_{**l}} e^{-\nu_0(t-s)}e^{-\nu_0(s-s_1)}|h(s_1,y,Z_{**})|dZ_{**}dZ_*ds_1ds\notag\\
\leq&	C_{N,\la} \int_0^t\int_0^s e^{-\nu_0(t-s)}e^{-\nu_0(s-s_1)}\big(\int_{\R^3\times \CZ_{j}}|h(s_1,y,Z_{**})|^2dZ_{**}dy\big)^\frac{1}{2}ds_1ds,
\end{align*}
which, combined with \eqref{L2decayR3}, yields
\begin{align}\label{estJ}
	|J|\leq&C_{N,\la}(1+t)^{-\frac{3}{4}}(\|f_0\|_{L^2_{v,I}(L^1_x\cap L^2_x)})\leq C_{N,\la}(1+t)^{-\frac{3}{4}}\|h_0\|_{L^\infty_{v,I}(L^1_x\cap L^\infty_x)}.
	\end{align}
The $L^2_x$ norm is more direct to control. It holds from the definition of $J$ and \eqref{L2decayR3} that
\begin{align}\label{estJL2}
	\|J\|_{L^2_x}&\leq C_N\int_0^t\int_0^s\int_{\CZ^B_{*j}\times \CZ^B_{**l}} e^{-\nu_0(t-s)}e^{-\nu_0(s-s_1)}\|f(s_1,Z_{**})\|_{L^2_x}dZ_{**}dZ_*ds_1ds\notag\\
	&\leq C_{N,\la}(1+t)^{-\frac{3}{4}}\|h_0\|_{L^\infty_{v,I}(L^1_x\cap L^\infty_x)}.
\end{align}
	Moreover, it follows from \eqref{1Ial}, \eqref{estJ} and \eqref{estJL2} that
	\begin{align}\label{case5}
		\|I^{\al_1,\al_2}_{4jl}\|_{L^2_x\cap L^\infty_x}&\leq \frac{C(1+t)^{-\frac{3}{4}}}{N}\sup_{0\leq t<\infty}\|(1+t)^{\frac{3}{4}}h_l(t)\|_{L^\infty_{v,I}(L^2_x\cap L^\infty_x)}+C_{N,\la}(1+t)^{-\frac{3}{4}}\|h_0\|_{L^\infty_{v,I}(L^1_x\cap L^\infty_x)}.
	\end{align}	
	We collect \eqref{case1}, \eqref{case2}, \eqref{case3}, \eqref{case4} and \eqref{case5} to get that
		\begin{align*}
		\|I^{\al_1,\al_2}_{4jl}\|_{L^2_x\cap L^\infty_x}&\leq C(\frac{1}{N^{1/8}}+
		\la)(1+t)^{-\frac{3}{4}}\sup_{0\leq t<\infty}\|(1+t)^{\frac{3}{4}}h_l(t)\|_{L^\infty_{v,I}(L^2_x\cap L^\infty_x)}+C_{N,\la}(1+t)^{-\frac{3}{4}}\|h_0\|_{L^\infty_{v,I}(L^1_x\cap L^\infty_x)},
	\end{align*}
	which further implies that
	\begin{multline}\label{I4}
		\|I_{4jl}\|_{L^2_x\cap L^\infty_x}\leq C(\frac{1}{N^{1/8}}+
		\la)(1+t)^{-\frac{3}{4}}\sup_{0\leq t<\infty}\|(1+t)^{\frac{3}{4}}h_l(t)\|_{L^\infty_{v,I}(L^2_x\cap L^\infty_x)}\\+C_{N,\la}(1+t)^{-\frac{3}{4}}\|h_0\|_{L^\infty_{v,I}(L^1_x\cap L^\infty_x)}.
	\end{multline}
The estimates on $I_{5jl}$, $I_{6jl}$ and $I_{7jl}$ are similar. We only list the cases as previously presented and show the sketch of the proof. To bound $I_{5jl}$, we recall its definition in \eqref{Ijl} and control the $L^2_x\cap L^\infty_x$ norm of the integral
\begin{align*}
	I^{\al_1}_{5jl}:=&\Big|\int_0^t e^{-\nu_i(Z)(t-s)}\int_{\CZ_j}k_{ij}^{\al_1}(Z,Z_*)\frac{w_{i\be}(Z)}{w_{j\be}(Z_*)}\notag\\
	&\qquad\qquad\times\int_0^s e^{-\nu_i(Z_*)(s-s_1)}\int_{\CZ_j}k_{jl}^{3}(Z_*,Z_{**})\frac{w_{j\be}(Z_*)}{w_{j\be}(Z_{**})}h_j(s_1,x_1-v_{**}(s-s_1),Z_{**})dZ_{**}  ds_1 dZ_*ds\Big|,
\end{align*}
with $\al_1=1,2.$ If 
\begin{align*}
\begin{cases}
	|v|>N,\quad \it{if}\ i\leq p\\
	|v|>N\ \it{or}\ I>N,\quad \it{if}\ i>p,
\end{cases}
\end{align*}
similar argument as in \eqref{case1} shows that
\begin{align}\label{case15}
\|I^{\al_1}_{5jl}\|_{L^2_x\cap L^\infty_x}\leq&\frac{C(1+t)^{-\frac{3}{4}}}{N^{1/8}}\sup_{0\leq t<\infty}\|(1+t)^{\frac{3}{4}}h_j(t)\|_{L^\infty_{v,I}(L^2_x\cap L^\infty_x)}.
\end{align}
If 
\begin{align*}
	\begin{cases}
		|v-v_*|>N,\quad \it{if}\ \al_1=1,\\
		|m_i v-m_j v_*|>N,\quad \it{if}\ \al_1=2,
	\end{cases}\quad
\text{or}\qquad
		|v_*-v_{**}|>N,
\end{align*}
it holds from the calculations as in \eqref{case2} that
\begin{align}\label{case25}
	\|I^{\al_1}_{5jl}\|_{L^2_x\cap L^\infty_x}\leq&\frac{C(1+t)^{-\frac{3}{4}}}{N^{1/8}}\sup_{0\leq t<\infty}\|(1+t)^{\frac{3}{4}}h_j(t)\|_{L^\infty_{v,I}(L^2_x\cap L^\infty_x)}.
\end{align}
If 
$I_*>N$ or $I_{**}>N$ for $j>p$,
or
$s-s_1\leq\la$,
we bound $I^{\al_1}_{5jl}$ as in \eqref{case3} and \eqref{case4} that
\begin{align}\label{case35}
	\|I^{\al_1}_{5jl}\|_{L^2_x\cap L^\infty_x}\leq&C(1+t)^{-\frac{3}{4}}(N^{-1/8}+\la)\sup_{0\leq t<\infty}\|(1+t)^{\frac{3}{4}}h_j(t)\|_{L^\infty_{v,I}(L^2_x\cap L^\infty_x)}.
\end{align}
The rest of the cases can be summarized as 
\begin{align*}
	\begin{cases}
		|v|\leq N,\quad \it{if}\ i\leq p\\
		|v|\leq N\ \it{and}\ I\leq N,\quad \it{if}\ i>p,
	\end{cases}\quad\text{and}\quad
	\begin{cases}
		|v-v_*|\leq N,\quad \it{if}\ \al_1=1,\\
		|m_i v-m_j v_*|\leq N,\quad \it{if}\ \al_1=2,\\
		|v_*-v_{**}|\leq N,
	\end{cases}
\end{align*}
$I_*\leq N$ and $I_{**}\leq N$ if $j>p,$
and $s-s_1>\la$. Noting that we have the bound
\begin{align*}
	|v|+|v_*|+|v_{**}|+I\chi_{\{i>p\}}+I_*\chi_{\{j>p\}}+I_{**}\chi_{\{j>p\}}\leq CN,
\end{align*}
we can choose an smooth approximation function with compact support $k^\al_{Nij}$ such that \eqref{approx} holds. Moreover, to approximate $k^3_{ij}$ we choose an smooth function with compact support $k^3_{Njl}$ such that
\begin{align}\label{approxk3}
	\sup_{Z\in \CZ^B_j}\int_{\CZ^B_{*j}}|k_{jl}^{3}(Z,Z_*)\frac{w_{j\be}(Z)}{w_{j\be}(Z_*)}-k^{3}_{Njl}(Z,Z_*)|dv_*dI_* \leq \frac{C}{N}.
\end{align}
Then similar arguments as in \eqref{1Ial}, \eqref{estJ} and \eqref{case5} give that
\begin{align}\label{case45}
	\|I^{\al_1}_{5jl}\|_{L^2_x\cap L^\infty_x}&\leq \frac{C(1+t)^{-\frac{3}{4}}}{N}\sup_{0\leq t<\infty}\|(1+t)^{\frac{3}{4}}h_j(t)\|_{L^\infty_{v,I}(L^2_x\cap L^\infty_x)}+C_{N,\la}(1+t)^{-\frac{3}{4}}\|h_0\|_{L^\infty_{v,I}(L^1_x\cap L^\infty_x)}.
\end{align}	
It follows from \eqref{case15}, \eqref{case25}, \eqref{case35} and \eqref{case45} that
\begin{align*}
	\|I^{\al_1}_{5jl}\|_{L^2_x\cap L^\infty_x}&\leq C(\frac{1}{N^{1/8}}+
	\la)(1+t)^{-\frac{3}{4}}\sup_{0\leq t<\infty}\|(1+t)^{\frac{3}{4}}h_j(t)\|_{L^\infty_{v,I}(L^2_x\cap L^\infty_x)}+C_{N,\la}(1+t)^{-\frac{3}{4}}\|h_0\|_{L^\infty_{v,I}(L^1_x\cap L^\infty_x)},
\end{align*}
which yields
\begin{align}\label{I5}
	\|I_{5jl}\|_{L^2_x\cap L^\infty_x}&\leq C(\frac{1}{N^{1/8}}+
	\la)(1+t)^{-\frac{3}{4}}\sup_{0\leq t<\infty}\|(1+t)^{\frac{3}{4}}h_j(t)\|_{L^\infty_{v,I}(L^2_x\cap L^\infty_x)}+C_{N,\la}(1+t)^{-\frac{3}{4}}\|h_0\|_{L^\infty_{v,I}(L^1_x\cap L^\infty_x)}.
\end{align}
For $I_{6jl}$, defined in \eqref{Ijl}, we estimate
\begin{align*}
	I^{\al_1}_{6jl}&=\Big|\int_0^t e^{-\nu_i(Z)(t-s)}\int_{\CZ_i}k^3_{ij}(Z,Z_*)\frac{w_{i\be}(Z)}{w_{i\be}(Z_*)}\notag\\
	&\qquad\qquad\times\int_0^s e^{-\nu_i(Z_*)(s-s_1)}\int_{\CZ_l}k_{il}^\al(Z_*,Z_{**})\frac{w_{i\be}(Z_*)}{w_{l\be}(Z_{**})}h_l(s_1,x_1-v_{**}(s-s_1),Z_{**})dZ_{**}  ds_1 dZ_*ds\Big|,
\end{align*}
with $\al_1=1,2.$ 
If \begin{align*}
	\begin{cases}
		|v|>N,\quad \it{if}\ i\leq p\\
		|v|>N\ \it{or}\ I>N,\quad \it{if}\ i>p,
	\end{cases}
\end{align*}
or
$|v-v_*|> N,$ or
\begin{align*}
	\begin{cases}
		|v_*-v_{**}|> N,\quad \it{if}\ \al_1=1,\\
		|m_iv_*-m_lv_{**}|> N,\quad \it{if}\ \al_1=2,
	\end{cases}\quad \text{or} \qquad
	\begin{cases}
		I_*>N\ \it{or}\ I_{**}>N,\quad \it{if}\ i,l>p\\
		I_*>N,\quad \it{if}\ i>p,\ l\leq p\\
		I_{**}>N,\quad \it{if}\ i\leq p,\ l>p.
	\end{cases}
\end{align*}
or $s-s_1\leq\la$, the arguments in \eqref{case1}, \eqref{case2}, \eqref{case3} and \eqref{case4} show
\begin{align}\label{case16}
	\|I^{\al_1}_{6jl}\|_{L^2_x\cap L^\infty_x}\leq&C(1+t)^{-\frac{3}{4}}(N^{-1/8}+\la)\sup_{0\leq t<\infty}\|(1+t)^{\frac{3}{4}}h_l(t)\|_{L^\infty_{v,I}(L^2_x\cap L^\infty_x)}.
\end{align}
The last case is when
\begin{align*}
	\begin{cases}
		|v|\leq N,\quad \it{if}\ i\leq p\\
		|v|\leq N\ \it{and}\ I\leq N,\quad \it{if}\ i>p,
	\end{cases}
\end{align*}
\begin{align*}
	\begin{cases}
		|v-v_*|\leq N,\\
		|v_*-v_{**}|\leq N,\quad \it{if}\ \al_1=1,\\
		|m_iv_*-m_lv_{**}|\leq N,\quad \it{if}\ \al_1=2,
	\end{cases}\quad\text{and}\quad
	\begin{cases}
		I_*\leq N\ \it{and}\ I_{**}\leq   N,\quad \it{if}\ i,l>p\\
		I_*\leq N,\quad \it{if}\ i>p,\ l\leq p\\
		I_{**}\leq N,\quad \it{if}\ i\leq p,\ l>p.
	\end{cases}
\end{align*}
and $s-s_1>\la$. Then, by choosing smooth approximation functions $k^{1,2,3}_{Nij}$ with compact support such that
\begin{align*}
	\sup_{Z\in \CZ^B_i}\int_{\CZ^B_{*j}}|k_{ij}^{\al}(Z,Z_*)\frac{w_{i\be}(Z)}{w_{l\be}(Z_*)}-k^{\al}_{Nij}(Z,Z_*)|dZ_* \leq \frac{C}{N},
\end{align*}
for $\al=1,2$, and 
\begin{align}\label{approxk3ij}
	\sup_{Z\in \CZ^B_i}\int_{\CZ^B_{*i}}|k_{ij}^{3}(Z,Z_*)\frac{w_{i\be}(Z)}{w_{i\be}(Z_*)}-k^{3}_{Nij}(Z,Z_*)|dZ_* \leq \frac{C}{N},
\end{align}
together with change of variables $y=x_1-v_*(s-s_1)$ for the $L^\infty_x$ bound, we obtain that
\begin{align}\label{case26}
	\|I^{\al_1}_{6jl}\|_{L^2_x\cap L^\infty_x}&\leq \frac{C(1+t)^{-\frac{3}{4}}}{N}\sup_{0\leq t<\infty}\|(1+t)^{\frac{3}{4}}h_l(t)\|_{L^\infty_{v,I}(L^2_x\cap L^\infty_x)}+C_{N,\la}(1+t)^{-\frac{3}{4}}\|h_0\|_{L^\infty_{v,I}(L^1_x\cap L^\infty_x)}.
\end{align}
Collecting \eqref{case16} and \eqref{case26}, and taking summation on $\al_1=1,2$, it holds that
\begin{multline}\label{I6}
	\|I_{6jl}\|_{L^2_x\cap L^\infty_x}\leq C(\frac{1}{N^{1/8}}+
	\la)(1+t)^{-\frac{3}{4}}\sup_{0\leq t<\infty}\|(1+t)^{\frac{3}{4}}h_l(t)\|_{L^\infty_{v,I}(L^2_x\cap L^\infty_x)}\\+C_{N,\la}(1+t)^{-\frac{3}{4}}\|h_0\|_{L^\infty_{v,I}(L^1_x\cap L^\infty_x)}.
\end{multline}
We turn to $I_{7jl}$, defined in \eqref{Ijl}. If \begin{align*}
	\begin{cases}
		|v|>N,\quad \it{if}\ i\leq p\\
		|v|>N\ \it{or}\ I>N,\quad \it{if}\ i>p,
	\end{cases}
\end{align*}
or
$|v-v_*|> N,$ or $|v_*-v_{**}|> N$,
or
$$
		I_*>N\ \it{or}\ I_{**}>N,\quad \it{if}\ i>p,
$$
or $s-s_1\leq\la$, we bound $I_{7jl}$ using the approach in \eqref{case1}, \eqref{case2}, \eqref{case3} and \eqref{case4} that
\begin{align}\label{case17}
	\|I_{7jl}\|_{L^2_x\cap L^\infty_x}\leq&C(1+t)^{-\frac{3}{4}}(N^{-1/8}+\la)\sup_{0\leq t<\infty}\|(1+t)^{\frac{3}{4}}h_i(t)\|_{L^\infty_{v,I}(L^2_x\cap L^\infty_x)}.
\end{align}
If \begin{align*}
	\begin{cases}
		|v|\leq N,\quad \it{if}\ i\leq p\\
		|v|\leq N\ \it{and}\ I\leq N,\quad \it{if}\ i>p,
	\end{cases}
\end{align*}
$\max\{|v-v_*|,|v_*-v_{**}|\}\leq N$, $\min\{I_*,I_{**}\}\leq N$ if $i>p$, and $s-s_1>\la$, we use the same approximation function $k^{3}_{Nij}$ employed in 
\eqref{approxk3ij} and perform the calculations as in \eqref{1Ial} and \eqref{estJ} to deduce that
\begin{multline}\label{case27}
	\|I_{7jl}\|_{L^2_x\cap L^\infty_x}\leq \frac{C(1+t)^{-\frac{3}{4}}}{N}\sup_{0\leq t<\infty}\|(1+t)^{\frac{3}{4}}h_i(t)\|_{L^\infty_{v,I}(L^2_x\cap L^\infty_x)}
	+C_{N,\la}(1+t)^{-\frac{3}{4}}\|h_0\|_{L^\infty_{v,I}(L^1_x\cap L^\infty_x)}.
\end{multline}	
The combination of \eqref{case17} and \eqref{case27} shows that
\begin{multline}\label{I7}
	\|I_{7jl}\|_{L^2_x\cap L^\infty_x}\leq C(\frac{1}{N^{1/8}}+
	\la)(1+t)^{-\frac{3}{4}}\sup_{0\leq t<\infty}\|(1+t)^{\frac{3}{4}}h_i(t)\|_{L^\infty_{v,I}(L^2_x\cap L^\infty_x)}\\
	+C_{N,\la}(1+t)^{-\frac{3}{4}}\|h_0\|_{L^\infty_{v,I}(L^1_x\cap L^\infty_x)}.
\end{multline}
We collect \eqref{L2hi}, \eqref{I1}, \eqref{I23}, \eqref{I4}, \eqref{I5}, \eqref{I6} and \eqref{I7} to obtain that
\begin{align*}
	\|h_i(t,Z_i)\|_{L^2_x\cap L^\infty_x}&\leq C(\frac{1}{N^{1/8}}+
	\la)(1+t)^{-\frac{3}{4}}\sup_{0\leq t<\infty}\|(1+t)^{\frac{3}{4}}h(t)\|_{L^\infty_{v,I}(L^2_x\cap L^\infty_x)}+C_{N,\la}(1+t)^{-\frac{3}{4}}\|h_0\|_{L^\infty_{v,I}(L^1_x\cap L^\infty_x)},
\end{align*}
which, by the summation on $1\leq i\leq n$ and taking $L^\infty_{v,I}$ norm, yields
\begin{align*}
	\sup_{0\leq t<\infty}(1+t)^{\frac{3}{4}}\|h(t)\|_{L^\infty_{v,I}(L^2_x\cap L^\infty_x)} &\leq C(\frac{1}{N^{1/8}}+\la)\sup_{0\leq t<\infty}\|(1+t)^{\frac{3}{4}}h(t)\|_{L^\infty_{v,I}(L^2_x\cap L^\infty_x)}+C_{N,\la}\|h_0\|_{L^\infty_{v,I}(L^1_x\cap L^\infty_x)}.
\end{align*}
Consequently, we obtain \eqref{boundh} by choosing $\frac{1}{N^{1/8}}+\la$ to be sufficiently small.

When $x\in \T^3$, the proof of \eqref{boundhT} follows at similar argument. The exponential decay is obtained by replacing the time-dependent weight $(1+t)^{\pm 3/4}$ by $e^{\pm \la t}$ accordingly and using \eqref{L2decaytorus} in Proposition \ref{leL2}.
\end{proof}

\medskip
\noindent {\bf Acknowledgment:}\,
Zongguang Li would like to thank the Research Centre for Nonlinear Analysis at The Hong Kong Polytechnic University for supporting his postdoc study.

\medskip

\noindent{\bf Conflict of Interest:} The authors declare that they have no conflict of interest.

\end{document}